\documentclass[reqno, twoside, 11pt]{article}
\usepackage[latin1]{inputenc}
\usepackage[T1]{fontenc}
\usepackage[english]{babel}
\usepackage{color}
\usepackage{tocloft}
\usepackage{verbatim}

\usepackage[
  hmarginratio={1:1},     
  vmarginratio={1:1},     
  textwidth=17cm,        
  textheight=24cm,
  heightrounded          
]{geometry}

\usepackage[noindentafter]{titlesec}
\titlespacing{\paragraph}{%
  0pt}{
0.25\baselineskip}{
1em}

\titlespacing{\section}{%
0pt}{
0.2cm}{
0em}

\titlespacing{\subsection}{%
0pt}{
0.2cm}{
0em}

\titlespacing{\subsubsection}{%
0pt}{
0.1cm}{
0em}

\usepackage{amsmath,amsthm,amssymb,accents}
\usepackage{mathrsfs, mathtools,xpatch}
\usepackage{dsfont}

\DeclareMathAlphabet{\mathbbe}{U}{bbold}{m}{n}

\mathtoolsset{showonlyrefs}

\allowdisplaybreaks

\usepackage[all]{xy} 
\input xy 
\xyoption{all}
\xyoption{2cell} 
\xyoption{v2}
\SelectTips{cm}{11} 

\usepackage{tikz}
\usetikzlibrary{snakes}
\usetikzlibrary{arrows}
\usepackage{tikz-cd}

\tikzset{string/.style={decorate, decoration={snake, segment length=3pt, amplitude=1pt}}}

\usepackage[hidelinks]{hyperref}

\usepackage[numbers,sort&compress]{natbib}
\makeatletter
\renewcommand{\@biblabel}[1]{[#1]\hfill}
\makeatother

\usepackage[mathscr]{euscript}
\DeclareMathAlphabet{\pazocal}{OMS}{zplm}{m}{n}

\usepackage[shortlabels]{enumitem}

\newtheoremstyle{thm}                                                           
{0.15cm}                                         
{0.15cm}                                         
{\itshape}      
{}                                      
{\bfseries}                             
{}                                      
{0.2cm}                                         
{\thmname{#1}~\thmnumber{#2}\thmnote{ (#3)}}%

\makeatletter
\xpatchcmd{\proof}{\topsep6\p@\@plus6\p@\relax}{}{}{}
\makeatother

\newtheoremstyle{rmk}                                                           
{0.15cm}                                         
{0.15cm}                                         
{}      
{}                                      
{\bfseries}                             
{}                                      
{0.2cm}                                         
{}                                      

\theoremstyle{thm}
\newtheorem{theorem}[equation]{Theorem}

\newtheorem{corollary}[equation]{Corollary}

\newtheorem{lemma}[equation]{Lemma}
\newtheorem{proposition}[equation]{Proposition}
\newtheorem{definition}[equation]{Definition}

\theoremstyle{rmk}
\newtheorem{example}[equation]{Example}
\newtheorem{remark}[equation]{Remark}

\numberwithin{equation}{section}

\makeatletter
\newlength{\@thlabel@width}%
\newcommand{\thmenumhspace}{\settowidth{\@thlabel@width}{(1)}\sbox{\@labels}{\unhbox\@labels\hspace{\dimexpr-\leftmargin+\labelsep+\@thlabel@width-\itemindent}}}
\makeatother

\newcommand{\pr}{\mathrm{pr}}

\newcommand{\diag}{\operatorname{diag}}

\newcommand{\MCG}{\mathrm{MCG}}
\newcommand{\pt}{\mathrm{pt}}

\newcommand{\rmB}{{\mathrm{B}}}
\newcommand{\rmU}{{\mathrm{U}}}

\newcommand{\rmh}{\mathrm{h}}

\newcommand{\rmD}{\mathrm{D}}

\newcommand{\ev}{\mathrm{ev}}

\newcommand{\Ev}{\mathrm{Ev}}

\newcommand{\dd}{\mathrm{d}}

\newcommand{\Plot}{\mathrm{Plot}}
\newcommand{\rmS}{\mathrm{S}}

\newcommand{\Hom}{\mathrm{Hom}}

\newcommand{\Diff}{\mathrm{Diff}}

\newcommand{\opp}{\mathrm{op}}

\newcommand{\Coh}{{\mathrm{Coh}}}

\newcommand{\holim}{\mathrm{holim}}
\newcommand{\hocolim}{{\mathrm{hocolim}}}

\newcommand{\hoRan}{{\mathrm{hoRan}}}

\newcommand{\colim}{\operatorname{colim}}
\newcommand{\triv}{{\mathrm{triv}}}
\newcommand{\Ho}{\mathrm{Ho}}

\newcommand{\rmL}{{\mathrm{L}}}

\newcommand{\XL}{{\mathrm{XL}}}

\renewcommand{\lim}{\operatorname{lim}}

\newcommand{\scI}{\mathscr{I}}
\newcommand{\scH}{\mathscr{H}}

\newcommand{\scT}{\mathscr{T}}

\newcommand{\scJ}{\mathscr{J}}

\newcommand{\scC}{\mathscr{C}}
\newcommand{\scD}{\mathscr{D}}

\newcommand{\scZ}{\mathscr{Z}}
\newcommand{\scY}{{\mathscr{Y}}}
\newcommand{\scF}{\mathscr{F}}

\newcommand{\scM}{{\mathscr{M}}}
\newcommand{\scN}{{\mathscr{N}}}

\newcommand{\scV}{{\mathscr{V}}}

\newcommand{\CG}{\mathcal{G}}

\newcommand{\CU}{{\mathcal{U}}}

\newcommand{\CA}{\mathcal{A}}

\newcommand{\sfL}{{\mathsf{L}}}

\newcommand{\sfT}{\mathsf{T}}
\newcommand{\sfZ}{\mathsf{Z}}

\newcommand{\sfD}{\mathsf{D}}
\newcommand{\sfN}{\mathsf{N}}

\newcommand{\sfX}{{\mathsf{X}}}

\newcommand{\sfc}{\mathsf{c}}
\newcommand{\sfS}{\mathsf{S}}

\newcommand{\sfC}{\mathsf{C}}

\newcommand{\ori}{{\mathsf{or}}}

\newcommand{\sfB}{\mathsf{B}}
\newcommand{\Mat}{\mathsf{Mat}}
\newcommand{\GL}{{\mathsf{GL}}}

\newcommand{\sfBord}{\mathsf{Bord}}

\newcommand{\bbDelta}{{\mathbbe{\Delta}}}

\newcommand{\CN}{\mathbb{C}}
\newcommand{\NN}{\mathbb{N}}
\newcommand{\RN}{\mathbb{R}}

\newcommand{\One}{\mathds{1}}
\newcommand{\bbS}{\mathbb{S}}

\newcommand{\fru}{{\mathfrak{u}}}

\newcommand{\ul}[1]{\underline{#1}}

\newcommand{\cC}{{\check{C}}}

\newcommand{\Grb}{{\mathscr{G}\mathrm{rb}}}

\newcommand{\Set}{{\mathscr{S}\mathrm{et}}}
\newcommand{\sSet}{{\mathscr{S}\mathrm{et}_{\hspace{-.03cm}\bbDelta}}}
\newcommand{\Fun}{{\mathrm{Fun}}}

\newcommand{\VBun}{{\mathscr{V}\hspace{-0.02cm}\mathscr{B}\hspace{-0.02cm}\mathrm{un}}}
\newcommand{\Cat}{{\mathscr{C}\mathrm{at}}}

\newcommand{\Mfd}{{\mathscr{M}\mathrm{fd}}}
\newcommand{\Dfg}{{\mathscr{D}\mathrm{fg}}}

\newcommand{\CSS}{\mathscr{CSS}}

\newcommand{\Cart}{{\mathscr{C}\mathrm{art}}}

\newcommand{\Sh}{{\mathscr{S}\mathrm{h}}}
\newcommand{\sfSh}{{\mathsf{Sh}}}
\newcommand{\PSh}{{\mathscr{PS}\mathrm{h}}}
\newcommand{\sfPSh}{{\mathsf{PSh}}}

\newcommand{\Desc}{{\mathscr{D}\mathrm{esc}}}

\newcommand{\Bord}{\mathscr{B}\mathrm{ord}}

\newcommand{\wH}{{\widehat{\scH}}}
\newcommand{\wQ}{{\widehat{Q}}}
\newcommand{\wC}{{\widehat{\scC}}}
\newcommand{\wsfC}{{\widehat{\mathsf{C}}_\infty}}
\newcommand{\wCart}{{\widehat{\Cart}}}

\newcommand{\wtau}{{\widehat{\tau}}}
\newcommand{\wY}{{\widehat{\scY}}}
\newcommand{\Op}{{\mathscr{O}\mathrm{p}}}

\newcommand{\ctau}{{\check{\tau}}}

\newcommand{\weq}{\overset{\sim}{\longrightarrow}}

\newcommand{\arisom}{\overset{\cong}{\longrightarrow}}

\newcommand{\iu}{\mathrm{i}}
\newcommand{\tr}{\mathrm{tr}}

\newenvironment{myitemize}{\begin{itemize}[itemsep=-0.1cm, leftmargin=*, topsep=0cm]}{\end{itemize}}
\newenvironment{myenumerate}{\begin{enumerate}[itemsep=-0.15cm, leftmargin=*, topsep=0cm, label=(\arabic*)]}{\end{enumerate}}

\newcommand{\qen}{\hfill$\triangleleft$}

\newcommand{\qandq}{\quad \text{and} \quad}

\newcommand{\dslash}{{/\hspace{-0.1cm}/}}

\newlength{\Displayskip}
\begin{document}

\setlength{\abovedisplayskip}{\Displayskip}
\setlength{\belowdisplayskip}{\Displayskip}

\begin{flushright}
\small
\textsf{Hamburger Beiträge zur Mathematik Nr.\,822}\\
{\sf ZMP--HH/20-2}
\end{flushright}


\begin{center}
\LARGE{\textbf{Homotopy Sheaves on Generalised Spaces}}
\end{center}
\begin{center}
{\large Severin Bunk}

{\em Mathematical Institute, The University of Oxford\\
Radcliffe Observatory Quarter, Oxford, OX6 2GG, UK}\\
Email: {\tt  severin.bunk@maths.ox.ac.uk}

\end{center}

\begin{abstract}
\noindent
We study the homotopy right Kan extension of homotopy sheaves on a category to its free cocompletion, i.e.~to its category of presheaves. Any pretopology on the original category induces a canonical pretopology of generalised coverings on the free cocompletion. We show that with respect to these pretopologies the homotopy right Kan extension along the Yoneda embedding preserves homotopy sheaves valued in (sufficiently nice) simplicial model categories. Moreover, we show that this induces an equivalence between sheaves of spaces on the original category and colimit-preserving sheaves of spaces on its free cocompletion. We present three applications in geometry and topology: first, we prove that diffeological vector bundles descend along subductions of diffeological spaces. Second, we deduce that various flavours of bundle gerbes with connection satisfy $(\infty,2)$-categorical descent. Finally, we investigate smooth diffeomorphism actions in smooth bordism-type field theories on a manifold. We show how these smooth actions allow us to extract the values of a field theory on any object coherently from its values on generating objects of the bordism category.
\end{abstract}

\vspace{-0.2cm}

\paragraph*{Conflict of interests statement}
There are no conflicts of interest with this paper.

\tableofcontents

\section{Introduction and main results}
\label{sec:introduction}

Local-to-global properties are ubiquitous in topology, geometry, and quantum field theory.
The prototypical example of a local-to-global, or \emph{descent}, property is the gluing of local sections of a sheaf:
given a manifold $M$ with an open covering $\{U_i\}_{i \in I}$, global sections of a sheaf $F$ on $M$ are in bijection with families $\{f_i \in F(U_i)\}_{i \in I}$ such that $f_{i|U_{ij}} = f_{j|U_{ij}}$ for all $i,j \in I$, with $U_{ij} \coloneqq U_i \cap U_j$.

Geometric structures on manifolds are, in general, not described by sheaves of sets; instead, one needs to pass to sheaves of higher categories.
For example, complex vector bundles form a sheaf on the site of manifolds and open coverings which is valued in categories rather than sets.
In the works of Schreiber~\cite{Schreiber:DCCT}, $n$-gerbes are described as sections of certain sheaves of $n$-groupoids, for any $n \in \NN$.
Geometric structures and their equivalences are thus described most generally in terms of sheaves of higher categories on a category of test manifolds $\scC$ which is endowed with a (pre)topology $\tau$.

To motivate the main questions of this paper, let us temporarily focus on sheaves of $\infty$-groupoids.
The $\infty$-category of these sheaves has a presentation as a left Bousfield localisation of the projective model category $\scH_{\infty,0} = \Fun(\scC^\opp, \sSet)$ of simplicial presheaves on $\scC$%
\footnote{In this introduction, we do not address sizes of sets; in the main text we do so using a nested triple of Grothendieck universes.}.
We denote its localisation at the $\tau$-coverings by $\scH_{\infty,0}^{loc}$.
Our prime example in this paper is the case $\scC = \Cart$, the category of \emph{cartesian spaces}. It consists of the manifolds diffeomorphic to $\RN^n$, for any $n \in \NN_0$, and all smooth maps between them.
It is an important (well-known) observation that sheaves of $\infty$-groupoids on $\Cart$ allow us to describe geometric structures not just on objects of $\Cart$, but on \emph{all} manifolds (and more)~\cite{Schreiber:DCCT}:
given a fibrant object $F \in \scH_{\infty,0}^{loc}$ and a manifold $M$, one defines the space of (derived) sections of $F$ on $M$ as the mapping space $\ul{\scH}_{\infty,0}(Q\ul{M},F)$.
Here, $\ul{M}$ is the presheaf on $\Cart$ consisting of smooth maps to $M$, the functor $Q$ is a cofibrant replacement%
\footnote{Throughout the main body of the text, we use Dugger's explicit construction of a cofibrant replacement functor for projective model categories of simplicial presheaves~\cite{Dugger:Universal_HoThys}.}
in $\scH_{\infty,0}$, and $\ul{\scH}_{\infty,0}(-,-)$ denotes the $\sSet$-enriched hom functor in $\scH_{\infty,0}$.
If $F$ is the simplicial presheaf on cartesian spaces which describes $G$-bundles or $n$-gerbes, for instance, then the Kan complex $\ul{\scH}_{\infty,0}(Q\ul{M},F)$ is the $\infty$-groupoid of $G$-bundles or $n$-gerbes on $M$, respectively.
In fact, we could use any presheaf $X$ on $\Cart$ in place of $\ul{M}$ and thus study geometries on any such $X$.
(Later we will even allow $X$ to be a presheaf of $\infty$-groupoids itself.)

We write \smash{$\wC = \Fun(\scC^\opp,\Set)$} for the category of (small) presheaves of sets on $\scC$.
Motivated by the above discussion, we pose the following two questions for sheaves on a generic site $(\scC,\tau)$:
\begin{myenumerate}
\item If $F$ satisfies descent with respect to the pretopology $\tau$, does the assignment $X \mapsto \ul{\scH}_{\infty,0}(QX,F)$ satisfy descent with respect to a related pretopology $\wtau$ on \smash{$\wC$}?

\item If so, can we understand which $\wtau$-sheaves on \smash{$\wC$} arise as extensions $\ul{\scH}_{\infty,0}(Q(-),F)$ of some ordinary sheaf $F$ on $(\scC,\tau)$?
\end{myenumerate}
We answer both of these questions in this paper, working over general Grothendieck sites $(\scC,\tau)$:
any pretopology $\tau$ on $\scC$ induces a pretopology $\wtau$ of generalised coverings on \smash{$\wC$}.
We prove that with respect to such a pair of pretopologies, the answer to the first question is always `yes'.
One way of saying this is that the process of extending any geometric structure from objects of $\scC$ to objects of \smash{$\wC$} does not destroy the descent property; we can still think of the extended structure as having geometric flavour and perform local-to-global constructions.

In concrete terms, the induced Grothendieck pretopology $\wtau$ on \smash{$\wC$} consists of the \emph{$\tau$-local epimorphisms}, also called \emph{generalised coverings}~\cite{DHI:Hypercovers_and_sPShs}, which are defined as follows.
Let \smash{$\scY \colon \scC \to \wC$} denote the Yoneda embedding of $\scC$.
A morphism $\pi \colon Y \to X$ in $\wC$ is a $\tau$-local epimorphism if, for every $c \in \scC$ and every map $\scY_c \to X$, there exists a covering $\{c_i \to c\}_{i \in I}$ in the site $(\scC,\tau)$ such that each composition $\scY_{c_i} \to \scY_c \to X$ factors through $\pi$; this is an abstract way of saying that the morphism $\pi$ has local sections.

Let $\scM$ be a left proper cellular or combinatorial simplicial model category.
We consider the projective model structures on $\Fun(\scC^\opp, \scM)$ and $\Fun(\wC{}^\opp, \scM)$, and their left Bousfield localisations $\Fun(\scC^\opp, \scM)^{loc}$ and $\Fun(\wC{}^\opp, \scM)^{loc}$ at the $\tau$-coverings and $\wtau$-coverings, respectively.
That is, the fibrant objects in these localised model categories are $\scM$-valued homotopy sheaves on $(\scC, \tau)$ and $(\wC, \wtau)$, respectively.
We show the following result (see Theorems~\ref{st:Re_M --| S_M is Quillen adjunction},~\ref{st:Re_scM --| S_scM descends to local MoStrs} and Proposition~\ref{st:S_M is hoRan} below):

\begin{theorem}
\label{st:intro Re_M --| S_M is QAd for local MoStrs}
Let $\scM$ be a left proper simplicial model category which is cellular or combinatorial and $\scC$ a category.
Let $\tau$ be a Grothendieck coverage on $\scC$, and $\wtau$ the induced Grothendieck coverage of $\tau$-local epimorphisms on $\wC$.
There is a Quillen adjunction
\begin{equation}
\label{eq:local Re_M -| S_M adjunction}
\begin{tikzcd}
	Re_\scM : \Fun(\wC^\opp, \scM)^{loc} \ar[r, shift left=0.15cm, "\perp"' yshift=0.05cm]
	& \Fun(\scC^\opp, \scM)^{loc} : S_\scM\,, \ar[l, shift left=0.15cm]
\end{tikzcd}
\end{equation}
whose right adjoint restricts on fibrant objects to the homotopy right Kan extension along the Yoneda embedding $\scY \colon \scC \hookrightarrow \wC$,
\begin{equation}
	S_{\scM|\scM_{fib}} \cong \hoRan_\scY\,.
\end{equation}

\end{theorem}

A prime application is the case of homotopy sheaves of higher categories:
denoting by $\CSS_n$ the model category of $n$-fold complete spaces~\cite{Rezk:HoTheory_of_HoTheories, Barwick:infty-n-Cat_as_closed_MoCat, BSP:Unicity}, let $\scH_{\infty,n}^{loc}$ be the left Bousfield localisation of the projective model structure on $\Fun(\scC^\opp, \CSS_n)$ at the \v{C}ech nerves of coverings in $(\scC,\tau)$.
Similarly, we let \smash{$\wH_{\infty,n}^{loc}$} denote the left Bousfield localisation of \smash{$\Fun(\wC^\opp, \CSS_n)$} at the \v{C}ech nerves of the $\tau$-local epimorphisms.
As a direct corollary to Theorem~\ref{st:intro Re_M --| S_M is QAd for local MoStrs} we obtain:

\begin{corollary}
\label{st:intro-Thm S^Q_(oo,n)}
Let $n \in \NN_0$, and let $\scC$ be a small category.
\begin{myenumerate}
\item For any fibrant $\scF \in \scH_{\infty,n}$, the presheaf $\ul{\scH}_{\infty,n} \big( Q(-), \scF \big)$ is equivalent to the homotopy right Kan extension $\hoRan_{\scY} \scF$ of $\scF$ along the Yoneda embedding \smash{$\scY \colon \scC \to \wC$}.
In particular, $\ul{\scH}_{\infty,n} \big( Q(-), \scF \big)$ presents the $\infty$-categorical right Kan extension of presheaves of $(\infty,n)$-categories along $\scY$.

\item If $(\scC,\tau)$ is a Grothendieck site, there is a Quillen adjunction
\begin{equation}
\begin{tikzcd}
	Re_{\infty,n} : \wH_{\infty,n}^{loc} \ar[r, shift left=0.15cm, "\perp"' yshift=0.04cm] & \scH_{\infty,n}^{loc} : \ul{\scH}_{\infty,n} \big( Q(-), - \big) \,. \ar[l, shift left=0.15cm]
\end{tikzcd}
\end{equation}
\end{myenumerate}
\end{corollary}

\vspace{-0.2cm}
This answers our first question affirmatively since $\ul{\scH}_{\infty,n} (Q(-), -)$ preserves fibrant objects.

In Section~\ref{sec:Dfg VBuns} we deduce the following strictification and descent result (Theorem~\ref{st:VBun_Cat = VBun_Dfg} and Theorem~\ref{st:Dfg VBuns descend along subductions}, respectively), thereby filling significant gaps in the literature on diffeological spaces:

\begin{theorem}
\label{st:intro Dfg VBuns}
Let $\VBun_\Dfg \colon \Dfg^\opp \to \Cat$ be the pseudo-functor which assigns to a diffeological space its category of diffeological vector bundles.
The following statements hold true:
\begin{myenumerate}
\item There is a strict functor $\iota^* \VBun_{str}$ (which we construct explicitly) and an objectwise equivalence of pseudo-functors $\Dfg^\opp \to \Cat_\rmL$,
\begin{equation}
	\CA \colon \iota^* \VBun_{str} \weq \VBun_\Dfg\,.
\end{equation}

\item Diffeological vector bundles satisfy descent along subductions of diffeological spaces.
\end{myenumerate}
\end{theorem}

Here we use Rezk's classifying diagram functor to deduce a 1-categorical analogue of Corollary~\ref{st:intro-Thm S^Q_(oo,n)} for $n=1$.
This readily provides the desired descent result; the main remaining work in proving Theorem~\ref{st:intro Dfg VBuns} goes into the strictification of $\VBun_\Dfg$.

Further, we deduce from Corollary~\ref{st:intro-Thm S^Q_(oo,n)} that various flavours of the 2-category of bundle gerbes with connection as introduced by Waldorf~\cite{Waldorf--More_morphisms} extend to sheaves of $(2,2)$-categories on \smash{$\wCart$} (Theorem~\ref{st:Grbs descend}); via a 2-categorical nerve construction, they further provide examples of homotopy sheaves of $(\infty,2)$-categories.
Our motivation for working model categorically up to this point stems from future applications to field theories:
these are functors out of $(\infty,n)$-categories of cobordisms, which exists as $n$-fold or $n$-uple (complete) Segal spaces~\cite{Lurie:Classification_of_TQFTs, CS:Bord_n, SP:Invertible_TFTs}.
This also feeds into another application:
in smooth field theories on a background manifold, one detects the full smooth structure of diffeomorphism groups rather than only their connected components (i.e.~mapping class groups).
We show how one can nevertheless coherently and smoothly reconstruct the values of a field theory on all objects from its values on only the generating objects of the bordism category (Theorems~\ref{st:ext via choice of diffeos} and~\ref{st:diagonal coherence zig-zag}).

In Section~\ref{sec:char of sheaves on wC} we answer the second question above.
To do so, we pass to quasi-categorical language:
let $\sfC$ be a quasi-category with a Grothendieck (pre)topology $\tau$.
Let $\wsfC$ denote the quasi-category of (small) presheaves of spaces on $\sfC$.
We write $\sfSh(\sfC,\tau)$ for the quasi-category of $\tau$-sheaves on $\sfC$ and $\sfSh_!(\wsfC,\wtau)$ for the quasi-category $\wtau$-sheaves on \smash{$\wsfC$} whose underlying functor \smash{$\wsfC \to \sfS^\opp$} preserves colimits.

\begin{theorem}
\label{st:intro image of h_*}
Let $(\sfC,\tau)$ be as above, and let $\scY_*$ denote the $\infty$-categorical right Kan extension of presheaves along the Yoneda embedding $\scY \colon \sfC^\opp \to \wsfC^\opp$.
The adjunction $\scY^* \dashv \scY_*$ restricts to an equivalence of quasi-categories
\begin{equation}
\begin{tikzcd}
	\sfSh_!(\wsfC,\wtau) \ar[r, shift left=0.1cm, "\scY^*"] & \sfSh(\sfC,\tau)\,. \ar[l, shift left=0.1cm, "\scY_*"]
\end{tikzcd}
\end{equation}
\end{theorem}

\vspace{-0.2cm}
Finally, to make contact with the first part of this paper, we prove a similar result where $\sfC = \sfN\scC$ is an ordinary category; this relates quasi-categorical $\tau$-sheaves of spaces on $\sfN\scC$ to $\wtau$-sheaves of spaces on the 1-category of set-valued presheaves $\wC = \Fun(\scC^\opp, \Set_\rmL)$ (Theorem~\ref{st:image of Y^0_*}).

\subsection*{Conventions}
\label{sec:conventions}

\paragraph*{Sizes and universes.}
Throughout, we choose and fix a nested triple of Grothendieck universes $\rmS \in \rmL \in \XL$%
\footnote{Note that, due to existing naming conventions in the literature, size $\mathrm{M}$ was not available.}%
, and we assume that $\rmS$ contains the natural numbers.
We write $\Set_\rmS$ and $\sSet_\rmS$ for the categories of $\rmS$-small sets and $\rmS$-small simplicial sets, respectively, and analogously for the universes $\rmL$ and $\XL$.
As is common in the literature, we refer to the elements in $\Set_\rmS$ simply as \textit{small} sets, those of $\Set_\rmL$ as \textit{large} sets and those of $\Set_\XL$ as \textit{extra large} sets.
All indexing sets will be assumed to be $\rmS$-small.

Let $\scC$ be a small category.
Consider the category $\wC \coloneqq \Fun(\scC^\opp, \Set_\rmS)$ of $\Set_\rmS$-valued presheaves on $\scC$.
Observe that $\wC$ is no longer small, since $\Set_\rmS$ is not small.
However, since $\Set_\rmS$ is a large category (its objects and morphisms each form sets in $\rmL$), it follows that \smash{$\wC$} is a large category.

The Yoneda embedding \smash{$\wY \colon \wC \to \Fun(\wC^\opp, \Set_\rmL)$} is fully faithful; observing that an object of $\Set_\rmS$ is also an object of $\Set_\rmL$, the standard proof applies.
Further, the Yoneda Lemma holds true for both $\scY$ and $\wY$ (again by the usual method of proof).
Finally, observe for any $c \in \scC$ and for any $X \in \wC$, by the Yoneda Lemma, there are canonical isomorphisms
\begin{equation}
\label{st:Yoneda embeddings together}
	\wY_X (\scY_c) = \wC(\scY_c, X) \cong X(c) \in \Set_\rmS \subset \Set_\rmL\,.
\end{equation}

\paragraph*{Categories and higher categories.}

In this article we use two different models for higher categories.
Section~\ref{sec:sites and local epis} and Appendix~\ref{app:Proof of VBun_Dfg Thm} do not make use of higher categories.
In Sections~\ref{sec:higher sheaves} and~\ref{sec:applications} we use $n$-fold complete Segal spaces (introduced in~\cite{Barwick:infty-n-Cat_as_closed_MoCat}) as a model for $(\infty,n)$-categories.
In particular, the term `$\infty$-groupoid' shall always refer to a Kan complex, and `$(\infty,n)$-category' shall always mean a $n$-fold complete Segal space.
In Section~\ref{sec:char of sheaves on wC}, and only there, all higher categories are modelled as quasi-categories in the sense of~\cite{BV:Ho-invar_alg_strs, Joyal:QCats_and_Kan_complexes, Lurie:HTT}.
In order to make clear the distinction between ordinary categories (possibly endowed with model structures), $n$-fold complete Segal spaces, and quasi-categories, we denote ordinary categories by script letters $\scC, \scD, \ldots$, $n$-fold complete Segal spaces (or presheaves thereof) by ordinary capitals $F, G, \ldots$, and quasi-categories by sans-serif capitals $\sfC, \sfD, \ldots$.
In particular, the quasi-categories of small spaces is denoted $\sfS_\rmS$, and analogously for the universes $\rmL$ and $\XL$.
The term `small space' (used on its own, as opposed to `Segal space', for instance) shall always mean an object in the quasi-category $\sfS_\rmS$ of small spaces, and analogously for the terms `large space' and `extra large space' and the universes $\rmL$ and $\XL$, respectively.

\paragraph*{Enriched categories.}

If $\scV$ is a monoidal category and $\scC$ is a $\scV$-enriched category (or $\scV$-category), we denote the $\scV$-enriched hom-objects of $\scC$ by $\ul{\scC}^\scV(-,-)$.
If $\scV = \sSet_\rmL$ is the category of simplicial sets, we also write $\ul{\scC}(-,-) \coloneqq \ul{\scC}^{\sSet_\rmL}(-,-)$.
If $\scV$ is a symmetric monoidal model category and $\scC$ is a $\scV$-enriched model category (in the sense of~\cite[Sec.~4]{Hovey:MoCats}), we will equivalently say that $\scC$ is a model $\scV$-category.

\paragraph*{Tractability.}
We recall from~\cite[Def.~1.21]{Barwick:Enriched_B-Loc} that a large model category $\scM$ is \textit{tractable} if there exists a regular $\rmL$-small cardinal $\lambda$ such that the category underlying $\scM$ is locally $\lambda$-presentable, and there exist $\rmL$-small sets $I, J$ of morphisms with the following properties:
the source and target of the morphisms in $I$ and $J$ are $\lambda$-presentable and cofibrant, the fibrations (resp.~trivial fibrations) in $\scM$ are precisely those morphisms satisfying the right lifting property with respect to $J$ (resp.~$I$).
In particular, $\scM$ is combinatorial.
For further details and background on the formalism of enriched Bousfield localisation, we refer the reader to~\cite{Barwick:Enriched_B-Loc}.

\paragraph*{Diagrams.}
For $\scJ$ a large category and $\scC$ a $\rmL$-tractable or $\rmL$-combinatorial model category (cf.~\cite{Barwick:infty-n-Cat_as_closed_MoCat,Barwick:Enriched_B-Loc}), the projective and the injective model structures on $\Fun(\scJ, \scC)$ exist and are again $\rmL$-tractable (resp.~$\rmL$-combinatorial); we denote them by $\Fun(\scJ, \scC)_{proj}$ and $\Fun(\scJ, \scC)_{inj}$, respectively.
If $\scJ$ is a Reedy category, we denote the Reedy model structure on $\Fun(\scJ, \scC)$ by $\Fun(\scJ, \scC)_{Reedy}$.

\paragraph*{Left Bousfield localisation.}
We recall that if $\scM$ is a large simplicial model category with a chosen collection of morphisms $A$, then
\begin{myitemize}
\item $z \in \scM$ is \emph{$A$-local} if it is fibrant in $\scM$ and for every morphism $f \colon x \to y$ in $A$ the induced morphism
\begin{equation}
	\ul{\scM}(Qy, z) \xrightarrow{(Qf)^*} \ul{\scM}(Qx,z)
\end{equation}
is a weak homotopy equivalence in $\sSet_\rmL$, and

\item $f \in \scM(x,y)$ is an \emph{$A$-local weak equivalence} if for every $A$-local object $z \in \scM$ the morphism
\begin{equation}
	\ul{\scM}(Qy, z) \xrightarrow{(Qf)^*} \ul{\scM}(Qx,z)
\end{equation}
is a weak equivalence in $\sSet_\rmL$.
\end{myitemize}
In that case, the left Bousfield localisation $L_A \scM$, provided it exists, is a large simplicial model category with a simplicial left Quillen functor $\scM \to L_A \scM$ which is universal among simplicial left Quillen functors out of $\scM$ that send $A$-local weak equivalences to weak equivalences (cf.~for instance~\cite[Def.~4.42, Def.~4.45]{Barwick:Enriched_B-Loc}).

\subsection*{Acknowledgements}

The author would like to thank Damien Calaque, Tobias Dyckerhoff, Geoffroy Horel, Corina Keller, Christoph Schweigert, Pelle Steffens and Konrad Waldorf for valuable discussions.
The author would further like to thank the anonymous referee for various suggestions that improved the paper noticeably---in particular leading to the much generalised version of Section~\ref{sec:HoSheaves vald in spl MoCats} below---and increased its readability.
The author was partially supported by RTG~1670, \emph{Mathematics Inspired by String Theory and Quantum Field Theory} and by the Deutsche Forschungsgemeinschaft (DFG, German Research Foundation) under the project number 468806966 and under Germany's Excellence Strategy---EXC 2121 ``Quantum Universe''---390833306.

\section{Grothendieck sites and local epimorphisms}
\label{sec:sites and local epis}

We recall the definition of a Grothendieck pretopology and a site.
Throughout this article, we assume that $\scC$ is a small category.

\begin{definition}
We write $\wC \coloneqq \Fun(\scC^\opp, \Set_\rmS)$ for the large category of small presheaves on $\scC$.
\end{definition}

Let $\scY \colon \scC \to \wC$, $c \mapsto \scY_c$, denote the Yoneda embedding of $\scC$.

\begin{definition}[{\cite[Def.~2.1.9]{Johnstone:Sketches_of_an_elephant_I}}]
\label{def:coverage and site}
Let $\scC$ be a small category.
\begin{myenumerate}
\item A \emph{coverage}, or \emph{Grothendieck pretopology}, on $\scC$ is given by assigning to every object $c \in \scC$ a small set $\tau(c)$ of families of morphisms $\{f_i \colon c_i \to c\}_{i \in I}$ (with $I \in \Set_\rmS$) satisfying the following properties:
$\{1_c\} \in \tau(c)$ for each $c \in \scC$, and for every family $\{f_i \colon c_i \to c\}_{i \in I} \in \tau(c)$ and any morphism $g \colon c' \to c$ in $\scC$ there exists a family $\{f'_j \colon c'_j \to c'\}_{j \in J} \in \tau(c')$ such that for every $j \in J$ we find some $i \in I$ and a commutative diagram
\begin{equation}
\begin{tikzcd}
	c'_j \ar[r] \ar[d, "f'_j"'] & c_i \ar[d, "f_i"]
	\\
	c' \ar[r, "g"'] & c
\end{tikzcd}
\end{equation}

\item The families in $\{f_i \colon c_i \to c\}_{i \in I} \in\tau(c)$ are called \emph{covering families for $c$}.

\item A \emph{(Grothendieck) site} is a category $\scC$ equipped with a coverage $\tau$.
\end{myenumerate}
\end{definition}

Later we will use the following technical condition:

\begin{definition}
\label{def:closed site}
We call a site $(\scC,\tau)$ \emph{closed} if it satisfies the following condition:
let $\{c_i \to c\}_{i \in I}$ be any covering family in $(\scC, \tau)$.
Further, for each $i \in I$, let $\{c_{i,j} \to c_i\}_{j \in J_i}$ be a covering family in $(\scC,\tau)$.
Then, there exists a covering family $\{d_k \to c\}_{k \in K}$ such that every morphism $d_k \to c$ factors through one of the composites $c_{i,j} \to c_i \to c$.
\end{definition}

\begin{example}
\label{eg:Cart as site}
Let $\Cart$ be the category of cartesian spaces, i.e.~of sub-manifolds of $\RN^\infty$ that are diffeomorphic to some $\RN^n$, with smooth maps between these manifolds as morphisms.
This category is small and has finite products.
A coverage on $\Cart$ is defined by calling a family $\{\iota_i \colon c_i \to c\}_{i \in I}$ a covering family if it satisfies
\begin{myenumerate}
\item all $\iota_i$ are open embeddings (in particular $\dim(c_i) = \dim(c)$ for all $i \in I$),

\item the $\iota_i$ cover $c$, i.e.~$c = \bigcup_{i \in I} \iota_i(c_i)$, and

\item every finite intersection $\iota_{i_0}(c_{i_0}) \cap \cdots \cap \iota_{i_m}(c_{i_m})$, with $i_0, \ldots, i_m \in I$, is a cartesian space.
\end{myenumerate}
Coverings $\{\iota_i \colon c_i \to c\}_{i \in I}$ with these properties are called \emph{differentiably good open coverings} of $c$.
They induce a coverage $\tau_{dgop}$ on $\Cart$, which turns $(\Cart,\tau_{dgop})$ into a site.
This site is even closed, because every open cover of $c \in \Cart$ admits a differentiably good refinement~\cite[Cor.~A.1]{FSS:Cech_diff_char_classes_via_L_infty}.
There is a canonical embedding $\Cart \hookrightarrow \Mfd$ of $\Cart$ into the category of all smooth manifolds.
Thereby, any manifold $M$ defines a presheaf $\ul{M}$ on $\Cart$, defined by $\ul{M}(c) = \Mfd(c,M)$.
\qen
\end{example}

\begin{example}
\label{eg:Op}
Let $\Op$ be the small category whose objects are open subsets of $\RN^n$ for any (varying) $n \in \NN_0$ and whose morphisms are all smooth maps between such open subsets.
The category $\Op$ has finite products and carries a coverage whose covering families are the open coverings.
The resulting site is denoted $(\Op,\tau_{op})$; observe that it is closed.
\qen
\end{example}

\begin{definition}
\label{def:tau-loc epi}
Let $(\scC,\tau)$ be a site.
A morphism $\pi \colon Y \to X$ in $\wC$ is a \emph{$\tau$-local epimorphism} if for every $c \in \scC$ and each morphism $\varphi \colon \scY_c \to X$ there exists a covering $\{f_j \colon c_j \to c\}_{j \in J} \in \tau(c)$ and morphisms $\{ \varphi_j \colon \scY_{c_j} \to Y \}_{j \in J}$ with $\pi \circ \varphi_j = \varphi \circ f_j$ for all $j \in J$.
\end{definition}

One directly checks the following:

\begin{lemma}
\label{st:subs and pullbacks}
For any site $(\scC,\tau)$, the class of $\tau$-local epimorphisms is stable under pullback.
In particular, the collection of $\tau$-local epimorphisms defines a coverage $\wtau$ on $\wC$.
\end{lemma}

Note that in the Grothendieck site $(\wC, \wtau)$ every covering family consists of a single morphism.
We list some general properties of $\tau$-local epimorphisms:

\begin{lemma}
\label{st:tau-loc epi properties}
Let $(\scC,\tau)$ be a site.
\begin{myenumerate}
\item Consider morphisms $p \in \wC(Y,X)$, $q \in \wC(Z,Y)$.
If $p \circ q$ is a $\tau$-local epimorphism, then so is $p$.

\item For any covering family $\{c_i \to c\}_{i \in I}$, the induced morphism $\coprod_{i \in I} \scY_{c_i} \to \scY_c$ is a $\tau$-local epimorphism.

\item $\tau$-local epimorphisms are stable under colimits:
let $\scJ$ be a small category, let $D', D \colon \scJ \to \wC$ be diagrams in $\wC$, and let $\pi \colon D' \to D$ be a morphism of diagrams such that each component $\pi_j \colon D'_j \to D_j$ is a $\tau$-local epimorphism, for any $j \in \scJ$.
Then, the induced morphism $\colim\, \pi \colon \colim\, D' \to \colim\, D$ is a $\tau$-local epimorphism.

\item If $X(c) \neq \emptyset$ for each $c \in \scC$, then the projection $X \times Z \to Z$ is a $\tau$-local epimorphism, for any $Z \in \wC$.

\item The site $(\scC,\tau)$ is closed if and only if $\tau$-local epimorphisms are stable under composition, if and only if $(\wC, \wtau)$ is closed.
\end{myenumerate}
\end{lemma}

\begin{proof}
Claims~(1) appears as property LE2 in~\cite[p.~390]{KS:Cats_and_Sheaves}.
Claim~(2) follows straightforwardly from the definition of a $\tau$-local epimorphism.
Claim~(3) can be found as~\cite[Prop.~16.1.12]{KS:Cats_and_Sheaves}.

Claim~(4) is immediate since by assumption the projection $X(c) \times Y(c) \to Y(c)$ is surjective for each $c \in \scC$; we can hence use the identity covering of $c$ to obtain the desired lift.

For~(5), it readily follows from the closedness of $(\scC,\tau)$ that $\tau$-local epimorphisms are stable under composition, and thus also that $(\wC,\wtau)$ is closed.
On the other hand, assume that $\tau$-local epimorphisms are stable under composition.
Consider an object $c \in \scC$, a covering family $\{f_i \colon c_i \to c\}_{i \in I}$, and for each $i \in I$ a covering family $\{c_{i,k} \to c_i\}_{k \in K_i}$.
By claims~(2) and~(3), we thus obtain $\tau$-local epimorphisms
\begin{equation}
	\coprod_{i \in I} \coprod_{k \in K_i} \scY_{c_{i,k}}
	\longrightarrow \coprod_{i \in I} \scY_{c_i}
	\longrightarrow \scY_c\,.
\end{equation}
By assumption, their composition is a $\tau$-local epimorphism again, and hence it follows (using that $\wC(\scY_c,-)$ preserves colimits) that $(\scC,\tau)$ is closed.
Finally, assume that $(\wC,\wtau)$ is closed.
Let $f \colon Z \to Y$ and $g \colon Y \to X$ be $\tau$-local epimorphisms, and let $p \coloneqq g \circ f$.
By assumption, there exists a $\tau$-local epimorphism $Z' \to X$ which factors through $g$; that is, there exists a morphism $q \colon Z' \to Z$ such that $p \circ q$ is a $\tau$-local epimorphism.
But then $p$ is a $\tau$-local epimorphism by claim~(1).
\end{proof}

Let $(\scC,\tau)$ be a site, and consider a covering $\CU = \{c_i \to c\}_{i \in I}$.
We can form its \emph{\v{C}ech nerve}, which is the simplicial object in $\wC$ whose level-$n$ object reads as
\begin{equation}
\label{eq:def Cech nerve}
	\cC\CU_n \coloneqq \coprod_{i_0, \ldots, i_n \in I} C_{i_0 \ldots i_n}\,,
	\quad \text{with} \quad
	C_{i_0 \ldots i_n} \coloneqq \scY_{c_{i_0}} \underset{\scY_c}{\times} \cdots \underset{\scY_c}{\times} \scY_{c_{i_n}}
	\quad \in \wC\,.
\end{equation}
Note that, in general, $C_{i_0 \ldots i_n} \in \wC$ is not representable as soon as $n \neq 0$.
The simplicial structure morphisms are given by projecting out or doubling the $i$-th factor, respectively.
Depending on the context we will view the \v{C}ech nerve $\cC\CU$ either as a simplicial object $\cC\CU_\bullet$ in $\wC$ or as an augmented simplicial object $\cC\CU_\bullet \to \scY_c$ in $\wC$.
For convenience, we recall the following definitions:

\begin{definition}
\label{def:sheaf}
Let $(\scC,\tau)$ be a site, and let $X \in \wC$ be a presheaf on $\scC$.
Then, $X$ is called a \emph{sheaf on $(\scC,\tau)$} if for every object $c \in \scC$ and for every covering family $\{f_i \colon c_i \to c\}_{i \in I} \in \tau(c)$ the diagram
\begin{equation}
\begin{tikzcd}
	X(c) = \wC(\scY_c, X) \ar[r, "X(f_i)"]
	& \displaystyle{\prod_{i_0 \in I}} \wC(C_{i_0}, X) \ar[r, shift left=0.1cm] \ar[r, shift left=-0.1cm]
	& \displaystyle{\prod_{i_0, i_1 \in I}} \wC\big( C_{i_0i_1}, X \big)
\end{tikzcd}
\end{equation}
is an equaliser diagram in $\Set_\rmL$.
The \emph{category $\Sh(\scC,\tau)$ of sheaves on $(\scC,\tau)$} is the full subcategory of $\wC$ on the sheaves.
\end{definition}

\section{Homotopy sheaves and descent along local epimorphisms}
\label{sec:higher sheaves}

Let $(\scC,\tau)$ be a small site, and let $\scY \colon \scC \to \wC$ denote its Yoneda embedding.
In this section we show that the homotopy right Kan extension along $\scY_*$ maps large higher $\tau$-sheaves on $\scC$ to large higher $\wtau$-sheaves on $\wC$.
Here, $\wtau$ is the pretopology of $\tau$-local epimorphisms on $\wC$ (see Definition~\ref{def:tau-loc epi} and Lemma~\ref{st:subs and pullbacks}).

\subsection{Sheaves of $\infty$-groupoids}
\label{sec:Sh_oo}

Let $\scC$ be a small category.
Let $\scH_{\infty,0} \coloneqq \Fun(\scC^\opp, \sSet_\rmL)$ denote the extra-large category of large simplicial presheaves on $\scC$ .
It is enriched, tensored and cotensored over $\sSet_\rmL$.
Composing the Yoneda embedding with the functor $\sfc_\bullet \colon \Set_\rmL \to \sSet_\rmL$, we obtain a fully faithful functor $\scY \colon \scC \to \scH_{\infty,0}$.
From now on, we view the category $\scH_{\infty,0}$ as endowed with the projective model structure.

\begin{definition}
A \emph{presheaf of $\infty$-groupoids on $\scC$} is a fibrant object in $\scH_{\infty,0}$.
\end{definition}

\begin{proposition}
\label{st:H_oo props}
$\scH_{\infty,0}$ is a left proper, tractable, model $\sSet_\rmL$-category.
If $\scC$ has finite products, then $\scH_{\infty,0}$ is additionally a symmetric monoidal model category.
\end{proposition}

\begin{proof}
The existence and tractability of the projective model structure follow from~\cite[Thm.~2.14]{Barwick:Enriched_B-Loc} (since $\scC$ is small).
$\scH_{\infty,0}$ is left proper since its pushouts and weak equivalences are defined objectwise, its cofibrations are (in particular) objectwise cofibrations, and $\sSet_\rmL$ is left proper.
Its simplicial enrichment is a consequence of~\cite[Prop.~4.50]{Barwick:Enriched_B-Loc}.
If $\scC$ has finite products, then $\scH_{\infty,0}$ is symmetric monoidal by~\cite[Prop.~4.52]{Barwick:Enriched_B-Loc}, using that the Yoneda embedding preserves limits and that $\scH_{\infty,0}$ is a simplicial model category.
\end{proof}

Note that~\cite[Prop.~4.52]{Barwick:Enriched_B-Loc} gives a more general criterion for when a projective model structure inherits a symmetric monoidal structure; however, for us the main case of interest will be where $\scC$ admits finite products.

Projective model categories of simplicial presheaves have an explicit cofibrant replacement functor $Q$~\cite{Dugger:Universal_HoThys}.
In order to write it down, it is convenient to introduce the two-sided bar construction; here, we follow~\cite[Sec.~4.2]{Riehl:Cat_HoThy}:

\begin{definition}
\label{def:bar construction}
Let $\scM$ be a large cocomplete category tensored over $\sSet_\rmL$.
Consider a large category $\scI$ and functors $H \colon \scI \to \scM$ and $G \colon \scI^\opp \to \sSet_\rmL$.
\begin{myenumerate}
\item The \emph{two-sided simplicial bar construction of $(G,H)$} is the simplicial object
\begin{equation}
	B_\bullet(G,\scI,H) \in \Fun(\bbDelta^\opp, \scM)\,,
	\qquad
	B_n(G,\scI,H) = \coprod_{i_0, \ldots, i_n} H(c_0) \otimes \scI(i_0, i_1) \otimes \cdots \otimes \scI(i_{n-1}, i_n) \otimes G(i_0)\,.
\end{equation}

\item The \emph{two-sided bar construction of $(F,G)$} is the realisation of $B_\bullet(G,\scI,H)$, i.e.
\begin{equation}
	B(G,\scI,H) = \int^{[n] \in \bbDelta^\opp} \Delta^n \otimes B_n(G,\scI,H)
	\quad \in \scM\,.
\end{equation}
\end{myenumerate}
\end{definition}

For later reference, we record the following direct consequences of this definition:

\begin{lemma}
\label{st:bar construction properties}
In the setting of Definition~\ref{def:bar construction}, the following statements hold true:
\begin{myenumerate}
\item Bar constructions commute:
let $\scJ$ be a second large category, and let $G' \colon \scJ^\opp \to \sSet_\rmL$ and $D \colon \scI \times \scJ \to \scM$ be diagrams.
Then, there is a canonical natural isomorphism
\begin{equation}
	B \big( G', \scJ, B(G, \scI, D) \big)
	\cong B \big( G, \scI, B(G', \scJ, D) \big)\,.
\end{equation}

\item Bar constructions are pointwise:
if $\scM = \Fun(\scC^\opp, \sSet_\rmL)$ is the category of large simplicial presheaves on a large category $\scC$ and $c \in \scC$ is any object, then evaluation at $c$ commutes with bar constructions:
there is a natural isomorphism
\begin{equation}
	B(G,\scI,H)(c) \cong B \big( G, \scI, H(c) \big)\,.
\end{equation}

\item In the above setting, let additionally $\scI = \bbDelta^\opp$.
Then there is a natural morphism
\begin{equation}
	B(*, \bbDelta^\opp, H) \longrightarrow \diag(H)\,,
\end{equation}
where $\diag$ denotes the functor taking the diagonal of a bisimplicial set.
This morphism is a weak equivalence of simplicial sets.

\item If $\scM$ is a large simplicial model category with cofibrant replacement functor $Q^\scM$, the bar construction is a model for the homotopy colimit, i.e.
\begin{equation}
	\underset{j \in \scJ}{\hocolim}^\scM (Dj)
	\simeq B(*, \scJ, Q^\scM \circ D)\,.
\end{equation}
In particular, if $Dj$ is already cofibrant in $\scM$ for each $j \in \scJ$, then we may compute the homotopy colimit of $D$ as
\begin{equation}
	\underset{j \in \scJ}{\hocolim}^\scM (Dj)
	\simeq B(*, \scJ, D)\,.
\end{equation}
\end{myenumerate}
\end{lemma}

\begin{proof}
Claim~(1) follows readily from the fact that colimits commute.
Claim~(2) is a direct consequence of the fact that colimits in presheaf categories are computed pointwise.
Finally, Claim~(3) is an application of the Bousfield-Kan map (see~\cite[Sec.~18.7]{Hirschhorn:MoCats} for details), and Claim~(4) is~\cite[Cor.~5.1.3]{Riehl:Cat_HoThy}.
\end{proof}

\begin{remark}
Analogous statements hold true in the Grothendieck universes $\rmS$ and $\XL$ of small and extra-large sets, respectively.
\qen
\end{remark}

We now define the cofibrant replacement functor in $\scH_{\infty,0}$ which we will be using throughout this article; it was introduced in~\cite{Dugger:Universal_HoThys}.

\begin{definition}
\label{def:Def Q}
We let $Q \colon \scH_{\infty,0} \to \scH_{\infty,0}$ denote the functor whose action on $F \in \scH_{\infty,0}$ is given by
\begin{equation}
\label{eq:Def Q}
	(QF)_n = B(F,\scC,\scY)_n
	= \coprod_{c_0, \ldots, c_n \in \scC} \scY_{c_0} \times \scC(c_0, c_1) \times \cdots \times \scC(c_{n-1}, c_n) \times F_n(c_n)\,.
\end{equation}
\end{definition}

Observe that for $X$ a \textit{simplicially constant} simplicial presheaf on $\scC$, we have a canonical isomorphism
\begin{equation}
	QX = B(*, \scC_{/X}, \scY)\,.
\end{equation}

\begin{definition}
\label{def:oo-shs}
Let $(\scC,\tau)$ be a small Grothendieck site, and let $\ctau$ denote the class of morphisms in $\scH_{\infty,0}$ consisting of \v{C}ech nerves of coverings in $(\scC,\tau)$ (see~\eqref{eq:def Cech nerve}).
Since $\scC$ is small, this is a small set.
We let \smash{$\scH_{\infty,0}^{loc}$} denote the left Bousfield localisation \smash{$\scH_{\infty,0}^{loc} \coloneqq L_\ctau \scH_{\infty,0}$} of the large category $\scH_{\infty,0}$.
A \emph{sheaf of $\infty$-groupoids on $(\scC,\tau)$} is a fibrant object in \smash{$\scH_{\infty,0}^{loc}$}.
\end{definition}

The model structure on $\scH_{\infty,0}^{loc}$ is also called the \emph{local projective model structure} on $\Fun(\scC^\opp, \sSet_\rmL)$.
Explicitly, an object $F$ in $\Fun(\scC^\opp, \sSet_\rmL)$ is fibrant if
\begin{myenumerate}
\item $F(c)$ is a Kan complex for every $c \in \scC$, and

\item for every covering family $\CU = \{c_i \to c\}_{i \in I}$ in $(\scC,\tau)$, the morphism
\begin{align}
\label{eq:H_oo fib generic}
	\scF(c) \longrightarrow &\underset{n \in \bbDelta}{\holim}^{\sSet_\rmL} \Big( \cdots \ul{\scH}_{\infty,0} \big( Q (\cC \CU_n),\, F \big)\ \cdots \Big)
	\\
	&= \underset{n \in \bbDelta}{\holim}^{\sSet_\rmL} \Big( \cdots \prod_{i_0, \ldots, i_n \in I} \ul{\scH}_{\infty,0}( Q C_{i_0 \ldots i_n},\, F)\ \cdots \Big)
\end{align}
is a weak equivalence in $\sSet_\rmL$ (compare~\eqref{eq:def Cech nerve}).
Here we have used that $Q$ is a left adjoint.
\end{myenumerate}

\begin{proposition}
\label{st:H_oo^loc props}
$\scH_{\infty,0}^{loc}$ is a proper, tractable, $\sSet_\rmL$-model category.
If $\scH_{\infty,0}$ is symmetric monoidal as a model category, then so is $\scH_{\infty,0}^{loc}$.
\end{proposition}

\begin{proof}
The first claim follows form~\cite[Thm.~2.14, Thm~4.56]{Barwick:Enriched_B-Loc} and the fact that $\sSet_\rmL$ is a large, proper, tractable model category.
The second claim is an application of~\cite[Thm.~4.58]{Barwick:Enriched_B-Loc}.
\end{proof}

\begin{remark}
If $\scC$ has finite products, then the conditions of~\cite[Prop.~4.52]{Barwick:Enriched_B-Loc} are satisfied, so that in this case $\scH_{\infty,0}$, and hence also $\scH_{\infty,0}^{loc}$, are symmetric monoidal model categories.
\qen
\end{remark}

The $\sSet_\rmL$-enriched hom in both $\scH_{\infty,0}$ and $\scH_{\infty,0}^{loc}$ is given by
\begin{equation}
	\ul{\scH}_{\infty,0}(F,G)
	= \int_{c \in \scC^\opp} (G(c))^{F(c)}
	\qquad \in \sSet_\rmL\,,
\end{equation}
where the argument of the end is given by the internal hom in $\sSet_\rmL$.
In particular, there is a natural isomorphism $\ul{\scH}_{\infty,0}(\scY_c, F) \cong F(c)$ for any $c \in \scC$ and any $F \in \scH_{\infty,0}$.

Since $\scC$ is a small category (hence also large, since $\rmS \in \rmL$) and $\Set_\rmS$ is a large category, $\wC = \Fun(\scC^\opp, \Set_\rmS)$ is a large category.
We define the extra-large category
\begin{equation}
	\wH_{\infty,0} \coloneqq \Fun \big( \wC{}^\opp, \sSet_\rmL \big)\,.
\end{equation}
Since $\sSet_\rmL$ is $\rmL$-tractable, the projective model structure on $\wH_{\infty,0}$ exists and is itself $\rmL$-tractable~\cite[Thm.~2.14]{Barwick:Enriched_B-Loc}.
We will always view $\wH_{\infty,0}$ as endowed with this projective model structure.

By Proposition~\ref{st:subs and pullbacks}, if $(\scC,\tau)$ is a small site, then $(\wC, \wtau)$ is a large site.
Let $(\wtau){}\check{}$ denote the large set of \v{C}ech nerves of coverings in $(\wC,\wtau)$ (this is a large set since $\wC$ is large).
Since $\wH_{\infty,0}$ is left proper and $\rmL$-tractable, and $(\wtau){}\check{}$ is $\rmL$-small, we can form the left Bousfield localisation
\begin{equation}
\label{eq:def wH^loc_oo}
	\wH_{\infty,0}^{loc} \coloneqq L_{(\wtau){}\check{}} \wH_{\infty,0}\,.
\end{equation}
Both $\wH_{\infty,0}$ and $\wH_{\infty,0}^{loc}$ are symmetric monoidal since $\wC$ has finite products (by Propositions~\ref{st:H_oo props} and~\ref{st:H_oo^loc props}, respectively).

The Yoneda embedding $\scY \colon \scC \to \wC$ induces a pullback functor $\scY^* \colon \wH_{\infty,0} \to \scH_{\infty,0}$.
Since both $\scC$ and $\wC$ are large, and since both $\wH_{\infty,0}$ and $\scH_{\infty,0}$ have all large limits and colimits, there is a triple of adjunctions
\begin{equation}
\begin{tikzcd}[column sep=2cm]
	\wH_{\infty,0} \ar[r, "\scY^*" description]
	& \scH_{\infty,0}\,. \ar[l, bend left=30, "\scY_*" description, "\perp"' {yshift=0.15cm,pos=0.51}] \ar[l, bend left=-30, "\scY_!" {description, pos=0.54}, "\perp" {yshift=-0.15cm,pos=0.54}]
\end{tikzcd}
\end{equation}
We can compute the right adjoint as
\begin{equation}
	(\scY_* F)(X) = \int_{c \in \scC^\opp} F(c)^{(\wC)^\opp(X,\scY_c)}
	\cong \int_{c \in \scC^\opp} F(c)^{X(c)}
	\cong \ul{\scH}_{\infty,0}(X,F)\,.
\end{equation}
(Note that here we view $X \in \wC$ as an object in $\scH_{\infty,0}$ via the canonical inclusions $\Set_\rmS \hookrightarrow \Set_\rmL \hookrightarrow \sSet_\rmL$.)

We now ask whether one of the functors $\scY_!$ or $\scY_*$ maps sheaves of $\infty$-groupoids on $(\scC,\tau)$ to sheaves of $\infty$-groupoids on $(\wC, \wtau)$, i.e.~preserves fibrant objects as a functor \smash{$\scH_{\infty,0}^{loc} \to \wH_{\infty,0}^{loc}$}.
Since this indicates that we are looking for a right Quillen functor, we focus on the right adjoint $\scY_*$.
In general, $\scY_*$ will not even preserve fibrant objects as a functor $\scH_{\infty,0} \to \wH_{\infty,0}$.
However, if $\tilde{Q} \colon \scH_{\infty,0} \to \scH_{\infty,0}$ is a functorial cofibrant replacement, then the functor
\begin{equation}
\label{Peq:def S^Q}
	S^{\tilde{Q}}_{\infty,0} \coloneqq \tilde{Q}^* \circ \scY_* \colon \scH_{\infty,0} \longrightarrow \wH_{\infty,0}\,,
	\qquad
	\big( S^{\tilde{Q}}_{\infty,0}(F) \big)(X) \coloneqq \ul{\scH}(\tilde{Q}X,F)
\end{equation}
preserves fibrations and trivial fibrations:
$\scH_{\infty,0}$ is a model $\sSet_\rmL$-category, and so the simplicially enriched hom
\begin{equation}
	\ul{\scH}_{\infty,0}(A,-) \colon \ul{\scH}_{\infty,0} \to \sSet_\rmL
\end{equation}
is a right Quillen functor whenever $A \in \ul{\scH}_{\infty,0}$ is cofibrant.

\begin{remark}
The above holds true for \textit{any} cofibrant replacement functor $\tilde{Q}$ in $\scH_{\infty,0}$.
However, for technical reasons we always use Dugger's functor $Q$ from~\eqref{eq:Def Q} for cofibrant replacement in the projective model structure of simplicial presheaves in the remainder of this article.
We abbreviate $S_{\infty,0} \coloneqq S^Q_{\infty,0}$.
\qen
\end{remark}

Let $F \in \scH_{\infty,0}$ and $X \in \wC$.
Using that $QX$ is a bar construction we compute
\begin{align}
	(S_{\infty,0} F)(X) &= \ul{\scH}_{\infty,0}(QX,F)
	\\
	&= \ul{\scH}_{\infty,0} \big( B^{\scH_{\infty,0}}(*, \scC_{/X}, \scY),\, F \big)
	\\
	&\cong C_{\sSet_\rmL} \big( *, (\scC_{/X})^\opp, F \big)\,,
\end{align}
where $C_{\sSet_\rmL}$ and $B^{\scH_{\infty,0}}$ are the cobar and bar constructions in $\sSet_\rmL$ and in $\scH_{\infty,0}$, respectively.
If $F$ is projectively fibrant, then the last expression is a model for the homotopy limit (see also recalled as Proposition~\ref{st:cobar computes holim} below),
\begin{equation}
	(S_{\infty,0} F)(X) \cong \holim^{\sSet_\rmL} \big( (\scC_{/X})^\opp \xrightarrow{\pr_{\scC^\opp}} \scC^\opp \overset{F}{\longrightarrow} \sSet_\rmL \big)\,.
\end{equation}
Now, using the Yoneda Lemma and~\cite[Ex.~9.2.11]{Riehl:Cat_HoThy} we conclude:

\begin{proposition}
\label{st:S^Q_oo = hoRan_(Y^op)}
 If $F \in \scH_{\infty,0}$ is fibrant, $S_{\infty,0} F$ is the homotopy right Kan extension of $F$ along the (opposite of the) Yoneda embedding of $\scC$, i.e.
\begin{equation}
	S_{\infty,0} F \simeq \hoRan_\scY (F)\,.
\end{equation}
\end{proposition}

\begin{remark}
In view of Section~\ref{sec:char of sheaves on wC}, we emphasize that \smash{$S_{\infty,0}$} is a presentation of the $\infty$-categorical right Kan extension of presheaves of spaces on $\scC$ to presheaves of spaces on $\wC$.
\qen
\end{remark}

In order to show that \smash{$S_{\infty,0}$} is a right Quillen functor, it remains to show that it is a right adjoint.
Since $\wC$ is a large category, $\wC$-indexed (co)ends exist in $\sSet_\rmL$, the extra-large category of large simplicial sets.
Therefore, a left adjoint to $S_{\infty,0}$ is given by
\begin{equation}
\label{eq:tildeRe^Q}
	\widetilde{Re}_{\infty,0} \colon \wH \to \scH\,,
	\qquad
	\widetilde{Re}_{\infty,0}(\widehat{G}) \coloneqq \int^{X \in \wC^\opp} \widehat{G}(X) \otimes QX\,.
\end{equation}
We also set
\begin{equation}
	Re_{\infty,0}(\widehat{G}) \coloneqq Q \circ \scY^*(\widehat{G})
\end{equation}

\begin{lemma}
\label{st:Re^Q = Q Y^*}
There is a canonical natural isomorphism
\begin{equation}
	\widetilde{Re}_{\infty,0}
	\cong Re_{\infty,0}\,.
\end{equation}
Thus, $Re_{\infty,0}$ is also a left adjoint to $S_{\infty,0}$.
\end{lemma}

\begin{proof}
Using~\eqref{eq:Def Q},  we have the following natural isomorphisms
\begin{align}
	\big( \widetilde{Re}_{\infty,0}(\widehat{G}) \big)_n
	&= \int^{X \in \wC^\opp} \coprod_{c_0, \ldots, c_n \in \scC} \scY_{c_0} \times \scC(c_0, c_1) \times \cdots \times \scC(c_{n-1}, c_n) \times \wC(\scY_{c_n},X) \times \widehat{G}_n(X)
	\\[0.1cm]
	&\cong \coprod_{c_0, \ldots, c_n \in \scC} \scY_{c_0} \times \scC(c_0, c_1) \times \cdots \times \scC(c_{n-1}, c_n) \times \widehat{G}_n(\scY_{c_n})
	\\*[0.1cm]
	&= \big( (Q \circ \scY^*)(\widehat{G}) \big)_n\,,
\end{align}
where the isomorphism is an application of the Yoneda Lemma for $\wC$.
\end{proof}

\begin{proposition}
\label{st:Re^Q -| S^Q}
The functors $Re_{\infty,0}$ and $S_{\infty,0}$ give rise to a simplicial Quillen adjunction which sits inside the following non-commutative diagram
\begin{equation}
\begin{tikzcd}[column sep=1.75cm, row sep=1.25cm]
	\wC \ar[d, "Q"'] \ar[r, "\widehat{\scY}"] & \wH_{\infty,0} \ar[dl, shift left=-0.1cm, "Re_{\infty,0}"' xshift=0.2cm]
	\\
	\scH_{\infty,0} \ar[ur, shift left=-0.1cm, "S_{\infty,0}"' xshift=-0.2cm] & 
\end{tikzcd}
\end{equation}
Further, there exists a natural isomorphism $\eta \colon Re_{\infty,0} \circ \widehat{\scY} \arisom Q$.
\end{proposition}

\begin{proof}
We have already shown above that $Re_{\infty,0} \dashv S_{\infty,0}$ and that $S_{\infty,0}$ preserves fibration and trivial fibrations.
The rest is then a direct application of~\cite[Prop.~2.3]{Dugger:Universal_HoThys}.
In the present case, $\eta$ as written down there turns out to be an isomorphism as a consequence of the Yoneda Lemma:
recall that $Re_{\infty,0} = Q \circ \scY^*$ and that there is a canonical isomorphism
\begin{equation}
	\scY^*(\wY_X) = \wC(\scY_{(-)}, X) \cong X\,,
\end{equation}
natural in $X \in \wC$.
\end{proof}

\begin{proposition}
The functors $Re_{\infty,0}$ and $S_{\infty,0}$ have the following properties:
\begin{myenumerate}
\item The functor $Re_{\infty,0} \colon \wH_{\infty,0} \to \scH_\infty$ is homotopical.

\item Let $\wQ \colon \wH_{\infty,0} \to \wH_{\infty,0}$ denote the cofibrant replacement functor on $\wH_{\infty,0}$ defined in analogy with~\eqref{eq:Def Q}.
For each fibrant object $F \in \scH_{\infty,0}$, there is a zig-zag of weak equivalences
\begin{equation}
	F \overset{\sim}{\longleftarrow} Q F
	\weq Re_{\infty,0} \circ S_{\infty,0}(F)
	\overset{\sim}{\longleftarrow} Re_{\infty,0} \circ \wQ \circ S_{\infty,0} (F) \,,
\end{equation}
which is natural in $F$.
\end{myenumerate}
\end{proposition}

\begin{proof}
By Lemma~\ref{st:Re^Q = Q Y^*}, \smash{$Re_{\infty,0} \cong Q \circ \scY^*$}.
The functor $\scY^*$ is homotopical since weak equivalences in both $\scH_{\infty,0}$ and in $\wH_{\infty,0}$ are defined objectwise.
Thus, claim (1) follows since $Q$ is homotopical.
As a consequence, the morphism
\begin{equation}
	Re_{\infty,0} \circ \wQ \circ S_{\infty,0} (F) \longrightarrow Re_{\infty,0} \circ S_{\infty,0}(F)
\end{equation}
is a weak equivalence.
For $F \in \scH_{\infty,0}$, we find that
\begin{align}
	\big( \scY^* \circ S_{\infty,0}(F) \big) (c)
	= S_{\infty,0}(F)(\scY_c)
	= \ul{\scH}_{\infty,0}(Q\scY_c, F)\,,
\end{align}
and since $F$ is fibrant and $\scY_c$ is cofibrant, the functor $\ul{\scH}_{\infty,0}(-,F)$ preserves the weak equivalence $Q\scY_c \to \scY_c$.
Hence, there is a natural weak equivalence $F \weq \scY^* \circ S_{\infty,0}(F)$.
To complete the proof, we apply the homotopical functor $Q$ to this morphism.
\end{proof}

We now investigate whether the Quillen adjunction $Re_{\infty,0} \dashv S_{\infty,0}$ between the projective model structures descends to a Quillen adjunction between the \emph{local} projective model structures.
We recall the following standard result:

\begin{proposition}
\label{st:QAds and BLocs}
\emph{\cite[Prop.~3.1.6, Prop.~3.3.18]{Hirschhorn:MoCats}}
Let $\scM$ and $\scN$ be two large simplicial model categories, and let $F \dashv G$ be a simplicial Quillen adjunction from $\scM$ to $\scN$.
Suppose that $A,B$ are large sets of morphisms in $\scM$ and in $\scN$, respectively, such that the left Bousfield localisations $L_A \scM$ and $L_B \scN$ exist.
If $G$ maps $B$-local objects to $A$-local objects, then $F \dashv G$ descends to a Quillen adjunction between the localised model categories.
\end{proposition}

\begin{theorem}[{\cite[Cor.~A.3]{DHI:Hypercovers_and_sPShs}}]
\label{st:local epi Thm}
Let $\scC$ be a small category endowed with a coverage $\tau$, and let $\pi \colon Y \to X$ be a $\tau$-local epimorphism in $\wC$ with \v{C}ech nerve $\cC \pi$.
The augmentation map $\pi^{[\bullet]} \colon \cC \pi \to X$ is a weak equivalence in $\scH_{\infty,0}^{loc}$.
\end{theorem}

The proof of Theorem~\ref{st:local epi Thm} is rather technical; it requires a ``wrestling match with the small object argument'' (\cite{DHI:Hypercovers_and_sPShs} Subsection~A.12).
We refer the reader to that reference for details.

\begin{remark}
In~\cite{DHI:Hypercovers_and_sPShs}, Theorem~\ref{st:local epi Thm} is proven for the \v{C}ech localisation of the \emph{injective} model structure $\scH^i$ on $\scH$.
However, it also holds true in $\scH_{\infty,0}$ because the injective and projective model structures on $\scH$ (and $\wH$) have the same weak equivalences, and so do the associated local model structures (see also~\cite[Prop.~2.6]{Bunk:R-loc_HoThy} for details).
\qen
\end{remark}

\begin{lemma}
\label{st:Cech nerve and hocolim}
Let $\pi \colon Y \to X$ be a $\tau$-local epimorphism in $\wC$.
Let $(\cC\pi)^\sim \colon \bbDelta^\opp \to \scH_{\infty,0}$ denote the simplicial diagram which sends $[n] \in \bbDelta$ to the simplicially constant presheaf $\cC\pi_n$.
There is a commutative triangle in $\scH_{\infty,0}$:
\begin{equation}
\label{eq:Cech nerve and hocolim}
\begin{tikzcd}[column sep={2.cm,between origins}, row sep={1.5cm,between origins}]
	\hocolim_{n \in \bbDelta^\opp}^{\scH_{\infty,0}} (\cC\pi)^\sim \ar[rr] \ar[dr]
	& & \cC \pi \ar[dl]
	\\
	& X &
\end{tikzcd}
\end{equation}
The top morphism is an objectwise weak equivalence, and the diagonal morphisms are weak equivalences in $\scH_{\infty,0}^{loc}$.
\end{lemma}

\begin{proof}
The top morphism is induced by the Bousfield-Kan, or last-vertex map.
It is an objectwise weak equivalence by Lemma~\ref{st:bar construction properties}(3).
The right-hand side is a local weak equivalence by Theorem~\ref{st:local epi Thm}.
Thus, it remains to show that the triangle commutes.
We check this explicitly:
for each object $c \in \scC$, there are canonical natural isomorphisms
\begin{equation}
	\hocolim_{n \in \bbDelta^\opp}^{\scH_{\infty,0}} (\cC\pi)^\sim
	= B \big( *, \bbDelta^\opp, (\cC\pi)^\sim \big)
	\cong \sfN \big( \bbDelta_{/(\cC\pi(c))} \big)_n\,,
\end{equation}
where $\sfN$ denotes the nerve functor.
For $X \in \sSet_\rmS$, recall the last-vertex map
\begin{equation}
	\sfN \big( \bbDelta_{/X} \big) \longrightarrow X
\end{equation}
of small simplicial sets (see, for instance,~\cite[Lemma~7.3.11, Par.~7.3.14]{Cisinski:Higher_Cats_and_Ho_Alg}).
We describe this explicitly in our case, i.e.~for $X = \cC \pi(c)$.
An $n$-simplex in this simplicial set is explicitly given by a sequence of morphisms
\begin{equation}
\begin{tikzcd}
	\Delta^{k_0} \ar[r, "\alpha_{01}"]
	& \Delta^{k_1} \ar[r, "\alpha_{01}"]
	& \cdots \ar[r, "\alpha_{n-1,n}"]
	& \Delta^{k_n} \ar[r, "\kappa"]
	& \cC\pi(c)
\end{tikzcd}
\end{equation}
of small simplicial sets.
We denote these data by $(\alpha, \kappa)$.
Consider the map $\varphi_\alpha \colon [n] \to [k_n]$ in $\bbDelta$ which acts as
\begin{equation}
	i \longmapsto \alpha_{n-1,n} \circ \cdots \circ \alpha_{i,i+1}(k_i)\,.
\end{equation}
The top morphism in diagram~\eqref{eq:Cech nerve and hocolim} acts by sending the $n$-simplex $(\alpha, \kappa)$ to the $n$-simplex of $\cC\pi(c)$ which is given by the composition
\begin{equation}
\begin{tikzcd}
	\Delta^n \ar[r]
	& \Delta^{k_n} \ar[r, "\kappa"]
	& \cC\pi(c)\,,
\end{tikzcd}
\end{equation}
where the first map is the one constructed above from the data $(\alpha, \kappa)$.
Explicitly, for a given $n$-simplex $\alpha \in \sfN \bbDelta$, the resulting map
\begin{equation}
	\cC\pi_{k_n} \cong \underbrace{Y \times_X \cdots \times_X Y}_{\text{$k_n$ copies of $Y$}}
	\longrightarrow \cC\pi_n \cong \underbrace{Y \times_X \cdots \times_X Y}_{\text{$n$ copies of $Y$}}
\end{equation}
is the projection onto the $Y$-factors in positions $\varphi_\alpha(0), \ldots, \varphi_\alpha(n)$.
This map commutes with the maps to $X$.
\end{proof}

\begin{proposition}
\label{st:S^Q preserves oo-stacks}
The functor $S_{\infty,0} \colon \scH_{\infty,0}^{loc} \to \wH_{\infty,0}^{loc}$ preserves local objects.
\end{proposition}

\begin{proof}
Let $\pi \colon Y \to X$ be a $\tau$-local epimorphism in $\wC$, and let \smash{$F \in \scH_{\infty,0}^{loc}$} be fibrant.
We need to show that the canonical morphism
\begin{equation}
	S_{\infty,0}(F)(X) \longrightarrow \underset{n \in \bbDelta^\opp}{\holim}^{\sSet_\rmL}\ \big( (S_{\infty,0} F)(\cC \pi_n) \big)
\end{equation}
is a weak equivalence in $\sSet_\rmL$.
Substituting the definition of \smash{$S_{\infty,0}$}, this is the same as the canonical map
\begin{equation}
	\ul{\scH}_{\infty,0}(QX, F) \longrightarrow \underset{n \in \bbDelta^\opp}{\holim}^{\sSet_\rmL}\ \ul{\scH}_{\infty,0} \big( Q (\cC \pi_n), F \big)\,.
\end{equation}
Observe that there is a canonical isomorphism
\begin{equation}
\begin{tikzcd}
	\underset{n \in \bbDelta^\opp}{\holim}^{\sSet_\rmL}\ \ul{\scH}_{\infty,0} \big( Q (\cC \pi_n), F \big)
	\ar[r, "\cong"]
	& \ul{\scH}_{\infty,0} \Big( \underset{n \in \bbDelta^\opp}{\hocolim}^{\scH_{\infty,0}} Q (\cC \pi_n),\, F \Big)
\end{tikzcd}
\end{equation}
in the homotopy category $\Ho \sSet_\rmL$ (in fact, this can even be modelled as an isomorphism in $\sSet_\rmL$ in light of Lemma~\ref{st:bar construction properties}(1) and Proposition~\ref{st:cobar computes holim} below).
Since the functor $Q$ is a bar construction (see~\eqref{eq:Def Q}), it commutes with the homotopy colimit (which is also a bar construction) by Lemma~\ref{st:bar construction properties}(1).
Then, the morphism
\begin{equation}
\begin{tikzcd}
	\ul{\scH}_{\infty,0}(QX, F) \ar[r]
	& \ul{\scH}_{\infty,0} \Big( \underset{n \in \bbDelta^\opp}{\hocolim}^{\scH_{\infty,0}} Q (\cC \pi_n),\, F \Big)
\end{tikzcd}
\end{equation}
is the image of the left-hand diagonal morphism in the commutative triangle~\eqref{eq:Cech nerve and hocolim} under the functor \smash{$\ul{\scH}_{\infty,0}(Q(-),F)$}.
Observe that the functor
\begin{equation}
	\ul{\scH}_{\infty,0}(Q(-),F) \colon \scH_{\infty,0}^{loc} \longrightarrow \sSet_\rmL
\end{equation}
is homotopical since $F$ is fibrant in \smash{$\scH_{\infty,0}^{loc}$} and \smash{$\scH_{\infty,0}^{loc}$} is enriched over $\sSet_\rmL$.
Combining this with Lemma~\ref{st:Cech nerve and hocolim}, we obtain that the above morphism is a weak equivalence in $\sSet_\rmL$.
\end{proof}

Combining Proposition~\ref{st:QAds and BLocs} and Proposition~\ref{st:S^Q preserves oo-stacks}, we obtain

\begin{theorem}
\label{st:QAd for H_(oo,0)^loc}
The Quillen adjunction $Re_{\infty,0}: \wH_{\infty,0} \rightleftarrows \scH_{\infty,0} :S_{\infty,0}$ induces a Quillen adjunction
\begin{equation}
\begin{tikzcd}
	Re_{\infty,0} : \wH_{\infty,0}^{loc} \ar[r, shift left=0.15cm, "\perp"' yshift=0.04cm] & \scH_{\infty,0}^{loc} : S_{\infty,0}\,. \ar[l, shift left=0.15cm]
\end{tikzcd}
\end{equation}
\end{theorem}

In Section~\ref{sec:char of sheaves on wC} we provide a version of Theorem~\ref{st:QAd for H_(oo,0)^loc} for fully coherent diagrams of spaces in a quasi-categorical language; on the underlying quasi-categories, we will show that $S_{\infty,0}$ is fully faithful and we identify its essential image.

\subsection{Homotopy sheaves with values in simplicial model categories}
\label{sec:HoSheaves vald in spl MoCats}

We now extend the results from Section~\ref{sec:Sh_oo} to homotopy sheaves on a small site $(\scC, \tau)$, valued not just in $\sSet_\rmL$, but in any large, combinatorial or cellular simplicial model category.
We begin by recalling some technology for enriched model categories.

Recall the notion of a $\scV$ be a symmetric monoidal model category.
A \emph{$\scV$-enriched model category}, or \emph{model $\scV$-category} $\scM$ (see, for instance,~\cite[Def.~4.2.18]{Hovey:MoCats}).
We denote the three functors which are part of the enrichment by
\begin{equation}
	(-) \otimes (-) \colon \scV \times \scM \to \scM\,,
	\qquad
	\ul{\scM}^\scV(-,-) \colon \scM^\opp \times \scM \to \scV\,,
	\qquad
	\{-,-\} \colon \scV^\opp \times \scM \to \scM\,.
\end{equation}

We import the following definition (using~\cite[Thms.~7.6.3, 11.5.1]{Riehl:Cat_HoThy}):

\begin{definition}
\label{def:functor (co)tensor}
Let $\scI$ be a small category.
\begin{myenumerate}
\item The \emph{weighted $\scI$-colimit functor} of $\scM$ is the functor
\begin{equation}
\label{eq:colim^scV}
	(-) \otimes_\scI (-) \colon \Fun(\scI^\opp, \scV) \times \Fun(\scI, \scM) \longrightarrow \scM\,,
	\qquad
	(D, F) \longmapsto \int^{i \in \scI} Di \otimes Fi\,.
\end{equation}

\item The \emph{weighted $\scI$-limit functor} of $\scM$ is the functor
\begin{equation}
\label{eq:lim^scV}
	\{-,-\}^\scI \colon \Fun(\scI, \scV)^\opp \times \Fun(\scI, \scM) \longrightarrow \scM\,,
	\qquad
	(D, F) \longmapsto \int_{i \in \scI} \{Di, Fi\}\,.
\end{equation}
\end{myenumerate}
\end{definition}

From now on we consider $\scV = \sSet_\rmL$ with the Kan-Quillen model structure, and thus $\scM$ a large simplicial model category.
The following result is \cite[Thms.~3.2, 3.3]{Gambino:Weighted_limits_in_spl_HoThy} (see also \cite[Thm.~11.5.1]{Riehl:Cat_HoThy}).

\begin{theorem}
\label{st:weighted (co)lim as Quillen functors}
Let $\scM$ be a large simplicial model category and $\scI$ a large category.
Suppose that the projective and injective model structures on $\Fun(\scI, \scM)$ both exist.
\begin{myenumerate}
\item The weighted $\scI$-colimit functor~\eqref{eq:colim^scV} is a left Quillen functor of two variables~\cite[Def.~4.2.1]{Hovey:MoCats} if its source and target are either endowed with the projective and injective model structures, respectively, or with the injective and projective model structures, respectively.

\item The weighted $\scI$-limit functor~\eqref{eq:lim^scV} is a right Quillen functor of two variables if its source and target are either both endowed with the projective model structures, or are both endowed with the injective model structures.
\end{myenumerate}
\end{theorem}

Recall that if $\scM$ is combinatorial (in $\rmL$), then the projective and injective model structures on $\Fun(\scI, \scM)$ exist and are again combinatorial (see, for instance,~\cite[Thms.~2.14, 2.16]{Barwick:Enriched_B-Loc}).
The projective model structure also exists when $\scM$ is cellular, and is itself cellular~\cite[Prop.~12.1.5]{Hirschhorn:MoCats}.

We also define functors
\begin{alignat}{3}
	(-) \otimes_\scI^{ptw} (-) &\colon \Fun(\scI^\opp, \sSet_\rmL) \times \scM
	\longrightarrow \Fun(\scI, \scM)\,,
	& \qquad
	(E \otimes m)(i) &\coloneqq E(i) \otimes m\,,
	\\
	\ul{\scM}^{ptw}_\scI(-,-) &\colon \scM^\opp \times \Fun(\scI, \scM)^\opp
	\longrightarrow \Fun(\scI^\opp, \sSet_\rmL)\,,
	& \qquad
	\ul{\scM}^{ptw}_\scI(m, F) (i) &\coloneqq \ul{\scM} \big( m, F(i) \big)\,.
\end{alignat}
In the original reference~\cite{Gambino:Weighted_limits_in_spl_HoThy} the following observation is the crucial ingredient in the proof of Theorem~\ref{st:weighted (co)lim as Quillen functors}:

\begin{proposition}
\label{st:2VQAd for Fun(scI, scM)}
In the situation of Theorem~\ref{st:weighted (co)lim as Quillen functors}(2), with either the projective or injective model structure on both $\Fun(\scI, \sSet_\rmL)$ and $\Fun(\scI, \scM)$, the functors
\begin{align}
	(-) \otimes_\scI^{ptw} (-) &\colon \Fun(\scI, \sSet_\rmL) \times \scM
	\longrightarrow \Fun(\scI, \scM)\,,
	\\
	\ul{\scM}^{ptw}_\scI(-,-) &\colon \scM^\opp \times \Fun(\scI, \scM)
	\longrightarrow \Fun(\scI, \sSet_\rmL)\,,
	\\
	\{-,-\}^\scI &\colon \Fun(\scI, \sSet_\rmL)^\opp \times \Fun(\scI, \scM)
	\longrightarrow \scM\,,
\end{align}
form a Quillen adjunction of two variables.
\end{proposition}

\begin{proof}
In the projective case the functor
\begin{equation}
	\ul{\scM}^{ptw}_\scI (-,-) \colon \scM^\opp \times \Fun(\scI, \scM)^\opp
	\longrightarrow \Fun(\scI^\opp, \sSet_\rmL)\,,
	\qquad
	\ul{\scM}^{ptw}_\scI (m, F) (i) \coloneqq \ul{\scM} \big( m, F(i) \big)
\end{equation}
applies the $\sSet_\rmL$-enrichment of $\scM$ objectwise.
Since pullbacks in functor categories are computed objectwise, it follows from the properties of this enrichment that $\ul{\scM}^{ptw}_\scI(-,-)$ satisfies the pullback-product axiom (see~\cite[Lemma~4.2.2]{Hovey:MoCats}).
One proves the injective case analogously, using the functor $(-) \otimes^{ptw}_\scI (-)$ in place of $\ul{\scM}^{ptw}_\scI$.
\end{proof}

As before, we let $Q \colon \scH_{\infty,0} \to \scH_{\infty,0}$ denote Dugger's cofibrant replacement functor (see Definition~\ref{def:Def Q}).
Consider the functors
\begin{alignat}{3}
	Re_\scM \colon \Fun(\wC^\opp, \scM) &\longrightarrow \Fun(\scC^\opp, \scM)\,,
	& \qquad
	Re_\scM(\widehat{G})(c) &\coloneqq Q(-)(c) \otimes_{\wC^\opp} \widehat{G}
	= \int^{X \in \wC} QX(c) \otimes \widehat{G}(X)\,,
	\\*
	S_\scM \colon \Fun(\scC^\opp, \scM) &\longrightarrow \Fun(\wC^\opp, \scM)\,,
	& \qquad
	S_\scM(F)(X) &\coloneqq \{QX, F\}^{\scC^\opp}
	= \int_{c \in \scC} \big\{ QX(c), Fc \big\}\,.
\end{alignat}
They form an adjoint pair
\begin{equation}
\label{eq:Re_M -| S_M adjunction}
\begin{tikzcd}
	Re_\scM : \Fun(\wC^\opp, \scM) \ar[r, shift left=0.15cm, "\perp"' yshift=0.05cm]
	& \Fun(\scC^\opp, \scM) : S_\scM\,. \ar[l, shift left=0.15cm]
\end{tikzcd}
\end{equation}

\begin{proposition}
\label{st:Re_M --| S_M is Quillen adjunction}
The adjunction~\eqref{eq:Re_M -| S_M adjunction} is a Quillen adjunction with respect to the projective model structures.
\end{proposition}

\begin{proof}
Let $X \in \wC$, and consider the functor
\begin{equation}
	\ev_X \circ S_\scM \colon \Fun(\scC^\opp, \scM) \longrightarrow \scM\,,
	\qquad
	F \longmapsto \int_{c \in \scC} \big\{ QX(c), Fc \big\}
	= \{QX, F\}^{\scC^\opp}\,.
\end{equation}
Here we have used the notation from~\eqref{eq:lim^scV}.
By Theorem~\ref{st:weighted (co)lim as Quillen functors} and~\cite[Rmk.~4.2.3]{Hovey:MoCats} this functor is right Quillen, for each $X \in \wC$.
Since fibrations and weak equivalences in the projective model structure on $\Fun(\wC^\opp, \scM)$ are defined objectwise, it follows that the functor $S_\scM$ preserves projective fibrations and trivial fibrations.
\end{proof}

We recall some useful technology for computing homotopy limits from~\cite{Riehl:Cat_HoThy} (see, in particular, Chapters 4--6).
Let $\scM$ be a large simplicial model category (enriched, tensored and cotensored over $\sSet_\rmL$).
Let $\scI$ be a large category and consider functors $G \colon \scI \to \sSet_\rmL$ and $F \colon \scI \to \scM$.
The (two-sided) cosimplicial cobar construction associated to these data is the cosimplicial object in $\scM$ whose $l$-th level reads as
\begin{equation}
	C^l(G, \scI, F) = \prod_{i \in (\sfN \scJ)_l} \{G i_0, F i_l \}
	\cong \prod_{i_0, \ldots, i_l \in \scI} \big\{ G i_0 \otimes (\scI(i_0, i_1) \times \cdots \times \scI(i_{l-1}, i_l)),\, F i_l \big\}
	\qquad
	\in \scM\,,
\end{equation}
where $\sfN \scI \in \sSet_\rmL$ denotes the nerve of $\scI$.
The (two-sided) cobar construction $C(G, \scI, F)$ in $\scM$ is the totalisation
\begin{equation}
	C(G, \scI, F)
	= \int_{l \in \bbDelta} \big\{ \Delta^l,\, C^l(G, \scI, F) \big\}\,,
\end{equation}
formed in $\scM$.
If the ambient model category $\scM$ is not clear from context, we will write $C_\scM$ for the cobar construction in $\scM$.

\begin{proposition}
\label{st:cobar computes holim}
\emph{\cite[Cor.~5.1.3]{Riehl:Cat_HoThy}}
The cobar construction is a model for the homotopy limit:
let $\scM$ be a large simplicial model category with fibrant replacement functor $R^\scM$, and let $F \colon \scI \to \scM$ a functor as above.
Then,
\begin{equation}
	\underset{i \in \scI}{\holim}^\scM (Fi)
	\simeq C(*, \scI, R^\scM F)\,.
\end{equation}
In particular, if $Fi \in \scM$ is fibrant for each $i \in \scI$, then
\begin{equation}
	\underset{i \in \scI}{\holim}^\scM (Fi)
	\simeq C(*, \scI, F)\,.
\end{equation}
\end{proposition}

\begin{lemma}
\label{st:cobar and fibrancy}
Let $\scM$ be a large simplicial model category, $\scI$ a large category and $F \colon \scI \to \scM$ a functor.
Suppose that the projective model structure on $\Fun(\scI, \scM)$ exists and that $F$ is projectively fibrant.
Then, the cobar construction $C(*, \scI, F) \in \scM$ is fibrant.
\end{lemma}

\begin{proof}
Recall the notation from Definition~\ref{def:functor (co)tensor}.
By~\cite[Thm.~6.6.1]{Riehl:Cat_HoThy} there is a canonical natural isomorphism
\begin{equation}
	C(*, \scI, D) \cong \int_{i \in \scI} \big\{ N(\scI_{/i}),\, Fi \big\}
	= \{N(\scI_{/(-)}), F\}^\scI\,.
\end{equation}
By Theorem~\ref{st:weighted (co)lim as Quillen functors}, the functor $\{-,-\} \colon \Fun(\scI, \sSet_\rmL)^\opp \times \Fun(\scI,\scM) \to \scM$ is a right Quillen functor in two variables if we endow $\Fun(\scI, \sSet)$ and $\Fun(\scI, \scM)$ with their projective model structures.
The claim then follows from the fact that the functor $N(\scI_{/(-)}) \colon \scI \to \sSet$ is a projectively cofibrant object in $\Fun(\scI, \scM)$~\cite[Cor.~14.8.8]{Hirschhorn:MoCats}.
\end{proof}

\begin{definition}
\label{def:Fun(C^op,M)^loc}
Suppose $\scM$ is left proper, as well as cellular or combinatorial.
Let $\tau$ be a Grothendieck coverage on $\scC$, and $\wtau$ the induced Grothendieck coverage of $\tau$-local epimorphisms on $\wC$.
\begin{myenumerate}
\item Let $S(\scM, \tau)$ be the collection of morphisms in $\Fun(\scC^\opp, \scM)$ of the form
\begin{equation}
	Q (\cC \CU \to \scY_c) \otimes_{\scC^\opp}^{ptw} m\,,
\end{equation}
where $\CU = \{c_i \to c\}_{i \in I}$ is a $\tau$-covering of $c \in \scC$, with associated \v{C}ech nerve $\cC \CU \in \scH_{\infty,0}$, and $m \in \scM$ is a cofibrant object in $\scM$.
Define the left Bousfield localisation
\begin{equation}
	\Fun(\scC^\opp, \scM)^{loc}
	\coloneqq L_{S(\scM, \tau)} \Fun(\scC^\opp, \scM)\,.
\end{equation}

\item Let $S(\scM, \wtau)$ be the collection of morphisms in $\Fun(\wC^\opp, \scM)$ of the form
\begin{equation}
	(\cC \pi \to \wY_X) \otimes_{\scC^\opp}^{ptw} m\,,
\end{equation}
where $\pi \colon Y \to X$ is a $\tau$-local epimorphism in $\wC$, with associated \v{C}ech nerve $\cC \pi \in \wH_{\infty,0}$, and $m \in \scM$ is a cofibrant object in $\scM$.
Define the left Bousfield localisation
\begin{equation}
	\Fun(\wC^\opp, \scM)^{loc}
	\coloneqq L_{S(\scM, \wtau)} \Fun(\wC^\opp, \scM)\,.
\end{equation}
\end{myenumerate}
\end{definition}

These Bousfield localisations exist by~\cite[Thm.~4.4.1]{Hirschhorn:MoCats} and~\cite[Thm.~4.7]{Barwick:Enriched_B-Loc}.
The following is a generalisation of~\cite[Thm.~4.56]{Barwick:Enriched_B-Loc} to the case where the model category $\scM$ is not monoidal:

\begin{proposition}
\label{st:fibrants in Fun(C^op,M)^loc}
In the situation of Definition~\ref{def:Fun(C^op,M)^loc} the following statements hold true:
\begin{myenumerate}
\item An object $F \in \Fun(\scC^\opp, \scM)$ is $S(\scM, \tau)$-local if and only if it is projectively fibrant and, for each $\tau$-covering $\CU = \{c_i \to c\}_{i \in I}$, the induced morphism
\begin{equation}
	F(c) \longrightarrow
	\underset{k \in \bbDelta}{\holim}\, \{ Q (\cC_k \CU), F \}^{\scC^\opp}
\end{equation}
in $\scM$ is a weak equivalence (where $\cC_k \CU$ is viewed as a simplicially constant presheaf on $\scC$).

\item An object $\widehat{G} \in \Fun(\wC^\opp, \scM)$ is $S(\scM, \wtau)$-local if and only if it is projectively fibrant and, for each $\wtau$-covering $\pi \colon Y \to X$ in $\wC$, the induced morphism
\begin{equation}
	\widehat{G}(X) \longrightarrow
	\underset{k \in \bbDelta}{\holim}\, \{ \cC_k \pi, \widehat{G} \}^{\wC^\opp}
\end{equation}
in $\scM$ is a weak equivalence.
\end{myenumerate}
\end{proposition}

\begin{remark}
Suppose that $(\scC, \tau)$ is such that, for each $\tau$-covering $\CU = \{c_i \to c\}_{i \in I}$, each presheaf $C_{i_0 \cdots i_k}$ from~\eqref{eq:def Cech nerve} is again representable.
In that case, the \v{C}ech nerve $\cC \CU$ is projectively cofibrant (it is then objectwise a coproduct of representables), and so we can omit the cofibrant replacement functor $Q$ from Definition~\ref{def:Fun(C^op,M)^loc}(1) and Proposition~\ref{st:fibrants in Fun(C^op,M)^loc}.
Furthermore, the enrichment and the Yoneda Lemma provide a canonical isomorphism
\begin{align}
	\{ \cC_k \CU, F \}^{\scC^\opp}
	= \prod_{i_0, \ldots, i_n \in I} \{ C_{i_0 \cdots i_n}, F \}^{\scC^\opp}
	\cong \prod_{i_0, \ldots, i_n \in I} F(C_{i_0 \cdots i_n})\,,
\end{align}
thus reproducing the more familiar form of the homotopy descent condition.
\qen
\end{remark}

\begin{proof}[Proof of Proposition~\ref{st:fibrants in Fun(C^op,M)^loc}]
We show Claim~(1); the proof of Claim~(2) is analogous.
By definition, an object $F \in \Fun(\scC^\opp, \scM)$ is $S(\scM, \tau)$-local if and only if it is projectively fibrant and, for each cofibrant $m \in \scM$ and each $\tau$-covering $\CU = \{c_i \to c\}_{i \in I}$ in $\scC$, the morphism
\begin{equation}
	\ul{\Fun(\scC^\opp, \scM)} \big( Q (\cC \CU) \otimes_{\scC^\opp}^{ptw} m, F \big)
	\longrightarrow \ul{\Fun(\scC^\opp, \scM)} \big( Q\scY_c \otimes_{\scC^\opp}^{ptw} m, F \big)
\end{equation}
induced by the augmentation $\cC\CU \to \scY_c$ is a weak homotopy equivalence in $\sSet$ (note that the left-hand arguments in both functors above are cofibrant in $\Fun(\scC^\opp, \scM)$, and so the above simplicially enriched hom spaces indeed compute mapping spaces).
By adjointness, this morphism is isomorphic to the morphism
\begin{equation}
	\ul{\scM} \big( m, \{ Q (\cC\CU), F \}^{\scC^\opp} \big)
	\longrightarrow \ul{\scM} \big( m, \{ Q \scY_c, F \}^{\scC^\opp} \big)
	\simeq \ul{\scM} \big( m, F(c) \big)\,.
\end{equation}
Next, we use the isomorphisms in the homotopy category $\Ho \sSet_\rmL$,
\begin{align}
	\ul{\scM} \big( m, \{ Q (\cC\CU), F \}^{\scC^\opp} \big)
	&\cong
	\ul{\scM} \big( m, \{ Q\, \underset{k \in \bbDelta^\opp}{\hocolim}\, \cC_k \CU, F \}^{\scC^\opp} \big)
	\\
	&\cong \ul{\scM} \big( m, \{ \underset{k \in \bbDelta^\opp}{\hocolim}\, Q (\cC_k \CU), F \}^{\scC^\opp} \big)
	\\
	&\cong \ul{\scM} \big( m, \underset{k \in \bbDelta}{\holim}\, \{ Q(\cC_k \CU), F \}^{\scC^\opp} \big)\,.
\end{align}
Here we have used that we can commute $Q$ and the homotopy colimit, since both are bar constructions.
The claim now follows from the fact that a morphism $f \colon a \to b$ in $\scM$ between fibrant objects is a weak equivalence if and only if, for each cofibrant $m \in \scM$, the induced morphism
\begin{equation}
	f^* \colon \ul{\scM}(m, b) \longrightarrow \ul{\scM}(m, a)
\end{equation}
is a weak homotopy equivalence; in other words, we use that the functors $\ul{\scM}(m, -)$, for $m \in \scM$ cofibrant, jointly detect weak equivalences between fibrant objects~\cite[Prop.~9.7.1]{Hirschhorn:MoCats}.
\end{proof}

\begin{definition}
If $(\scD, \tau_\scD)$ is an $\rmL$-small category with a Grothendieck coverage $\tau_\scD$ and $\scM$ a left proper, $\rmL$-small, combinatorial or cellular model category, then we call the fibrant objects of $L_{S(\scM, \tau_\scD)}\Fun(\scD^\opp, \scM)$ the \emph{$\scM$-valued (homotopy) sheaves on $\scD$}.
\end{definition}

\begin{theorem}
\label{st:Re_scM --| S_scM descends to local MoStrs}
Let $\scM$ be a left proper cellular or combinatorial simplicial model category in the universe $\rmL$ and $\scC$ a small category.
Let $\tau$ be a Grothendieck coverage on $\scC$, and $\wtau$ the induced Grothendieck coverage of $\tau$-local epimorphisms on $\wC$.
The Quillen adjunction~\eqref{eq:Re_M -| S_M adjunction} descends to a Quillen adjunction
\begin{equation}
\label{eq:local Re_M -| S_M adjunction}
\begin{tikzcd}
	Re_\scM : \Fun(\wC^\opp, \scM)^{loc} \ar[r, shift left=0.15cm, "\perp"' yshift=0.05cm]
	& \Fun(\scC^\opp, \scM)^{loc} : S_\scM\,. \ar[l, shift left=0.15cm]
\end{tikzcd}
\end{equation}
\end{theorem}

\begin{proof}
By Proposition~\ref{st:QAds and BLocs}, it suffices to show that the right adjoint $S_\scM$ preserves local objects.
To that end, let $F \in \Fun(\scC^\opp, \scM)$ be $S(\scC, \tau)$-local, and let $\pi \colon Y \to X$ be a $\tau$-local epimorphism in $\wC$.
By Proposition~\ref{st:fibrants in Fun(C^op,M)^loc} we thus need to show that the canonical morphism
\begin{equation}
\label{eq:mp for hoDescent of S_M F}
	(S_\scM F)(X) \longrightarrow
	\underset{k \in \bbDelta}{\holim}\, \{ \cC_k \pi, S_\scM F \}^{\wC^\opp}
\end{equation}
is a weak equivalence in $\scM$.
Since $\cC_k \pi$ is representable in $\Fun(\wC^\opp, \sSet_\rmS)$, the Yoneda Lemma for $\wC$ implies that
\begin{equation}
	\{ \cC_k \pi, S_\scM F \}^{\wC^\opp}
	\cong (S_\scM F)(\cC_k \pi)
	= \{Q(\cC_k \pi), F\}^{\scC^\opp}\,.
\end{equation}
The problem is thus equivalent to showing that the canonical morphism
\begin{equation}
	\{QX, F\}^{\scC^\opp}
	\longrightarrow \underset{k \in \bbDelta}{\holim}\, \{Q(\cC_k \pi), F\}^{\scC^\opp}
\end{equation}
is a weak equivalence in $\scM$.
Since the functors $\ul{\scM}(m,-) \colon \scM_{fib} \to \sSet_\rmL$ jointly detect weak equivalences, where $m$ ranges over all cofibrant objects in $\scM$ (and where $\scM_{fib} \subset \scM$ is the full subcategory on the fibrant objects), it suffices to show that, for each cofibrant $m \in \scM$, the morphism
\begin{equation}
	\ul{\scM} \big( m, \{QX, F\}^{\scC^\opp} \big)
	\longrightarrow \ul{\scM} \big( m, \underset{k \in \bbDelta}{\holim}\, \{Q(\cC_k \pi), F\}^{\scC^\opp} \big)
\end{equation}
is a weak homotopy equivalence.
Here we have used that $F$ is projectively fibrant, so that the functor $\{Q(-), F\}^{\scC^\opp}$ takes values in $\scM_{fib}$, and that the homotopy limit of an objectwise fibrant diagram is again a fibrant object.
Using the simplicial enrichment and the two-variable Quillen adjunction from Proposition~\ref{st:2VQAd for Fun(scI, scM)}, the above problem is further equivalent to showing that, for each cofibrant $m \in \scM$, the morphism
\begin{equation}
	\ul{\Fun(\scC^\opp, \sSet_\rmL)} \big( QX, \ul{\scM}_{\scC^\opp}^{ptw}(m, F) \big)
	\longrightarrow \underset{k \in \bbDelta}{\holim}\ \ul{\Fun(\scC^\opp, \sSet_\rmL)} \big( Q(\cC_k \pi), \ul{\scM}_{\scC^\opp}^{ptw}(m, F) \big)
\end{equation}
is a weak homotopy equivalence.
Note that this morphism is the same as the canonical morphism
\begin{equation}
	S_{\infty,0} \big( \ul{\scM}_{\scC^\opp}^{ptw}(m, F) \big) (X)
	\longrightarrow \underset{k \in \bbDelta}{\holim}\ S_{\infty,0} \big( \ul{\scM}_{\scC^\opp}^{ptw}(m, F) \big) (\cC_k \pi)
\end{equation}
in the homotopy descent condition for the simplicial presheaf
\begin{equation}
	S_{\infty,0} \big( \ul{\scM}_{\scC^\opp}^{ptw}(m, F) \big)
	\in \Fun(\wC^\opp, \sSet_\rmL)\,.
\end{equation}
By Theorem~\ref{st:QAd for H_(oo,0)^loc} it now suffices to show that the simplicial presheaf
\begin{equation}
	\ul{\scM}_{\scC^\opp}^{ptw}(m, F)
	\in \Fun(\scC^\opp, \sSet_\rmL)
\end{equation}
is $\tau$-local, for each cofibrant object $m \in \scM$.

To that end, let $\CU = \{c_i \to c\}_{i \in I}$ be a $\tau$-covering in $\scC$.
We have to check that the canonical morphism
\begin{equation}
	\ul{\scM}_{\scC^\opp}^{ptw}(m, F)(c)
	\longrightarrow \underset{k \in \bbDelta}{\holim}\, \ul{\scH}_{\infty,0} \big( Q \cC_k \CU, \ul{\scM}_{\scC^\opp}^{ptw}(m, F) \big)
\end{equation}
is a weak homotopy equivalence.
Again using the two-variable Quillen adjunctions, this is equivalent to checking that
\begin{equation}
	\ul{\scM} \big( m, F(c) \big)
	\longrightarrow \underset{k \in \bbDelta}{\holim}\, \ul{\scM} \big( m, \{ Q \cC_k \CU, F\}^{\scC^\opp} \big)
	\cong \ul{\scM} \big( m, \underset{k \in \bbDelta}{\holim}\, \{ Q \cC_k \CU, F\}^{\scC^\opp} \big)
\end{equation}
is a weak homotopy equivalence, for each cofibrant $m \in \scM$.
Again since the functors $\ul{\scM}(m,-)$ jointly detect weak equivalences between fibrant objects in $\scM$, this is the case if and only if
\begin{equation}
	F(c) \longrightarrow \underset{k \in \bbDelta}{\holim}\, \{ Q \cC_k \CU, F\}^{\scC^\opp}
\end{equation}
is a weak equivalence in $\scM$.
However, that is true by the assumption that $F$ is $S(\scM, \tau)$-local.
\end{proof}

Finally, we record a generalisation of Proposition~\ref{st:S^Q_oo = hoRan_(Y^op)}:

\begin{proposition}
\label{st:S_M is hoRan}
In the setting of Theorem~\ref{st:Re_M --| S_M is Quillen adjunction}, for each projectively fibrant $F \in \Fun(\scC^\opp, \scM)$ and any $X \in \wC$, there is a canonical natural weak equivalence
\begin{equation}
	S_\scM(F)(X) \simeq (\hoRan_\scY F)(X)\,.
\end{equation}
\end{proposition}

\begin{proof}
We have canonical weak equivalences
\begin{align}
	S_\scM(F)(X)
	&= \int_{c \in \scC} \{ QX(c), F(c) \}
	\\
	&= \int_{c \in \scC} \big\{ B(*, \scC_{/X}, \scY)(c), F(c) \}
	\\
	&\cong \int_{c \in \scC} \int_{n \in \bbDelta} \prod_{c_0, \dots, c_n} \big\{ \scC(c, c_0) \otimes \scC(c_0, c_1) \otimes \cdots \otimes \scC(c_{n-1}, c_n) \otimes X(c_n) \otimes \Delta^n,\, F(c) \big\}
	\\
	&\cong \int_{n \in \bbDelta} \prod_{c_0, \dots, c_n} \big\{ \scC(c_0, c_1) \otimes \cdots \otimes \scC(c_{n-1}, c_n) \otimes X(c_n) \otimes \Delta^n,\, F(c_0) \big\}
	\\
	&\cong C_\scM (*, (\scC_{/X})^\opp, F)
	\\*
	&\simeq \underset{\substack{\scY_c \to X \\ \in (\scC_{/X})^\opp}}{\holim} F(c)\,.
\end{align}
In the second-to-last line we have used the enriched cobar construction for the simplicially enriched category $\scM$.
The final weak equivalence follows from Proposition~\ref{st:cobar computes holim}.
\end{proof}

\subsection{Sheaves of $(\infty,n)$-categories}
\label{sec:(oo,n)-sheaves}

The specific instances of Theorem~\ref{st:Re_scM --| S_scM descends to local MoStrs} in which we are interested most arise from choosing the target model category $\scM$ to be a model for $(\infty,n)$-categories.
Various such model categories exist in the literature, including the model structures for $n$-fold complete Segal spaces~\cite{Rezk:HoTheory_of_HoTheories, Barwick:infty-n-Cat_as_closed_MoCat, BSP:Unicity} and $\Theta_n$-spaces~\cite{Rezk:Cartesian_pres_of_oon-Cats, Rezk:Corrigendum_to_Cart_pres} (see also~\cite{BSP:Unicity} for a general characterisation of such models, as well as~\cite{BR:Comparison_for_oonCats_I, BR:Comparison_for_oonCats_II} for a comparison between the aforementioned two models).
Theorem~\ref{st:Re_scM --| S_scM descends to local MoStrs} applies to both homotopy sheaves of $n$-fold complete Segal spaces and $\Theta_n$-spaces.
Since this paper is motivated by applications to functorial field theories, where $n$-fold complete Segal spaces are the prevalent model for $(\infty,n)$-categories, we will focus on the case of $n$-fold complete Segal spaces.

For each $n \in \NN_0$, we consider the category
\begin{equation}
	s^n\sSet_\rmL = \Fun (\bbDelta^{n,\opp}, \sSet_\rmL)
\end{equation}
of $n$-fold simplicial diagrams in large simplicial sets.
For fixed $n \in \NN$, we write $([k_0], \ldots, [k_{n-1}]) \eqqcolon \vec{k} \in \bbDelta^n$.
The category $s^n\sSet_\rmL$ is simplicial, with simplicial mapping space
\begin{equation}
	\ul{s^n\sSet_\rmL}(X,Y) = \int_{\vec{k} \in \bbDelta^n} \ul{\sSet_\rmL}(X_{\vec{k}}, Y_{\vec{k}})\,.
\end{equation}
Further, the category $s^n\sSet_\rmL$ is cartesian closed, with internal hom given by
\begin{equation}
	(Y^X)_{\vec{k}}
	= \big( \ul{s^n\sSet_\rmL}^{s^n\sSet_\rmL} (X, Y) \big)_{\vec{k}}
	= \ul{s^n\sSet_\rmL} (\scY^{\bbDelta^n}_{\vec{k}} \times X, Y)
	\cong \int_{\vec{l} \in \bbDelta^n}\, \ul{\sSet_\rmL} \big( \scY^{\bbDelta^n}_{\vec{k}}(\vec{l}) \times X_{\vec{l}}\,, Y_{\vec{l}} \big)\,,
\end{equation}
where $\scY^{\bbDelta^n} \colon \bbDelta^n \to \sSet_\rmL$ is the Yoneda embedding of $\bbDelta^n$.
For each $n \in \NN$, there is a monoidal adjunction
\begin{equation}
\begin{tikzcd}
	\sfc_\bullet : s^{n-1}\sSet_\rmL \ar[r, shift left=0.15cm, "\perp"' yshift=0.05cm]
	& s^n\sSet_\rmL : \Ev_{[0]}\,, \ar[l, shift left=0.15cm]
\end{tikzcd}
\end{equation}
where $\sfc_\bullet$ is the constant-diagram functor, i.e.~$(\sfc_\bullet X)_{k_0, \ldots, k_{n-1}} = X_{k_1, \ldots, k_{n-1}}$.
Therefore, the category $s^n\sSet_\rmL$ is enriched, tensored, and cotensored over $s^k\sSet$ for any $0 \leq k \leq n$.

Moreover, for each $k,l \in \NN$ there is a functor
\begin{equation}
	(-) \boxtimes (-) \colon s^k\sSet_\rmL \times s^l\sSet_\rmL \longrightarrow s^{k+l}\sSet_\rmL\,,
	\qquad
	(A \boxtimes B)_{i_0, \ldots, i_{k+l-1}}
	= A_{i_0, \ldots, i_{k-1}} \times B_{i_k, \ldots, i_{k+l-1}}\,.
\end{equation}
Iterating this, we have a functor
\begin{equation}
	\prod_{i = 1}^n \sSet_\rmL \longrightarrow s^n\sSet_\rmL\,,
	\qquad
	(A_0, \ldots, A_n) \longmapsto A_0 \boxtimes \cdots \boxtimes A_{n-1}\,,
\end{equation}
which satisfies that $\scY^{\bbDelta^n}_{\vec{k}} = \Delta^{k_0} \boxtimes \cdots \boxtimes \Delta^{k_{n-1}}$ and $\sfc_\bullet X = \Delta^0 \boxtimes X$, for each $X \in s^{n-1}\sSet_\rmL$.

The category $s\sSet_\rmL$ of large bisimplicial sets carries a Reedy model structure, and it is a well-known fact that this coincides with the injective model structure~\cite[Thm.~15.8.7]{Hirschhorn:MoCats}.
It is a less well-known fact that an analogous identity of the injective and Reedy model structures holds true for the categories $s^n\sSet_\rmL = \Fun(\bbDelta^{n, \opp}, \sSet_\rmL)$~\cite[Prop.~3.15]{BR:Reedy_Cats_and_Theta_construction}.
Whenever we refer to $s^n\sSet_\rmL$ as a model category, we will mean the Reedy (or, equivalently, the injective) model structure.
The model category $s^n\sSet_\rmL$ is cellular, left proper, combinatorial and cartesian closed, i.e.~it is a symmetric monoidal model structure in the sense of~\cite[Def.~4.2.6]{Hovey:MoCats} whose monoidal product agrees with the categorical binary product:
this follows readily from the fact that $\sSet_\rmL$ with the Kan-Quillen model structure is cartesian closed~\cite[Prop.~4.2.8]{Hovey:MoCats} and the Reedy model structure on $s^n\sSet_\rmL$ agrees with the injective model structure.

For the reader's convenience, we briefly recall  the presentation of a model category for $n$-fold complete Segal spaces, following~\cite[App.~A]{JFS:Oplax_twQFTs_and_Morita} (see, in particular, the `Lurie-presentation' in loc.~cit.).

For $k \in \NN$, $k \geq 2$, define the \textit{$k$-th spine inclusion}
\begin{equation}
\begin{tikzcd}
	\sigma_k \colon \underbrace{\Delta^1 \underset{\Delta^0}{\sqcup} \cdots \underset{\Delta^0}{\sqcup} \Delta^1}_{\text{$k$-times}}
	\ar[r, hookrightarrow]
	& \Delta^k
\end{tikzcd}
\end{equation}
in $\sSet_\rmL$.
Localising at maps of this type will implement the Segal conditions.

Let $K \in \sSet_\rmL$ denote the pushout
\begin{equation}
\begin{tikzcd}
	\Delta^{\{02\}} \sqcup \Delta^{\{13\}} \ar[r] \ar[d]
	& \Delta^3 \ar[d]
	\\
	\Delta^0 \ar[r]
	& K
\end{tikzcd}
\end{equation}
where the top morphism consists of the inclusions of $\Delta^1$ into $\Delta^3$ as the edge $0 \to 2$ and $1 \to 3$, respectively.
Let $\kappa$ denote the collapse map
\begin{equation}
	\kappa \colon K \to \Delta^0\,.
\end{equation}
Localising at maps of the type $\kappa$ will implement the completeness conditions below.

\begin{definition}
\label{def:CSS_n}
Let $n \in \NN$.
The \emph{model category for (large) $n$-fold complete Segal spaces} is obtained by left Bousfield localisation of $s^n\sSet_\rmL$ at the following morphisms:
\begin{myenumerate}
\item \emph{(Segal conditions)}
For each $k \in \NN$, $i \in \{0, \ldots, n-1\}$ and each $k_0, \ldots, k_{i-2}, k_i, \ldots, k_{n-1} \in \NN_0$, we localise at the morphism
\begin{equation}
	\Delta^{k_0} \boxtimes \cdots \boxtimes \Delta^{k_{i-2}} \boxtimes \sigma_k \boxtimes \Delta^{k_i} \boxtimes \cdots \boxtimes \Delta^{k_{n-1}}\,.
\end{equation}

\item \emph{(Completeness)}
For each $i \in \{0, \ldots, n-1\}$ and each $k_0, \ldots, k_{i-2}, k_i, \ldots, k_{n-1} \in \NN_0$, we localise at the morphism
\begin{equation}
	\Delta^{k_0} \boxtimes \cdots \boxtimes \Delta^{k_{i-2}} \boxtimes \kappa \boxtimes \Delta^{k_i} \boxtimes \cdots \boxtimes \Delta^{k_{n-1}}\,.
\end{equation}

\item \emph{(Essential constancy)}
For each $i \in \{0, \ldots, n-1\}$ and each $k_0, \ldots, k_{i-2}, k_i, \ldots, k_{n-1} \in \NN_0$, we localise at the morphism
\begin{equation}
	\Delta^{k_0} \boxtimes \cdots \boxtimes \Delta^{k_{i-2}} \boxtimes \Delta^0 \boxtimes \Delta^{k_i} \boxtimes \cdots \boxtimes \Delta^{k_{n-1}}
	\longrightarrow 
	\Delta^{k_0} \boxtimes \cdots \boxtimes \Delta^{k_{i-2}} \boxtimes \Delta^0 \boxtimes \Delta^0 \boxtimes \cdots \boxtimes \Delta^0\,.
\end{equation}
\end{myenumerate}
We denote the resulting model category by $\CSS_n$.
Fibrant objects in $\CSS_n$ are called \emph{$n$-fold complete Segal spaces}, or \emph{$(\infty,n)$-categories}.
\end{definition}

\begin{remark}
The model categories $\CSS_n$ exist and are left proper, cellular and simplicial~\cite[Lemma~A.6]{JFS:Oplax_twQFTs_and_Morita}.
In particular, we can apply Theorem~\ref{st:Re_scM --| S_scM descends to local MoStrs} to presheaves with values in $\CSS_n$, for each $n \in \NN$.
\qen
\end{remark}

\begin{lemma}
\label{st:holims in CSS_n}
Let $n \in \NN$, let $\scI$ be a large category, and let $D \colon \scI \to \CSS_n$ be a projectively fibrant diagram, i.e.~$Di$ is an $n$-fold complete Segal space for every $i \in \scJ$.
\begin{myenumerate}
\item The homotopy limit of $D$ can be computed levelwise: for each $k \in \NN_0$, there exist canonical natural isomorphisms
\begin{equation}
	\big( \holim^{\CSS_n}_\scJ(D) \big)_k \cong \holim^{\CSS_{n-1}}_\scJ(D_k)\,,
\end{equation}
where for $i \in \scJ$ the object $D_k i = (Di)_k$ is the $k$-th simplicial level of $Di \in s(s^{n-1}\sSet_\rmL)$.

\item The homotopy limit of $D$ can be computed objectwise when viewing the diagram as a functor $D \colon \scJ \times \bbDelta^n \to \sSet_\rmL$.
That is, for any $\vec{k} \in \bbDelta^n$ there is an isomorphism (natural in $\vec{k} \in \bbDelta^n$)
\begin{equation}
	\big( \holim^{\CSS_n}_\scJ(D) \big)_{\vec{k}}
	\cong \holim^\sSet_\scJ(D_{\vec{k}})\,.
\end{equation}
\end{myenumerate}
\end{lemma}

\begin{proof}
We use the explicit model for homotopy limits in simplicially enriched categories from Proposition~\ref{st:cobar computes holim}.
For the sake of brevity, for $l \in \NN_0$, we set
\begin{equation}
	\scY^n = \scY^{\bbDelta^n}_{(-,0,\ldots,0)} \colon \bbDelta \to s^n\sSet_\rmL\,,
	\qquad
	\scY^n_{[k]} = \Delta^k \boxtimes \underbrace{\Delta^0 \boxtimes \cdots \boxtimes \Delta^0}_{\text{$(n{-}1)$-times}}\,.
\end{equation}
Since $D$ is objectwise fibrant, a model for the homotopy limit of $D$ is given by the two-sided cobar construction $C(*,\scI,D) \in \CSS_n$~\cite[Cor.~5.1.3]{Riehl:Cat_HoThy}.

For any $n \in \NN_0$, there are canonical natural isomorphisms of simplicial sets
\begin{align}
	\big( \holim^{\CSS_n}_\scJ(D) \big)_m
	&= \big( C_{s^n\sSet}(*,\scJ,D) \big)_m
	\\*
	&= \ul{\CSS}_n^{\CSS_{n-1}} \big( \scY^n_{[m]}, C_{s^n\sSet}(*,\scJ,D) \big)
	\\
	&\cong \int_{k \in \bbDelta} \ul{s^n\sSet}^{s^{n-1}\sSet} \Big( \scY^n_{[m]}, \big( C^k_{s^n\sSet}(*,\scJ,D) \big)^{\sfc_\bullet^n \Delta^k} \Big)
	\\
	&\cong \int_{k \in \bbDelta} \ul{s^n\sSet}^{s^{n-1}\sSet} \big( \scY^n_{[m]} \otimes \sfc_\bullet^n \Delta^k, C^k_{s^n\sSet}(*,\scJ,D) \big)
	\\
	&\cong \int_{k \in \bbDelta} \ul{s^{n-1}\sSet}^{s^{n-1}\sSet} \big( \sfc_\bullet^{n-1} \Delta^k, \big( C^k_{s^n\sSet}(*,\scJ,D) \big)_{(m,-,\cdots,-)} \big)
	\\
	&\cong \int_{k \in \bbDelta} \ul{s^{n-1}\sSet}^{s^{n-1}\sSet} \Big( \sfc_\bullet^{n-1} \Delta^k, \prod_{i \in (\sfN\scJ)_k} (Di_k)_{(m,-,\cdots,-)} \Big)
	\\
	&\cong \int_{k \in \bbDelta} \Big( \prod_{i \in (\sfN\scJ)_k} (Di_k)_m \Big)^{\sfc_\bullet^{n-1} \Delta^k}
	\\
	&= C_{s^{n-1}\sSet}(*,\scJ,D_m)
	\\*
	&= \holim^{\CSS_{n-1}}_\scJ(D_m)\,.
\end{align}
The last identity relies on~\cite[Cor.~5.1.3]{Riehl:Cat_HoThy} (see also Proposition~\ref{st:cobar computes holim}), together with the following facts:
the functor $D \colon \scI \to \CSS_n$ is objectwise fibrant.
In particular, for each $i \in \scJ$, the object $D(i) \in \Fun(\bbDelta^\opp, s^{n-1}\sSet_\rmL)$ is fibrant in the Reedy model structure (equivalently, the injective model structure), where $s^{n-1}\sSet_\rmL$ carries the Reedy (equivalently, injective) model structure.
Forgetting the localisations in Definition~\ref{def:CSS_n} at morphisms which are non-trivial in the first $\bbDelta$-factor, we further obtain that $(Di)_m$ is fibrant in $\CSS_{n-1}$, for each $i \in \scI$ and $m \in \bbDelta$.
Consequently, the functor $D_m \colon \bbDelta \to \CSS_{n-1}$, $i \mapsto D_m(i) = (Di)_m$ is also fibrant in the \textit{projective} model structure on $\Fun(\bbDelta^\opp, \CSS_{n-1})$.
Thus, the cobar construction in the second-to-last line indeed models its homotopy limit.
Part~(2) follows in analogy with part~(1), replacing $\scY^n$ with $\scY$ (i.e.~replacing $\Delta^k \boxtimes \Delta^0 \boxtimes \cdots \boxtimes \Delta^0$ by $\Delta^{k_0} \boxtimes \cdots \boxtimes \Delta^{k_{n-1}}$) and using that injectively fibrant diagrams are, in particular, projectively fibrant.
\end{proof}

\begin{definition}
\label{def:(oo,n)-(pre)sheaves}
Let $\scC$ be a small category.
\begin{myenumerate}
\item We write
\begin{equation}
	\scH_{\infty,n} \coloneqq \Fun(\scC^\opp, \CSS_n)
\end{equation}
and view this as endowed with the projective model structure.
A fibrant object in $\scH_{\infty,n}$ is called a \emph{presheaf of $(\infty,n)$-categories on $\scC$}.

\item In the notation of Definition~\ref{def:Fun(C^op,M)^loc}, we define the left Bousfield localisation
\begin{equation}
	\scH_{\infty,n}^{loc} \coloneqq L_{S(\CSS_n,\tau)} \scH_{\infty,n}\,.
\end{equation}
A fibrant object in $\scH_{\infty,n}^{loc}$ is called a \emph{sheaf of $(\infty,n)$-categories on $(\scC,\tau)$}.
\end{myenumerate}
\end{definition}

\begin{proposition}
Let $\scC$ be a small category.
Then $\scH_{\infty,n}$ is a left proper, tractable (large) simplicial model category.
Further, if $\scC$ is endowed with a Grothendieck pretopology $\tau$, then $\scH_{\infty,n}^{loc}$ is a $\CSS_n$-model category, and it is symmetric monoidal if $\scH_{\infty,n}$ is so.
\end{proposition}

\begin{proof}
The statement about $\scH_{\infty,n}$ is a consequence of~\cite[Prop.~4.50]{Barwick:Enriched_B-Loc}.
For $\scH_{\infty,n}^{loc}$, the claim follows from~\cite[Prop.~4.47]{Barwick:Enriched_B-Loc}.
The statement about symmetric monoidal structures is an application of~\cite[Thm.~4.58]{Barwick:Enriched_B-Loc}.
\end{proof}

Analogously to Sections~\ref{sec:Sh_oo}, $\scH_{\infty,n}$ is symmetric monoidal whenever $\scC$ has finite products.

\begin{definition}
\label{def:(oo,n)-(pre)sheaves on hat(C)}
Let $\scC$ be a small category.
As before, we write $\wC = \Fun(\scC^\opp, \Set_\rmS)$.
\begin{myenumerate}
\item We write
\begin{equation}
	\wH_{\infty,n} \coloneqq \Fun(\wC^\opp, \CSS_n)
\end{equation}
for the category $\Fun(\wC^\opp, \CSS_n)$, endowed with the projective model structure.
A fibrant object in $\wH_{\infty,n}$ is called a \emph{presheaf of $(\infty,n)$-categories on $\wC$}.

\item In the notation of Definition~\ref{def:Fun(C^op,M)^loc}, we define the left Bousfield localisation
\begin{equation}
	\wH_{\infty,n}^{loc} \coloneqq L_{S(\CSS_n,\wtau)} \wH_{\infty,n}\,.
\end{equation}
A fibrant object in $\wH_{\infty,n}^{loc}$ is called a \emph{sheaf of $(\infty,n)$-categories on $(\wC,\wtau)$}.
\end{myenumerate}
\end{definition}

\begin{remark}
Lemma~\ref{st:holims in CSS_n} implies that an object $\scF \in \Fun(\scC^\opp, \CSS_n)$ is $S(\CSS_n,\tau)$-local if and only if, for each $[\vec{k}] \in \bbDelta^n$ the object $\scF_{k_0, \ldots, k_{n-1}} \in \Fun(\scC^\opp, \sSet_\rmL)$ is $\tau$-local, and analogously for $(\wC, \wtau)$ in place of $(\scC, \tau)$.
\qen
\end{remark}

In the notation of~\eqref{eq:Re_M -| S_M adjunction}, we further write
\begin{equation}
	S_{\infty,n} \coloneqq S_{\CSS_n} \colon \scH_{\infty,n} \longrightarrow \wH_{\infty,n}\,,
	\qquad
	(S_{\infty,n} \scF)(X) = \{ QX, \scF \}^{\scC^\opp}\,.
\end{equation}
We can compute this functor more explicitly as follows:
for $[\vec{k}] \in \bbDelta^n$, we have canonical isomorphisms
\begin{align}
\label{eq:S_(oo,n) levelwise}
	(S_{\infty,n} \scF)_{\vec{k}}(X)
	&= \big( \{ QX, \scF \}^{\scC^\opp} \big)_{\vec{k}}
	\\
	&\cong s^n\sSet_\rmL \big( \Delta^{k_0} \boxtimes \cdots \boxtimes \Delta^{k_{n-1}}, \{ QX, \scF \}^{\scC^\opp} \big)
	\\
	&\cong \Fun(\scC^\opp, s^n\sSet_\rmL) \big( QX \otimes (\Delta^{k_0} \boxtimes \cdots \boxtimes \Delta^{k_{n-1}}), \scF \big)
	\\
	&\cong \Fun(\scC^\opp, \sSet_\rmL) (QX, \scF_{k_0, \ldots, k_{n-1}})
	\\
	&\cong S_{\infty,0}(\scF_{k_0, \ldots, k_{n-1}})(X)\,.
\end{align}

We have the following direct corollary of Proposition~\ref{st:S_M is hoRan}:

\begin{proposition}
\label{st:S^Q_(oo,n) = hoRan_(Y^op)}
On any fibrant object $\scF \in \scH_{\infty,n}$, the functor $S_{\infty,n}$ agrees with the homotopy right Kan extension of $\scF \colon \scC^\opp \to \CSS_n$ along the Yoneda embedding,
\begin{equation}
	S_{\infty,n}\scF \cong \hoRan_\scY(\scF)\,.
\end{equation}
\end{proposition}

As a direct application of Proposition~\ref{st:Re_M --| S_M is Quillen adjunction}, we have:

\begin{proposition}
\label{st:Re^Q_(oo,n) -| S^Q_(oo,n) is QAd for proj MoStrs}
The functor $S_{\infty,n}$ is the right adjoint in a Quillen adjunction
\begin{equation}
\begin{tikzcd}[column sep=1.cm]
	Re_{\infty,n} : \wH_{\infty,n} \ar[r, shift left=0.15cm, "\perp"' yshift=0.03cm]
	& \scH_{\infty,n} \ar[l, shift left=0.15cm] : S_{\infty,n}\,.
\end{tikzcd}
\end{equation}
\end{proposition}

Moreover, Theorem~\ref{st:Re_scM --| S_scM descends to local MoStrs} implies the following result:

\begin{corollary}
\label{st:QAd for H_(oo,n)^loc}
Let $(\scC,\tau)$ be a small site.
For any $n \in \NN$, the Quillen adjunction $Re_{\infty,n} \dashv S_{\infty,n}$ descends to a Quillen adjunction
\begin{equation}
\begin{tikzcd}[column sep=1.cm]
	Re_{\infty,n} : \wH_{\infty,n}^{loc} \ar[r, shift left=0.13cm, "\perp"' yshift=0.05cm]
	& \scH_{\infty,n}^{loc} \ar[l, shift left=0.13cm] : S_{\infty,n}\,.
\end{tikzcd}
\end{equation}
\end{corollary}

\begin{remark}
\label{rmk:Shs of Cats}
The descent results in this subsection also imply corresponding results for strict presheaves of 1-categories, i.e.~functors
\begin{equation}
	F \colon \scC^\opp \to \Cat_\rmL\,,
\end{equation}
from $\scC^\opp$ to the 1-category of large categories.
Rezk's classifying diagram functor $N_{rel} \colon \Cat_\rmL \to \CSS_1$ from large categories to large complete Segal spaces is a fully faithful embedding of categories which preserves and detects weak equivalences (weak equivalences in $\Cat_\rmL$ being equivalences of categories) and which preserves exponentials; this is the content of~\cite[Thm.~3.7]{Rezk:HoTheory_of_HoTheories}.
Explicitly, if $\scD^\sim$ denotes the maximal subgroupoid of a category $\scD$, we have, for each $n \in \NN_0$,
\begin{equation}
	(N_{rel} \scD)_n = N \big( (\scD^{[n]})^\sim \big)\,.
\end{equation}

Now let $(\scC, \tau)$ be a site.
One checks that, given a presheaf $F \colon \scC^\opp \to \Cat_\rmL$ and a $\tau$-covering family $\CU = \{c_i \to c\}_{i \in \Lambda}$ such that all $C_{i_0 \ldots i_n}$ are representable (see~\eqref{eq:def Cech nerve} for the notation), there is an equivalence
\begin{equation}
	\ul{\scH}_{\infty,1}(Q \cC \CU, N_{rel} F)
	\simeq \ul{\scH}_{\infty,1}(\cC \CU, N_{rel} F)
	\simeq N_{rel} \Desc(F; \CU)\,,
\end{equation}
where $\Desc(F; \CU)$ is the usual descent category associated to $F$ and the covering family $\CU$ (see, for instance,~\cite[Def.~2.5]{Moerdijk:Intro_to_Stacks} or~\cite[Def.~2.9]{Bunk:Thesis}).
It follows that $F$ is a sheaf of categories precisely if $N_{rel} F$ is fibrant in $\scH_{\infty,1}^{loc}$.
Given a small presheaf $X \in \wC$, one can check that
\begin{equation}
	\ul{\scH}_{\infty,1}(QX, N_{rel} F)
	\cong N_{rel} \Big( \holim^{\Cat_\rmL} \big( (\scC_{/X})^\opp \longrightarrow \scC^\opp \overset{F}{\longrightarrow} \Cat_\rmL \big) \Big)\,,
\end{equation}
in $\CSS_1$, where the homotopy limit is computed in the canonical model structure on $\Cat_\rmL$ (see, for instance,~\cite{Rezk:Canonical_MoStr_for_Cats}).
We define a functor
\begin{align}
	&S_{1,1} \colon \Fun(\scC^\opp, \Cat_\rmL)
	\longrightarrow \Fun(\wC^\opp, \Cat_\rmL)\,,
	\\
	&S_{1,1}(F)(X) \coloneqq \holim^{\Cat_\rmL} \big( (\scC_{/X})^\opp \longrightarrow \scC^\opp \overset{F}{\longrightarrow} \Cat_\rmL \big)\,.
\end{align}
Given a $\tau$-local epimorphism $\pi \colon Y \to X$, we thus have a commutative diagram
\begin{equation}
\begin{tikzcd}[column sep=2cm, row sep=1cm]
	(N_{rel}F)(X) \ar[r, "\pi^*"]
	&\underset{k \in \bbDelta}{\holim}^{\CSS_1} \big( S_{\infty,1}(F)(\cC_k\pi) \big)
	\\
	N_{rel} \big( F(X) \big) \ar[r, "N_{rel}\pi^*"'] \ar[u, equal]
	& N_{rel} \Big( \underset{k \in \bbDelta}{\holim}^{\Cat_\rmL} \big( S_{1,1}(X)(\cC_k\pi) \big) \Big) \ar[u, "\simeq"]
\end{tikzcd}
\end{equation}
Since $N_{rel}$ reflects weak equivalences, it now follows from Corollary~\ref{st:QAd for H_(oo,n)^loc} that $S_{1,1}F$ satisfies $\wtau$-descent whenever $F$ satisfies $\tau$-descent.
\qen
\end{remark}

\section{Applications in geometry and field theory}
\label{sec:applications}

Here we apply the results of Section~\ref{sec:higher sheaves} to diffeological vector bundles, gerbes with connection, 2-vector bundles and smooth field theories on background manifolds.

\subsection{Strictification of diffeological vector bundles and descent along subductions}
\label{sec:Dfg VBuns}

Vector bundles form a sheaf of categories on the site of manifolds and surjective submersions.
One can show this either by providing an explicit descent construction, or by using a strictification of vector bundles which describes them in terms of their transition functions.

Diffeological spaces form a popular generalisation of manifolds; their category $\Dfg$ sits in a sequence of canonical fully faithful inclusion functors
\begin{equation}
\begin{tikzcd}
	\Mfd \ar[r, hookrightarrow]
	& \Dfg \ar[r, hookrightarrow]
	& \Sh(\Cart, \tau_{dgop}) \ar[r, hookrightarrow]
	& \PSh(\Cart)\,.
\end{tikzcd}
\end{equation}
Apart from the first inclusion, these functors each have a left adjoint.
Diffeological spaces have been introduced in~\cite{IZ:Diffeology} and are increasingly used to carry out smooth constructions in geometry and topology (see, for instance,~\cite{Kihara:Smooth_ho_of_oo-dim_Mfds, Minichiello:Dfg_prBuns_and_Pr-oo-Buns}).
Despite this, there is currently no descent theorem for diffeological vector bundles.
That is a large gap in the literature, which we close here:
as the main step, we extend the strictification of vector bundles on manifolds to diffeological vector bundles.
The desired descent property of diffeological vector bundles then follows readily from Corollary~\ref{st:QAd for H_(oo,n)^loc} and Remark~\ref{rmk:Shs of Cats}.

\subsubsection{Background on diffeological spaces}

We start by recalling the category of diffeological spaces, mostly following~\cite{BH:Convenient_Cats_of_Smooth_Spaces}.
As a slight variation to the standard literature, we define diffeological spaces as concrete sheaves on a site which differs from the usual choice, but this is just for technical convenience (cf.~Remark~\ref{rmk:Op Cart ulM}).
Throughout this section, we let $\scC$ be a small category.

\begin{definition}[{\cite[Def.~17]{BH:Convenient_Cats_of_Smooth_Spaces}}]
A site $(\scC,\tau)$ is called \emph{subcanonical} if every representable presheaf on $\scC$ is a sheaf.
\end{definition}

\begin{definition}[{\cite[Def.~18]{BH:Convenient_Cats_of_Smooth_Spaces}}]
\label{def:concrete site}
A site $(\scC,\tau)$ is \emph{concrete} if $\scC$ has a terminal object $* \in \scC$, $(\scC,\tau)$ is subcanonical and satisfies the following two additional conditions:
\begin{myenumerate}

\item The functor $\Ev_* = \scC(*,-) \colon \scC \to \Set_\rmS$ is faithful.

\item Each covering $\{f_i \colon c_i \to c\}_{i \in I} \in \tau(c)$ is \emph{jointly surjective}: we have
$
\bigcup_{i \in I} (f_i)_* \big( \scC(*,c_i) \big) = \scC(*,c)\,.
$
\end{myenumerate}
\end{definition}

\begin{definition}[{\cite[Def.~19, Def.~46]{BH:Convenient_Cats_of_Smooth_Spaces}}]
\label{def:concrete sheaf}
A presheaf $X$ on a concrete site $(\scC,\tau)$ is called \emph{concrete} if for every object $c \in \scC$ the map
\begin{equation}
	\Ev_* \colon X(c) \cong \wC(\scY_c,X) \longrightarrow \Set_\rmS \big( \scY_c(*), X(*) \big)\,,
	\qquad
	\varphi \longmapsto \varphi_{|*}
\end{equation}
is injective.
In this case, we define the subset $\Plot_X(c) \coloneqq \Ev_*(X(c)) \subset \Set_\rmS \big( \scY_c(*), X(*) \big)$ of \emph{plots of $X$ over $c \in \scC$}.
\end{definition}

In other words, a concrete presheaf $X$ has an \textit{underlying set} $X(*)$, and, for each $c \in \scC$, the set $X(c)$ can canonically be described as a subset of the maps of sets $\scY_c(*) \to X(*)$.

\begin{definition}[{\cite[Def.~20]{BH:Convenient_Cats_of_Smooth_Spaces}}]
\label{def:Dfg}
Let $(\scC,\tau)$ be a concrete site.
The \emph{category $\Dfg(\scC,\tau)$ of (small) $(\scC,\tau)$-spaces} is the full subcategory of $\wC = \Fun(\scC^\opp, \Set_\rmS)$ whose objects are the concrete small sheaves on $(\scC, \tau)$.
For $X \in \Dfg(\scC,\tau)$, the elements of $X(c)$ are called \emph{plots of $X$ over $c$}.
We refer to $(\Cart,\tau_{dgop})$-spaces as \emph{diffeological spaces}, and we abbreviate $\Dfg \coloneqq \Dfg(\Cart,\tau_{dgop})$.
\end{definition}

We commit an abuse of notation and denote the underlying set of a diffeological space $X$ also by $X$; it will be clear from context whether we are referring only to the underlying set, or to the underlying set endowed with the structure of a diffeological space.

\begin{example}
The site $(\Cart,\tau_{dgop})$ from Example~\ref{eg:Cart as site} is concrete:
first, note that $\Cart$ has a terminal object $* = \RN^0$.
Further, the functor $\Cart(\RN^0, -) \colon \Cart \to \Set_\rmS$ takes the underlying set of $c \in \Cart$; thus, it is faithful.
Since smooth maps to a target manifold form a sheaf with respect to open coverings, $(\Cart,\tau_{dgop})$ is subcanonical.
Finally, the fact that covering families in $(\Cart,\tau_{dgop})$ are jointly surjective is immediate from the definition of $\tau_{dgop}$.
Analogously, one checks that the site $(\Op,\tau_{op})$ from Example~\ref{eg:Op} is concrete.
\qen
\end{example}

The relation between Sections~\ref{sec:higher sheaves} and Section~\ref{sec:applications} relies on the following observation.

\begin{remark}
\label{rmk:subductions}
The category $\Dfg(\scC,\tau)$, together with the collection of those $\tau$-local epimorphisms whose source and target are $(\scC,\tau)$-spaces, form a site which is contained in the site $(\wC, \wtau)$ as a full subcategory.
The $\tau$-local epimorphisms between $(\scC,\tau)$-spaces are also called \emph{subductions}~\cite{IZ:Diffeology}.
We denote the site of $(\scC,\tau)$-spaces with this coverage by $(\Dfg(\scC,\tau), \tau_{subd})$.
\qen
\end{remark}

\begin{remark}
\label{rmk:Op Cart ulM}
In~\cite{IZ:Diffeology}, diffeological spaces are defined as concrete presheaves on $(\Op,\tau_{op})$, whereas here we define them on $(\Cart, \tau_{dgop})$.
However, the two categories are equivalent because the canonical inclusion of $(\Cart, \tau_{dgop})$ into $(\Op,\tau_{op})$ is an inclusion of a dense subsite.
For any manifold $M$, the presheaf $\ul{M}$ from Example~\ref{eg:Cart as site} is a diffeological space.
\qen
\end{remark}

\begin{proposition}[{\cite[Section~5.3]{BH:Convenient_Cats_of_Smooth_Spaces}}]
\label{st:Psh-Dfg adjunction}
Let $(\scC,\tau)$ be a concrete site.
There exists an adjunction
\begin{equation}
\begin{tikzcd}
	\mathsf{Dfg} : \wC \ar[r, shift left=0.15cm, "\perp"' yshift=0.03cm] & \Dfg(\scC,\tau) : \iota \ar[l, shift left=0.15cm]
\end{tikzcd}
\end{equation}
whose right adjoint is a fully faithful inclusion, i.e.~$\Dfg(\scC,\tau) \subset \wC$ is a reflective localisation.
\end{proposition}

\subsubsection{Strictification and descent of diffeological vector bundles}

We now specialise to the case $(\scC,\tau) = (\Cart, \tau_{dgop})$ and to $\Dfg = \Dfg(\Cart, \tau_{dgop})$.
We will denote the underlying set $\scY_c(*)$ of a cartesian space $c \in \Cart$ again by $c$.

\begin{definition}
\label{def:Dfg VBun}
Let $X \in \Dfg$ be a diffeological space.
\begin{myenumerate}
\item A \emph{(complex rank-$k$) vector bundle on $X$} is a pair $(E,\pi)$ of a diffeological space $E$ and a morphism $\pi \colon E \to X$ in $\Dfg$ with the structure of a $\CN$-vector space on each fibre $E_{|x} \coloneqq \pi^{-1}(\{x\})$, for $x \in X$, satisfying the following condition:
for each plot $\varphi \colon c \to X$ there exists an isomorphism of diffeological spaces
\begin{equation}
	\Phi \colon c \times \CN^k \longrightarrow c \underset{X}{\times} E
\end{equation}
such that $\pr_c \circ \Phi = \pr_c$ and such that for every $y \in c$ the restriction $\Phi_{|y} \colon \CN^k \to E_{|\varphi(y)}$ is linear.

\item Let $(E,\pi)$ and $(E',\pi')$ be vector bundles on $X$.
A \emph{morphism $(E,\pi) \to (E', \pi')$} is a morphism $\Psi \in \Dfg(E,E')$ with $\pi' \circ \Psi = \pi$ and such that $\Psi_{|x} \colon E_{|x} \to E'_{|x}$ is a linear map for every $x \in X$.
\end{myenumerate}
This defines a category $\VBun_\Dfg(X)$ of diffeological vector bundles on $X$.
\end{definition}

\begin{example}
If $M$ is a smooth manifold, which we view as a diffeological space, then the category $\VBun_\Dfg(X)$ is canonically equivalent to the ordinary category of vector bundles on $M$.
\qen
\end{example}

Any morphism $f \in \Dfg(X',X)$ gives rise to a pullback functor
\begin{equation}
	f^* \colon \VBun_\Dfg(X) \longrightarrow \VBun_\Dfg(X')\,,
	\qquad
	(E,\pi) \longmapsto \big( X' \underset{X}{\times} E, \pi' \big)\,,
\end{equation}
where $\pi'$ is the pullback of $\pi$ along $f$.

\begin{definition}
We let $\VBun_\Dfg \colon \Dfg^\opp \to \Cat_\rmL$ denote the pseudo-functor resulting from the above assignments (where $\Cat_\rmL$ is the 2-category of large categories).
\end{definition}

The goal of this subsection is to prove the following theorem:

\begin{theorem}
\label{st:Dfg VBuns descend along subductions}
The pseudo-functor $\VBun_\Dfg \colon \Dfg^\opp \longrightarrow \Cat_\rmL$ satisfies descent along subductions of diffeological spaces (see Remark~\ref{rmk:subductions}).
\end{theorem}

\begin{remark}
We use the following strategy for the proof of Theorem~\ref{st:Dfg VBuns descend along subductions}:
\begin{myenumerate}
\item Replace the pseudo-functor $\VBun_\Dfg$ by a strictification, i.e.~find a (strict) functor
\begin{equation}
	\VBun_{str} \colon \Dfg^\opp \to \Cat_\rmL
\end{equation}
and an equivalence of pseudo-functors between $\VBun_{str}$ and $\VBun_\Dfg$.

\item Recall that \smash{$(\Dfg, \tau_{subd}) \subset (\wCart, \wtau)$} is a subsite (see Remark~\ref{rmk:subductions}).
Show that $\VBun_{str}$ is in fact a restriction along the inclusion $\iota \colon \Dfg \hookrightarrow \wCart$ of a functor \smash{$\VBun_{str} \colon \wCart{}^\opp \to \Cat_\rmL$}.

\item Show that the functor \smash{$\VBun_{str} \colon \wCart{}^\opp \to \Cat_\rmL$} satisfies descent with respect to $\wtau_{dgop}$.
\end{myenumerate}
Theorem~\ref{st:Dfg VBuns descend along subductions} then readily follows from the equivalences in points~(2) and~(3).

We will achieve the above goals by means of Remark~\ref{rmk:Shs of Cats}:
we give a strict functor
\begin{equation}
	\VBun_\triv \colon \Cart^\opp \to \Cat_\rmL
\end{equation}
satisfying the following two properties:
\begin{myitemize}
\item The functor $\VBun_\triv \colon \Cart^\opp \to \Cat_\rmL$ satisfies $\tau_{dgop}$-descent.

\item Consider the functor $\VBun_{str} \coloneqq S_{1,1}\VBun_\triv \in \Fun(\wCart{}^\opp, \Cat_\rmL)$.
Its restriction to $\Dfg \subset \wCart$ is a strictification of $\VBun_\Dfg$ (thus yielding point~(1) above).
\end{myitemize}
Remark~\ref{rmk:Shs of Cats} then ensures that points~(2) and~(3) above are satisfied.
\qen
\end{remark}

We now carry out the above strategy:

\begin{definition}
We define a strict presheaf $\VBun_\triv$ of categories on $\Cart$ as follows:
for $c \in \Cart$, the category $\VBun_\triv(c)$ has objects $(c,n)$, where $n \in \NN_0$, and a morphism $(c,n) \to (c,m)$ is a smooth function $\psi \colon c \to \Mat( m{\times}n, \CN)$ from $c$ to the vector space of complex $m$-by-$n$ matrices; to make clear that $\psi$ is defined over $c$ we also write $(c,\psi)$ instead of just $\psi$.
A smooth map $f \colon c' \to c$ acts via
\begin{equation}
	f^* \colon \VBun_\triv(c) \longrightarrow \VBun_\triv(c')\,,
	\qquad
	(c,n) \mapsto (c',n)\,,
	\quad
	(c,\psi) \mapsto (c', \psi \circ f)\,.
\end{equation}
Using the functor $S_{1,1} \colon \Fun(\scC^\opp, \Cat_\rmL) \longrightarrow \Fun(\wC^\opp, \Cat_\rmL)$ from Remark~\ref{rmk:Shs of Cats} we also define the strict functor
\begin{equation}
	\VBun_{str} \coloneqq (S_{1,1} \VBun_\triv) \
	\in \Fun(\Cart^\opp, \Cat_\rmL)\,.
\end{equation}
\end{definition}

Let $\scH_{\infty,1}$ denote the model category for presheaves of $\infty$-categories on the site $(\Cart, \tau_{dgop})$ (as in Definition~\ref{def:(oo,n)-(pre)sheaves}).
As in Remark~\ref{rmk:Shs of Cats}, let $N_{rel} \colon \Cat_\rmL \to s\sSet_\rmL$ denote Rezk's classifying diagram functor~\cite[Sec.~3]{Rezk:HoTheory_of_HoTheories}.

\begin{proposition}
\label{st:VBun_triv and VBun are local}
Both $N_{rel} \circ \VBun_\triv$ and $N_{rel} \circ \VBun_{str}$ are fibrant objects in $\scH_{\infty,1}^{loc}$ and in $\wH_{\infty,1}^{loc}$, respectively.
Equivalently (by Remark~\ref{rmk:Shs of Cats}), $\VBun_\triv$ satisfies $\tau_{dgop}$-descent, and $\VBun_{str}$ satisfies $\wtau_{dgop}$-descent.
\end{proposition}

\begin{proof}
We only need to show that \smash{$\VBun_\triv \in \scH_{\infty,1}^{loc}$} is fibrant.
That $\VBun_{str}$ is fibrant in \smash{$\wH_{\infty,1}^{loc}$} will then follow from Corollary~\ref{st:QAd for H_(oo,n)^loc}.
First, $\VBun_\triv$ is fibrant in $\scH_{\infty,1}$:
for each $c \in \Cart$, the bisimplicial set $\VBun_\triv(c) = N_{rel}(\VBun_\triv(c))$ is the classification diagram of an ordinary category and hence a complete Segal space by~\cite[Prop.~6.1]{Rezk:HoTheory_of_HoTheories}.
In order to show that $\VBun_\triv$ is also fibrant in \smash{$\scH_{\infty,1}^{loc}$}, it suffices to show that the presheaf of categories $\VBun_\triv$ satisfies descent with respect to $\tau_{dgop}$, which is standard:
for instance, one can see this by using that vector bundles glue along isomorphisms over open coverings of manifolds.
Thus, given descent data for $\VBun_\triv$ with respect to a differentiably good open covering of a cartesian space $c$, we obtain a vector bundle on $c$.
This vector bundle is, in general, non-trivial, but since $c \cong \RN^n$ for some $n \in \NN_0$, it is isomorphic to a trivial vector bundle.
This provides the essential surjectivity of the descent functor for vector bundles.
Its full faithfulness follows since morphisms of trivial vector bundles are the same as smooth matrix-valued functions, which form sheaves of sets with respect to open coverings of manifolds.
\end{proof}

For the proof of Theorem~\ref{st:Dfg VBuns descend along subductions} we are thus left to show that $\VBun_{str} \in \Fun(\wCart{}^\opp, \Cat_\rmL)$ restricts to a strictification of the pseudo-functor $\VBun_\Dfg$ on the full subcategory \smash{$\Dfg \subset \wCart$}.
To that end, we first identify a more concrete description of the functor $\iota^* \VBun_{str} = \iota^* S_{1,1} (\VBun_\triv)$, where \smash{$\iota \colon \Dfg \hookrightarrow \wCart$} denotes the canonical inclusion:
it is the homotopy right Kan extension of $\VBun_\triv$ along the inclusion $\Cart \hookrightarrow \Dfg$:

\begin{lemma}
\label{st:VBun_str as hoRan}
We can compute the values of $\iota^* \VBun_{str} \colon \Dfg^\opp \to \Cat_\rmL$ as
\begin{equation}
\label{eq:iota^* VBun_str via holims}
	(\iota^* \VBun_{str})(X) =
	\begin{tikzcd}[column sep=1.5cm]
		\holim^{\Cat_\rmL} \big( (\Cart_{/X})^\opp \ar[r, "\pr"]
		& \Cart^\opp \ar[r, "\VBun_\triv"]
		& \Cat_\rmL \big)\,,
	\end{tikzcd}
\end{equation}
\end{lemma}

\begin{proof}
This follows directly from Remark~\ref{rmk:Shs of Cats}.
\end{proof}

\begin{remark}
In this proof we first show that each diffeological vector bundle is determined by its \textit{strictified} transition data (i.e.~matrix-valued functions), and then use that the latter satisfy descent, using Remark~\ref{rmk:Shs of Cats}.
Alternatively, we could have shown first an analogue of Lemma~\ref{st:VBun_str as hoRan} for $\VBun_\Dfg$ and then compared $\iota^* \VBun_\Dfg$ and $\VBun_{str}$ on $\Cart \subset \Dfg$ only.
However, the proof of the first step is not significantly less complicated than the proof of Theorem~\ref{st:VBun_Cat = VBun_Dfg} below.
Additionally, from the route chosen here it becomes evident how one can assign a total space to any vector bundle on any diffeological space from its transition data, as well as that $N_{rel}\VBun_{str}$ is a classifying object in $\scH_{\infty,1}$ for vector bundles.
\qen
\end{remark}

We obtain the following strictification theorem for diffeological vector bundles:

\begin{theorem}
\label{st:VBun_Cat = VBun_Dfg}
There is an objectwise equivalence of (non-strict) presheaves of categories on $\Dfg$
\begin{equation}
\label{eq:CA: VBun_Cat --> VBun_Dfg}
	\CA \colon \iota^* \VBun_{str} \weq \VBun_\Dfg\,.
\end{equation}
\end{theorem}

\begin{proof}
The proof Theorem~\ref{st:VBun_Cat = VBun_Dfg} consists of several individual steps, which we present in detail in Appendix~\ref{app:Proof of VBun_Dfg Thm}:
first, we fix an arbitrary diffeological space $X$ and restrict our attention to the evaluation of~\eqref{eq:CA: VBun_Cat --> VBun_Dfg} at $X$, i.e.~to the functor
\begin{equation}
	\CA_{|X} \colon \iota^* \VBun_{str}(X) \weq \VBun_\Dfg(X)\,.
\end{equation}
Using Lemma~\ref{st:VBun_str as hoRan} we define the functor $\CA_{|X}$ on objects and (Definition~\ref{def:CA on objects}) and show in Proposition~\ref{st:CA produces Dfg VBuns} that it indeed produces diffeological vector bundles on $X$.
Next, we define $\CA_{|X}$ on morphisms (Definition~\ref{def:CA on morphisms}) and show that, for each $X \in \Dfg$, it is fully faithful (Proposition~\ref{st:CA is fully faithful}) and essentially surjective~\ref{st:CA is essentially surjective}.
Only then do we let $X \in \Dfg$ vary and complete the proof by showing that the functors $\CA_{|X}$ we constructed for each $X \in \Dfg$ assemble into a morphism of pseudo-functors $\Cart^\opp \to \Cat_\rmL$.
Since we have already shown that this morphism of pseudo-functors is objectwise an equivalence of categories, we obtain the desired equivalence result.
\end{proof}

This completes the proof of Theorem~\ref{st:Dfg VBuns descend along subductions}.

\subsection{Descent for the $(\infty,2)$-sheaf of gerbes with connection}
\label{sec:Descent for gerbes}

An important example of higher geometric structures consists of (higher) gerbes with connection.
The simplicial presheaf of $n$-gerbes with connection on $\Cart$ is obtained by applying the Dold-Kan correspondence to the chain complex
\begin{equation}
\begin{tikzcd}
	\rmU(1) \ar[r, "\dd \log"]
	& \Omega^1(-;\iu \RN) \ar[r, "\dd"]
	& \cdots \ar[r, "\dd"]
	& \Omega^{n+1}(-; \iu\RN)\,,
\end{tikzcd}
\end{equation}
with the sheaf $\rmU(1)$ of smooth $\rmU(1)$-valued functions in degree zero.
The resulting simplicial presheaf is denoted $\rmB^{n+1}_\nabla\rmU(1)$.
It satisfies descent with respect to $\tau_{dgop}$, and thus \smash{$S_{\infty,0}(\rmB_\nabla^{n+1}\rmU(1)) \in \wH_{\infty,0}$} is a fibrant object.
It classifies $n$-gerbes on objects of \smash{$\wC$}.

However, this only captures invertible morphisms of $n$-gerbes with connection.
At least for $n=1$ there also exists a theory of interesting non-invertible morphisms, due to Waldorf~\cite{Waldorf--More_morphisms}.
There, Waldorf showed that the morphisms between two fixed gerbes on a manifold satisfy descent along surjective submersions.
This was later improved upon by Nikolaus-Schweigert in~\cite{NS--Equivariance_in_higher_geometry}%
; they proved that the presheaf of 2-categories which assigns to a manifold its gerbes and their 1- and 2-morphisms satisfies descent along surjective submersions of manifolds.
Descent properties for gerbes and their non-invertible morphisms on more general spaces are so far unknown.

However, given our results in Section~\ref{sec:higher sheaves}, we can establish such results:
consider the following strict presheaf $\Grb_\triv^\nabla$ of strict 2-categories on $\Cart$:
for $c \in \Cart$, an object of $\Grb_\triv^\nabla(c)$ is a 2-form $\rho \in \Omega^2(c;\fru(1))$.
A 1-morphism $\rho_0 \to \rho_1$ consists of a 1-form $A \in \Omega^1(c;\fru(k))$, for some $k \in \NN_0$.
Composition of 1-morphisms is given by $A' \circ A = A' \otimes \One + \One \otimes A$, with unit matrices $\One$ of the respective dimensions of $A$ and $A'$.
Given a pair $A \in \Omega^2(c;\fru(k))$, $A' \in \Omega^2(c;\fru(k'))$ of 1-morphisms $\rho_0 \to \rho_1$, a 2-morphism $A \to A'$ consists of a smooth, matrix-valued function $\psi \colon c \to \Mat(k {\times} k';\CN)$.
Vertical composition of 2-morphisms is given by pointwise matrix multiplication.
Note that a 1-morphism $A \in \Omega^1(c;\fru(k))$ is invertible precisely if $k=1$ and a 2-morphism is invertible precisely if it takes values in $\GL(k;\CN)$.
Then, $\Grb_\triv^\nabla$ satisfies descent with respect to $\tau_{dgop}$~\cite{Waldorf--More_morphisms, NS--Equivariance_in_higher_geometry}.

\begin{remark}
$\Grb_\triv^\nabla(c)$ as defined above is the most relaxed definition of a 2-category of gerbes with connection on $c$.
One can restrict attention to parallel 2-morphisms, i.e.~satisfying $\dd \psi + (A-A')\psi = 0$, and independently to 1-morphisms satisfying either $\tr(\dd A + \frac{1}{2} [A,A]) = \rho_1 - \rho_0$ (as in~\cite{Waldorf--More_morphisms}) or the fake curvature condition $\dd A + \frac{1}{2} [A,A] = \rho_1 - \rho_0$.
These choices produce various 2-categories of gerbes with connection on $c$ with different properties; our arguments below apply to each of these cases.
\qen
\end{remark}

Applying the $(\infty,2)$-categorical nerve from~\cite[Thm.~B]{Moser:Double-oo-Cat_nerve} objectwise at each $c \in \Cart$ produces a fibrant presheaf $N_{2,rel} \Grb_\triv^\nabla \in \scH_{\infty,2}$.
It follows from the descent property of $\Grb_\triv^\nabla$ that $N_{2,rel} \Grb_\triv^\nabla$ is even fibrant in $\scH_{\infty,2}^{loc}$, i.e.~satisfies descent with respect to $\tau_{dgop}$.
By Corollary~\ref{st:QAd for H_(oo,n)^loc}, the homotopy right Kan extension \smash{$S_{\infty,2}(N_{2,rel} \Grb_\triv^\nabla) \in \wH_{\infty,2}^{loc}$} is again fibrant and classifies gerbes with connection and their not necessarily invertible 1- and 2-morphisms on generic objects in $\wC$.
We thus obtain:

\begin{theorem}
\label{st:Grbs descend}
Gerbes with connection and their generic (not necessarily invertible or parallel) 2-morphisms form a sheaf of $(\infty,2)$-categories on the site $(\wC, \wtau_{dgop})$.
\end{theorem}

In particular, restricting to objects $X \in \wC$ of the form $X = \ul{M}$ for some manifold $M$, we directly obtain the result of~\cite{NS--Equivariance_in_higher_geometry} that gerbes with connection form a sheaf of 2-categories on the site of manifolds and surjective submersions.

\begin{remark}
To conclude this subsection, we remark that since $N_{2,rel} \Grb_\triv^\nabla(c)$ is a $(2,2)$-category, i.e.~it is truncated, for each $c \in \Cart$, and since truncatedness is stable under limits, $S_{\infty,2}(N_{2,rel} \Grb_\triv^\nabla)$ should be a sheaf of $(2,2)$-categories as well (see, for instance,~\cite[Sec.~11]{Rezk:Cartesian_pres_of_oon-Cats} for truncations of $(\infty,n)$-categories).
\qen
\end{remark}

\subsection{Descent and coherence for smooth functorial field theories}
\label{sec:Descent and coherence}

We now apply our findings to smooth functorial field theories (FFTs), a family-version of topological quantum field theories (TQFTs).

\subsubsection{Background on presentations of topological quantum field theories}

We briefly recall some well-known background on TQFTs, highlighting only those technical details which are relevant to this paper.
TQFTs were introduced in~\cite{Atiyah:TQFTs, Segal1987}; for some modern introductions, we refer the reader to~\cite{Kock:2D-TFTs, Safronov:TQFT_Notes}.
The formulation of $d$-dimensional TQFTs rests on the $d$-dimensional cobordism category $\sfBord_d$.
Its objects are $(d{-}1)$-dimensional closed, oriented manifolds.
Its morphisms are diffeomorphism classes of $d$-dimensional oriented cobordisms between these.
Composition is given by gluing bordisms along the mutual boundary.
The category $\sfBord_d$ is symmetric monoidal under disjoint union of manifolds.
If $\sfT$ is a symmetric monoidal category, a $d$-dimensional $\sfT$-valued TQFT is a symmetric monoidal functor $\sfZ \colon \sfBord_d \to \sfT$.

If $Y$ is an object in $\sfBord_d$, then any orientation-preserving diffeomorphism $f \colon Y \to Y$ induces a morphism $C_f \colon Y \to Y$ in $\sfBord_d$.
Isotopic diffeomorphisms induce the same morphism.
Via the TQFT functor $\sfZ$, the object $\sfZ(Y) \in \sfT$ thus carries an action of the mapping class group of $Y$.
Let $c\Mfd^{or}_{d-1}$ denote the groupoid of closed oriented $(d{-}1)$-manifolds and their orientation-preserving diffeomorphisms.
We let $c\Mfd^{or}_{d-1}/{\sim}$ denote the groupoid with the same objects, but with morphisms given by isotopy classes of diffeomorphisms.
Observe that there are functors $c\Mfd^{or}_{d-1} \longrightarrow c\Mfd^{or}_{d-1}/{\sim} \longrightarrow \sfBord_d$.
If $Y_0 \in c\Mfd^{or}_{d-1}$ and $\MCG(Y_0)$ is its mapping class group, let $Y_0 \dslash \MCG(Y_0)$ denote the associated action groupoid.
The inclusion $Y_0 \dslash \MCG(Y_0) \hookrightarrow c\Mfd^{or}_{d-1} /{\sim}$ is an equivalence onto a connected component, and hence the inclusion
\begin{equation}
\begin{tikzcd}
	\displaystyle{\coprod_{[Y_0] \in \pi_0 (c\Mfd^{or}_{d-1})}} Y_0 \dslash \MCG(Y_0) \ \ar[r, hookrightarrow]
	& c\Mfd^{or}_{d-1}
\end{tikzcd}
\end{equation}
is an equivalence.
Consequently, we can recover the values of any TQFT $\sfZ$ on all of $c\Mfd^{or}_{d-1}$ from its restrictions to $Y_0 \dslash \MCG(Y_0)$, where $Y_0$ ranges over all diffeomorphism classes of closed, oriented $(d{-}1)$-manifolds.
Further, considering a generating set of bordisms between the chosen representatives $Y_0$, we can then recover the full TQFT up to natural isomorphism from its values on the $Y_0$, the respective mapping class group actions and the generating bordisms.

\subsubsection{Diffeomorphism actions in smooth functorial field theories}

The picture changes when we pass to smooth FFTs.
These differ from TQFTs in several ways: most importantly, we would like to decorate the manifolds underlying objects and morphisms in $\sfBord_d$ with additional geometric structure, and we would like to keep track of how the value of a field theory varies when we vary this additional geometric structure.
In the most common and relevant case the additional structure consists of a smooth map to a fixed background manifold $M$.
We restrict our attention to this case here.
Smooth FFTs were introduced by Stolz and Teichner in~\cite{ST:SuSy_FTs_and_generalised_coho}.
For background, we refer the reader there; the specific formalism employed in this paper was developed in~\cite{BW:Smooth_OCFFTs}.

\begin{remark}
One can set up smooth FFTs with geometric structures in a fully general way:
the additional data is modelled by a section of a (higher) sheaf on a site of families of $d$-manifolds and embeddings; we refer the reader to~\cite{LS:1d_smooth_FFTs, GP:Smooth_Geo_Bord} for the full framework.
\qen
\end{remark}

In this setup, the bordism category $\sfBord_d$ is replaced by a presheaf of symmetric monoidal categories $\Bord^M_d \colon \Cart^\opp \to \Cat^\otimes$; this assigns to $c \in \Cart$ the symmetric monoidal category $\Bord^M_d(c)$, whose objects are essentially pairs of a closed, oriented $d$-manifold $Y$ and a smooth map $f \colon c \times Y \to M$.
Its morphisms $(Y_0,f_0) \to (Y_1,f_1)$ can be described as bordisms $\Sigma \colon Y_0 \to Y_1$, together with a smooth map $\sigma \colon c \times \Sigma \to M$ which restricts to $f_0$ and $f_1$ on the respective incoming and outgoing boundaries of the bordism (for the full details of this model for $\Bord_d^M$, see~\cite{BW:Smooth_OCFFTs}).
Finally, for any smooth map $\varphi \colon c \times \Sigma \to c \times \Sigma'$ which is a fibrewise diffeomorphism of bordisms $Y_0 \to Y_1$ over $c$, one identifies $(\Sigma,\sigma)$ and $(\Sigma', \sigma \circ \varphi)$ as morphisms in $\Bord_d^M(c)$.
This defines a strict functor $\Cart^\opp \to \Cat_\rmL$.
We further enhance this to a strict functor taking values in the 2-category $\Cat_\rmL^\otimes$ of symmetric monoidal categories:
the symmetric monoidal structure on $\Bord_d(c)$ takes disjoint unions of underlying manifolds on objects as well as morphisms; that is, on objects we set $(Y, f) \otimes (Y', f') = (Y \sqcup Y', f \sqcup f')$, and on morphisms $[\Sigma, \sigma] \otimes [\Sigma' , \sigma'] = [\Sigma \sqcup \Sigma', \sigma \sqcup \sigma']$, where $[\Sigma, \sigma]$ denotes an equivalence class under the relation of fibrewise diffeomorphism described above.
We thereby obtain a strict functor $\Bord_d \colon \Cart^\opp \to \Cat_\rmL^\otimes$.
Let $\scT \colon \Cart^\opp \to \Cat_\rmL^\otimes$ be a presheaf of symmetric monoidal categories on $\Cart$.
A smooth $d$-dimensional FFT on $M$ with values in $\scT$ is a morphism of such presheaves $\scZ \colon \Bord^M_d \to \scT$.

In this setting, diffeomorphisms $\varphi \colon Y \to Y'$ of $d$-manifolds still induce morphisms in $\Bord_d^M(*)$, but because of the additional maps to the background space $M$, isotopic diffeomorphisms will, in general, no longer give rise to the same morphisms.
Instead, the value of a smooth FFT on an object $(Y,f) \in \Bord_d^M(*)$ comes endowed with a \emph{smooth} action of the \emph{full} group $\Diff_+(Y)$ of orientation-preserving diffeomorphisms of $Y$.
This action will not, in general, factor through the mapping class group of $Y$.

\begin{example}
Consider a 2-dimensional smooth FFT on $M$ with values in the higher sheaf of vector bundles, $\scZ \colon \Bord_1^M \to \VBun$.
Fix $Y = \bbS^1$.
The collection of all objects $(\bbS^1,f)$ in $\Bord_d^M$ forms the smooth free loop space $LM$ of $M$, which carries a canonical smooth $\Diff_+(\bbS^1)$-action.
The FFT $\scZ$ contains the data of a smooth, $\Diff_+(\bbS^1)$-equivariant vector bundle on $LM$.
An explicit construction of an FFT of this type is presented in~\cite{BW:Smooth_OCFFTs}.
\qen
\end{example}

In general, FFTs can take values in higher categories; therefore, we study fully coherent equivariant structures (or homotopy fixed points) on sections of higher sheaves under diffeomorphism actions.

If we aim to reconstruct the values of a smooth FFT on any closed $(d{-}1)$-manifold $Y$ with smooth map $f \colon Y \to M$, we need to work with the smooth $\Diff_+(Y)$-actions rather than mapping class group actions.
This motivates the following consideration:
we restrict our attention to the diffeomorphism class $[Y]$ of a closed, oriented $(d{-}1)$-manifold $Y$.
Let $\Mfd^\ori$ denote the groupoid of oriented manifolds and orientation-preserving diffeomorphisms, and let $\scM_Y$ be the connected component of $Y$ in $\Mfd^\ori$.
For $X,X' \in \Mfd^\ori$, we let $\rmD(X,X')$ denote the presheaf on $\Cart$ consisting of diffeomorphisms from $X$ to $X'$ (this is even a diffeological space).
Concretely, an element of $D(X,X')(c)$ is a smooth map $\varphi^\dashv \colon c \times X \to X'$ which is a diffeomorphism at each $x \in c$ and such that the map of pointwise inverses is also smooth.
We equivalently write this as a morphism $\varphi \colon \scY_c \to \rmD(X,X')$.
This establishes both $\Mfd^\ori$ and $\scM_Y$ as categories enriched in \smash{$\widehat{\Cart}$}; the enriched mapping objects are \smash{$\ul{\scM}^{\widehat{\Cart}}_Y(Y_0,Y_1) = \rmD(Y_0,Y_1)$}.
We also use the shorthand notation $\rmD(Y) \coloneqq \rmD(Y,Y)$.

\begin{remark}
One could even consider $\Mfd$ and $\Mfd^\ori$ as enriched in $\Dfg$; the results in this section still go through because they only rely on the (pre)sheaf aspects of the diffeological spaces involved.
\qen
\end{remark}

Consider \smash{$\widehat{\Cart}$}-enriched functors $P \colon \scM_Y^\opp \longrightarrow \widehat{\Cart}$ and $G \colon \scM_Y \longrightarrow \widehat{\Cart}$, as well as the functors $G^n$, with $G^n(X) \coloneqq (G(X))^n$.
Setting $G^0(X) \coloneqq *$ gives rise to an augmented simplicial object
\begin{equation}
	G^{\bullet+1} \in \big( \Cat_\rmL(\scM_Y, \widehat{\Cart}) \big)^{\bbDelta_+^\opp}\,,
\end{equation}
where $\bbDelta_+$ is the simplex category $\bbDelta$ with an initial object $[-1]$ adjoined to it.

\begin{example}
In the context of smooth field theories on a manifold $M$, the relevant choice of $P$ is $M^{(-)}$, which assigns to $Y' \in \scM_Y$ the mapping presheaf $M^{Y'}$.
Given two diffeomorphic manifolds $Y_0 \cong Y_1$ in $\scM_Y$, there is a canonical morphism of presheaves
\begin{equation}
	\scM_Y(Y_0,Y_1) = \rmD(Y_0,Y_1)
	\longrightarrow \widehat{\Cart}\big( M^{Y_1}, M^{Y_0} \big)
\end{equation}
which sends a map $\varphi^\dashv \colon c \times Y_0 \to Y_1$ to the composition
\begin{equation}
	M^{Y_1}(c) \ni f
	\longmapsto \Big(
	\begin{tikzcd}
		c \times Y_0 \ar[r, "\Delta"]
		& c \times c \times Y_0 \ar[r, "1 \times \varphi^\dashv"]
		& c \times Y_1 \ar[r, "f"]
		& M
	\end{tikzcd}
	\Big) \in M^{Y_0}(c)\,.
\end{equation}
This makes $M^{(-)}$ into an enriched functor.
\qen
\end{example}

This setup allows us to form, for each $n \in \NN_0$, the enriched two-sided simplicial bar construction~\cite[Section~9.1]{Riehl:Cat_HoThy}, which produces a simplicial object \smash{$B_\bullet \big( G^n, \scM_Y^\opp, P \big) \ \in \widehat{\Cart}{}^{(\bbDelta^\opp)}$}.
Explicitly,
\begin{align}
\label{eq:enriched spl bar construction}
	B_k \big( G^n, \scM_Y^\opp, P \big)
	&= \coprod_{Y_0, \ldots, Y_k \in \scM_Y} P(Y_0) \times \rmD(Y_1, Y_0) \times \cdots \times \rmD(Y_k, Y_{k-1}) \times G^n(Y_k)\,.
\end{align}

We now consider the case of $G = \rmD(Y,-) \colon \scM_Y \to \Dfg$.
Define morphisms $\Phi_n$ as the composition
\begin{equation}
\begin{tikzcd}[row sep=1cm, column sep=1.75cm]
	P(Y_0) \times \rmD(Y,Y_0)^{n+1} \ar[r] \ar[d, dashed, "\Phi_n"'] & P(Y_0) \times \rmD(Y,Y_0) \times \rmD(Y,Y_0)^n \ar[d, "1 \times \Delta \times 1"]
	\\
	P(Y) \times \rmD(Y,Y)^n & P(Y_0) \times \rmD(Y,Y_0)^2 \times \rmD(Y,Y_0)^n \ar[l, "\Ev  \times \delta^n"]
\end{tikzcd}
\end{equation}
where $\Delta$ denotes the diagonal morphism, and where $\Ev \colon P(Y_0) \times \rmD(Y,Y_0) \to P(Y)$ is defined via the tensor adjunction.
The morphism $\delta$ acts as
\begin{equation}
	(f, f_1, \ldots, f_n) \longmapsto (f^{-1}f_1, f_1^{-1} f_2, \ldots, f_{n-1}^{-1}f_n)\,.
\end{equation}
Let $\sfB\rmD(Y) \subset \scM_Y$ denote the full \smash{$\wCart$}-enriched subcategory on the object $Y$.
We call
\begin{equation}
	\big( P(Y) \dslash \rmD(Y) \big)_\bullet
	\coloneqq B_\bullet \big( *, \sfB\rmD(Y)^\opp, P \big) \in \widehat{\Cart}{}^{(\bbDelta^\opp)}
\end{equation}
the \emph{action groupoid} of the $\rmD(Y)$-action via $\Ev$ on \smash{$P(Y) \in \widehat{\Cart}$}.
We further set
\begin{equation}
	P \dslash \scM_Y \coloneqq B_\bullet \big( *, \scM_Y^\opp, P \big) \in \widehat{\Cart}{}^{(\bbDelta^\opp)}\,.
\end{equation}

\begin{lemma}
For every $n \in \NN_0$, the morphism
\begin{equation}
	\Phi_n \colon B_0 \big( \rmD(Y,-)^{n+1}, \scM_Y^\opp, P \big)
	\longrightarrow \big( P(Y) \dslash \rmD(Y) \big)_n
\end{equation}
is an augmentation of the simplicial object $B_\bullet \big( \rmD(Y,-)^{n+1}, \scM_Y^\opp, P \big)$ in \smash{$\widehat{\Cart}$}.
\end{lemma}

\begin{proof}
This follows directly from the compatibility of $\Ev$ with compositions.
\end{proof}

We further define morphisms $\Psi_k \colon B_k (\rmD(Y,-), \scM_Y^\opp, P) \longrightarrow (P \dslash \scM_Y)_k$, which arise from the augmentation $\rmD(Y,-)^\bullet \to \rmD(Y,-)^0 = *$\,.

\begin{proposition}
\label{st:Psi_k and Phi_n}
Let $k \in \NN_0$.
\begin{myenumerate}
\item The morphism
\begin{equation}
	\Psi_k \colon B_k \big( \rmD(Y,-), \scM_Y^\opp, P \big)
	\longrightarrow B_k (*, \scM_Y^\opp, P) = (P \dslash \scM_Y)_k
\end{equation}
is an augmentation of the simplicial object $B_k \big( \rmD(Y,-)^{\bullet+1}, \scM_Y^\opp, P \big)$ in \smash{$\widehat{\Cart}$}.

\item The morphism $\Psi_k$ is a $\wtau_{dgop}$-local epimorphism, and the simplicial object $B_k (\rmD(Y,-)^{\bullet+1}, \scM_Y^\opp, P)$ is isomorphic to the \v{C}ech nerve of $\Psi_k$.
\end{myenumerate}
\end{proposition}

\begin{proof}
Part (1) is immediate from the construction of $\Psi_k$.
For part (2), we only need to observe that $\Psi_k$ is a coproduct of projections onto a factor in a product and apply Lemma~\ref{st:tau-loc epi properties}.
\end{proof}

For $k,l \in \NN_0$, we introduce the short-hand notation
\begin{equation}
	\Coh_{k,l}(Y,P) \coloneqq B_k \big( \rmD(Y,-)^{l+1}, \scM_Y^\opp, P \big)
	\quad \in \wCart\,.
\end{equation}

\begin{corollary}
We obtain a bisimplicial object in \smash{$\wCart$} which is augmented in each direction,
\begin{equation}
\label{eq:bi-augmented bar construction}
\begin{tikzcd}[row sep=1cm]
	\Coh(Y,P) \ar[r, "\Phi_\bullet"] \ar[d, "\Psi_\bullet"]
	& \big( P(Y) \dslash \rmD(Y) \big)_\bullet
	\\
	(P \dslash \scM_Y)_\bullet
\end{tikzcd}
\end{equation}
and where the vertical simplicial objects are the \v{C}ech nerves of the subductions $\Psi_k$.
\end{corollary}

Explicitly, diagram~\eqref{eq:bi-augmented bar construction} reads as
\begin{equation}
\begin{tikzcd}[row sep=1cm]
	\stackrel{\vdots}{\cdots \displaystyle\coprod_{Y_0,Y_1 \in \scM_Y} P(Y_0) \times \rmD(Y_1,Y_0) \times \rmD(Y,Y_1)^2}
	\ar[r, shift left=0.2cm, "d^h_i"] \ar[r, shift left=-0.2cm]
	\ar[d, shift left=-0.2cm, "d^v_i"'] \ar[d, shift left=0.2cm]
	& \stackrel{\vdots}{\displaystyle\coprod_{Y_0 \in \scM_Y} P(Y_0) \times \rmD(Y,Y_0)^2}
	\ar[r, "\Phi_1"]
	\ar[d, shift left=-0.2cm, "d^v_i"'] \ar[d, shift left=0.2cm]
	\ar[l]
	& \stackrel{\vdots}{P(Y) \times \rmD(Y)}
	\ar[d, shift left=-0.2cm, "d_i"'] \ar[d, shift left=0.2cm]
	\\
	\cdots \displaystyle\coprod_{Y_0,Y_1 \in \scM_Y} P(Y_0) \times \rmD(Y_1,Y_0) \times \rmD(Y,Y_1)
	\ar[r, shift left=0.2cm, "d^h_i"] \ar[r, shift left=-0.2cm]
	\ar[d, "\Psi_1"]
	\ar[u]
	& \displaystyle\coprod_{Y_0 \in \scM_Y} P(Y_0) \times \rmD(Y,Y_0)
	\ar[r, "\Phi_0"]
	\ar[d, "\Psi_0"]
	\ar[l]
	\ar[u]
	& P(Y)
	\ar[u]
	\\
	\cdots \displaystyle\coprod_{Y_0,Y_1 \in \scM_Y} P(Y_0) \times \rmD(Y_1,Y_0)
	\ar[r, shift left=0.2cm, "d_i"] \ar[r, shift left=-0.2cm]
	& \displaystyle\coprod_{Y_0 \in \scM_Y} P(Y_0)
	\ar[l]
	&
\end{tikzcd}
\end{equation}

\begin{proposition}
\label{st:Psi, Phi yield weqs}
The morphisms $\Phi_\bullet$ and $\Psi_\bullet$ have the following properties:
\begin{myenumerate}
\item Each $\Phi_n$ induces a weak equivalence in $\scH_{\infty,0}$:
\begin{equation}
	\Phi_n \colon \Coh_{\bullet,n}(Y,P)
	\weq \sfc_\bullet \big( P(Y) \dslash \rmD(Y) \big)_n\,.
\end{equation}

\item Each $\Psi_k$ induces a weak equivalence in $\scH_{\infty,0}^{loc}$:
\begin{equation}
	\Psi_k \colon \Coh_{k,\bullet}(Y,P)
	\weq \sfc_\bullet (P \dslash \scM_Y)_k\,.
\end{equation}
\end{myenumerate}
\end{proposition}

\begin{proof}
Ad~(1):
We claim that the augmented simplicial object
\begin{equation}
\begin{tikzcd}
	\Coh_{\bullet,n}(Y,P) = B_\bullet \big( \rmD(Y,-)^{n+1}, \scM_Y^\opp, P \big) \ar[r, "\Phi_n"]
	& P(Y) \times \rmD(Y)^n
\end{tikzcd}
\end{equation}
in $\scH_{\infty,0}$ admits extra degeneracies.
To see this, we define morphisms in \smash{$\Cart$}
\begin{align}
	s_{-1|-1} \colon P(Y) \times \rmD(Y)^n
	&\longrightarrow
	\displaystyle\coprod_{Y_0 \in \scM_Y} P(Y_0) \times \rmD(Y,Y_0)^{n+1}\,,
	\\
	(\gamma, f_1, \ldots, f_n) &\longmapsto (\gamma, 1_Y, f_1, \ldots, f_n)\,,
\end{align}
where the image lies in the summand labelled by $Y \in \scM_Y$.
For $k \in \NN_0$, we set
\begin{align}
	s_{-1|k} \colon &\displaystyle\coprod_{Y_0, \ldots, Y_k \in \scM_Y} P(Y_0) \times \rmD(Y_1,Y_0) \times \ldots \times \rmD(Y_k,Y_{k-1}) \times \rmD(Y,Y_k)^{n+1}
	\\*
	&\longrightarrow
	\displaystyle\coprod_{Y_0, \ldots, Y_{k+1} \in \scM_Y} P(Y_0) \times \rmD(Y_1,Y_0) \times \ldots \times \rmD(Y_{k+1},Y_k) \times \rmD(Y,Y_{k+1})^{n+1}\,,
	\\
	&(\gamma, g_0, \ldots, g_{k-1}, f_0, \ldots, f_n)
	\longmapsto
	(\gamma, g_0, \ldots, g_{k-1}, 1_Y, f_0, \ldots, f_n)\,,
\end{align}
where the image lies in the component of the coproduct labelled by $(Y_0, \ldots, Y_k, Y)$.
These morphisms yield the desired extra degeneracies.

Ad~(2):
This follows from Proposition~\ref{st:Psi_k and Phi_n} together with Theorem~\ref{st:local epi Thm}.
\end{proof}

\begin{proposition}
\label{st:Q diag}
With the above notation, the following statements hold true:
\begin{myenumerate}
\item We have a weak equivalence in $\scH_{\infty,0}$:
\begin{equation}
	Q \diag(\Phi) \colon Q \circ \diag \big( \Coh(Y,P)\big)
	\weq Q \big( P(Y) \dslash \rmD(Y) \big)\,.
\end{equation}

\item We have a weak equivalence in $\scH_{\infty,0}^{loc}$:
\begin{equation}
	Q \diag(\Psi) \colon Q \circ \diag \big( \Coh(Y,P)\big)
	\weq Q (P \dslash \scM_Y)\,.
\end{equation}
\end{myenumerate}
\end{proposition}

\begin{proof}
By Proposition~\ref{st:Psi, Phi yield weqs}, the morphism $\Phi \colon \Coh(Y,P) \to \sfc_\bullet(P(Y) \dslash \rmD(Y))$ is a levelwise weak equivalence of simplicial objects in $\scH_{\infty,0}$.
Since the functor $\diag \cong |{-}| \colon s\sSet_\rmL \to \sSet$ is homotopical (it is left Quillen and all objects in $s\sSet $ are cofibrant), $\diag (\Phi)$ is a weak equivalence in $\scH_{\infty,0}$.
This implies~(1), since $Q$ is homotopical.

For claim~(2), we observe that we have a commutative diagram
\begin{equation}
\begin{tikzcd}[column sep=2.5cm, row sep=1.25cm]
	\underset{k \in \bbDelta^\opp}{\hocolim}^{\scH_{\infty,0}} Q \big( \Coh(Y,P)_k \big) \ar[r, "\hocolim\, Q \Psi_k", "\sim"'] \ar[d, "\sim"]
	& \underset{k \in \bbDelta^\opp}{\hocolim}^{\scH_{\infty,0}} Q \big( (P \dslash \scM_Y)_k \big) \ar[d, "\sim"]
	\\
	Q \circ \diag \big( \Coh(Y,P) \big) \ar[r, "Q \circ \diag (\Psi)"']
	& Q (P \dslash \scM_Y)
\end{tikzcd}
\end{equation}
The top morphism is a local weak equivalence by Proposition~\ref{st:Psi, Phi yield weqs}, and the vertical morphisms are projective weak equivalences by Proposition~\ref{st:bar construction properties}.
It follows that the bottom morphism is a local weak equivalence.
\end{proof}

\begin{definition}
\label{def:equivar and coherent sections}
Let $n \in \NN_0$, and let $\scF \in \scH_{\infty,n}$ be a presheaf of $(\infty,n)$-categories on $\Cart$.
\begin{myenumerate}
\item We define $(\infty,n)$-categories
\begin{align}
	\scF(P)^{\rmD(Y)} &\coloneqq C_{\CSS_n} \big( *, \bbDelta, (S_{\infty,n} \scF) \big( P(Y) \dslash \rmD(Y) \big) \big)\,,
	\\[0.1cm]
	\scF(P)^{\rmD(Y)}_{red} &\coloneqq \big\{ Q (P(Y) \dslash \rmD(Y)),\, \scF \big\}^{\Cart^\opp}\,,
\end{align}
where we make use of the simplicial enrichment of $\CSS_n$.
We call $\scF(P)^{\rmD(Y)}$ the $(\infty,n)$-category of \emph{equivariant sections of $\scF$ over \smash{$P(Y) \in \wCart$}}.

\item We define $(\infty,n)$-categories
\begin{align}
	\scF(P)^{coh} &\coloneqq C_{\CSS_n} \big( *, \bbDelta, (S_{\infty,n} \scF) (P \dslash \scM_Y) \big)
	\\[0.1cm]
	\scF(P)^{coh}_{red} &\coloneqq \big\{ Q (P \dslash \scM_Y),\, \scF \big\}^{\Cart^\opp}\,,
\end{align}
and we call $\scF(P)^{coh}$ the $(\infty,n)$-category of \emph{coherent sections of $\scF$ over $P$}.
\end{myenumerate}
\end{definition}

Note that by Proposition~\ref{st:cobar computes holim} $\scF(P)^{\rmD(Y)}$ and $\scF(P)^{coh}$ are models for the homotopy limits
\begin{align}
	\scF(P)^{\rmD(Y)} &= \underset{k \in \bbDelta}{\holim}^{\CSS_n}\ (S_{\infty,n} \scF) \big( (P(Y) \dslash \rmD(Y))_k \big)\,,
	\\
	\scF(P)^{coh} &= \underset{k \in \bbDelta}{\holim}^{\CSS_n}\ (S_{\infty,n} \scF) \big( (P \dslash \scM_Y)_k \big)\,.
\end{align}
By Lemma~\ref{st:cobar and fibrancy} these are indeed both $(\infty,n)$-categories, i.e.~fibrant objects in $\CSS_n$.
The first homotopy limit describes the $(\infty,n)$-category of homotopy fixed points of the smooth action of $\rmD(Y)$ on sections of $\scF$ over $P(Y)$.
In particular, the homotopy fixed point data depends smoothly on the diffeomorphisms in $\rmD(Y)$.
The second homotopy limit can be described as follows:
for each manifold $Y_0 \in \scM_Y$, we can form the $(\infty,n)$-category of sections of $\scF$ over $P(Y_0)$; it is given by $S_{\infty,n}(\scF)(P(Y_0))$.
Given another manifold $Y_1 \in \scM_Y$, there is a small presheaf (even a diffeological space) $\rmD(Y_0,Y_1)$ of diffeomorphisms from $Y_0$ to $Y_1$.
We can pull back sections of $\scF$ over $P(Y_1)$ along any such diffeomorphism, and the resulting section over $P(Y_0)$ should depend smoothly on the diffeomorphism.
The homotopy limit $\scF^{coh}$ can be understood as the $(\infty,n)$-category of families of sections of $\scF(Y_0)$ over $P(Y_0)$, for all $Y_0 \in \scM_Y$; these sections further are homotopy coherent with respect to all diffeomorphisms $Y_0 \to Y_1$, and their homotopy coherence data depends smoothly on these diffeomorphisms.

After this preparation, we give two different proofs (in Theorems~\ref{st:ext via choice of diffeos} and~\ref{st:diagonal coherence zig-zag}) that the $(\infty,n)$-categories of equivariant sections and coherent sections are equivalent.
The first proof relies on the fact that the inclusion $\sfB\rmD(Y) \hookrightarrow \scM_Y$ of groupoid objects in \smash{$\widehat{\Cart}$} is an equivalence, whereas the second proof relies on the descent results from Section~\ref{sec:higher sheaves}.
Both of these proofs have different advantages: the first is more slick and less involved.
However, the choice of an inverse equivalence to the inclusion $\sfB\rmD(Y) \hookrightarrow \scM_Y$ is by no means canonical, thus making the equivalences obtained in this proof less helpful in practice.
In the second proof, the inverse equivalence is obtained as the inverse of a pullback along a $\wtau_{dgop}$-local epimorphism.
That is, it is given precisely by descent.
In many important cases, descent functors are available, so that the second perspective is more helpful for practical purposes.

\begin{theorem}
\label{st:ext via choice of diffeos}
Let $\scF \in \scH_{\infty,n}$ \emph{projectively} fibrant.
There are (weakly) inverse weak equivalences
\begin{equation}
\begin{tikzcd}[column sep=1.25cm]
	\scF(P)^{\rmD(Y)}_{red} \ar[r, shift left=-0.1cm, "\sim" yshift=-0.05cm]
	& \scF(P)^{coh}_{red}\,. \ar[l, shift left=-0.1cm]
\end{tikzcd}
\end{equation}
\end{theorem}

\begin{proof}
We view the truncations to simplicial levels zero and one of $P(Y) \dslash \rmD(Y)$ and of $P \dslash \scM_Y = B(*,\scM_Y^\opp, P)$ as strict presheaves of groupoids on $\Cart$.
Let us denote these by $Gr(P(Y) \dslash \rmD(Y))$ and $Gr(B(*,\scM_Y^\opp, P))$, respectively.
Observe that
\begin{equation}
	P(Y) \dslash \rmD(Y) = N \circ Gr \big( P(Y) \dslash \rmD(Y) \big)
	\qandq
	P \dslash \scM_Y = N \circ Gr (P \dslash \scM_Y)\,,
\end{equation}
and that there is a canonical inclusion
\begin{equation}
	\iota \colon Gr(P(Y) \dslash \rmD(Y)) \hookrightarrow Gr(P \dslash \scM_Y)\,.
\end{equation}
Any choice of a family of diffeomorphisms $\{ f_{Y_0} \colon Y_0 \to Y \}_{Y_0 \in \scM_Y}$ determines a morphism
\begin{equation}
	p \colon Gr(P \dslash \scM_Y) \hookrightarrow Gr \big( P(Y) \dslash \rmD(Y) \big)\,,
\end{equation}
for which there exist 2-isomorphisms (i.e.~natural isomorphisms, coherent over $\Cart$) $p \circ \iota = 1$ and $\eta \colon \iota \circ p \to 1$ of presheaves of groupoids over $\Cart$.
In particular, $N \eta$ determines a simplicial homotopy
\begin{equation}
	N \eta \colon (P \dslash \scM_Y) \otimes \Delta^1 \longrightarrow P \dslash \scM_Y
\end{equation}
between the identity and $N(p \circ \iota)$ in $\scH_{\infty,0}$.
This, in turn, shows that both $\iota$ and $p$ are homotopy equivalences, weakly inverse to each other.
Consequently, we obtain a pair of projective weak equivalences
\begin{equation}
\begin{tikzcd}[column sep=1.5cm]
	Q \big( P(Y) \dslash \rmD(Y) \big) \ar[r, shift left=0.1cm, "Q \iota", "\sim"' yshift=0.025cm]
	& Q (P \dslash \scM_Y)\,. \ar[l, shift left=0.1cm, "Q p"]
\end{tikzcd}
\end{equation}
The weak equivalences in the claim are $(Q\iota)^*$ and $(Qp)^*$.
\end{proof}

\begin{theorem}
\label{st:diagonal coherence zig-zag}
For any $n \in \NN_0$ and for any fibrant $\scF \in \scH_{\infty,n}^{loc}$ there is a canonical zig-zag of weak equivalences between fibrant objects in $\CSS_n$,
\begin{equation}
\label{eq:diagonal coherence zig-zag}
\begin{tikzcd}[column sep={2.5cm,between origins}, row sep=1.cm]
	& \scF(P)^{\rmD(Y)}_{red} \ar[dl, "\sim"'] \ar[dr, "\sim" near start, "{(Q \diag (\Phi))^*}"' description]
	& 
	& \scF(P)^{coh}_{red} \ar[dl, "\sim"' near start, "{(Q \diag (\Psi))^*}" description] \ar[dr, "\sim"]
	&
	\\
	\scF(P)^{\rmD(Y)}
	& & \ul{\scH}_{\infty,n}^{\CSS_n} \big( Q \circ \diag \big( \Coh(Y,P) \big),\, \scF \big)
	& & \scF(P)^{coh}
\end{tikzcd}
\end{equation}
\end{theorem}

\begin{proof}
The second and third weak equivalences are direct consequences of Proposition~\ref{st:Q diag}.
The first and last weak equivalences follow by a direct computation:
we have canonical isomorphisms
\begin{align}
	\scF(P)^{coh} &= C_{\CSS_n} \big( *, \bbDelta, (S_{\infty,n} \scF) (P \dslash \scM_Y)_\bullet \big)
	\\
	&= C_{\CSS_n} \big( *, \bbDelta, \{ Q(P \dslash \scM_Y)_\bullet, \scF \}^{\Cart^\opp} \big)
	\\
	&\cong \big\{ B(*, \bbDelta^\opp, Q(P \dslash \scM_Y)_\bullet),\, \scF \big\}^{\Cart^\opp}\,.
\end{align}
Using the Bousfield-Kan map in the first argument of the functor tensor product now induces a canonical weak equivalence to this object from the object
\begin{equation}
	\{ Q(P \dslash \scM_Y), \scF \}^{\Cart^\opp}
	= \scF(P)^{coh}_{red}\,.
\end{equation}
This completes the proof.
\end{proof}

\begin{remark}
That $(Q \diag(\Psi))^*$ is a weak equivalence relies on the descent results in Section~\ref{sec:higher sheaves}; these entered in the proof of Proposition~\ref{st:Psi, Phi yield weqs} and hence also Proposition~\ref{st:Q diag}.
\end{remark}

\begin{remark}
In the case where $\scF$ is the sheaf of diffeological vector bundles from Definition~\ref{def:Dfg VBun} the insights from Theorem~\ref{st:diagonal coherence zig-zag} were used in~\cite{BW:Smooth_OCFFTs} in order to obtain a coherent vector bundle on $P$, where $Y = \bbS^1$ and $P = M^{(-)}$ for some manifold $M$, from the equivariant structure of the transgression line bundle of a bundle gerbe with connection.
Moreover, this procedure was also applied to bundles over certain spaces of paths in $M$, and the coherent vector bundles thus obtained were subsequently assembled into a smooth open-closed FFT on $M$.
In this sense, that smooth FFT was built from its values on generating objects by means of the equivalence of equivariant and coherent sections of a sheaf of higher categories as explored here.
There is also great interest in extended field theories, which are formulated in the language of $n$-fold complete Segal spaces~\cite{CS:Bord_n, Lurie:Classification_of_TQFTs}.
It is with this in mind that we prove Theorems~\ref{st:ext via choice of diffeos} and~\ref{st:diagonal coherence zig-zag} in the setting of $(\infty,n)$-categories.
\qen
\end{remark}

\subsection{Remarks on descent for sheaves of 2-vector bundles}
\label{sec:2VBdls}

In this section we point out potential applications of Corollary~\ref{st:QAd for H_(oo,n)^loc} to the theory of categorified vector bundles.
The gerbes of Section~\ref{sec:Descent for gerbes} can be considered as a model for rank-one 2-vector bundles.
The theory for 2-vector bundles of higher rank, and more generally for higher categorifications of vector bundles, is currently only little understood (see below for references).
Thus, this section does not contain formal results, but outlines how the results of Section~\ref{sec:(oo,n)-sheaves} will be applicable to the theory of higher vector bundles.
This will be similar to the ideas in Sections~\ref{sec:Dfg VBuns} and~\ref{sec:Descent for gerbes}.
We start by briefly mentioning several models for categorified vector bundles:

In~\cite[Sec.~4.3]{NS--Equivariance_in_higher_geometry} the authors consider 2-categorical presheaves of Kapranov-Voevodsky 2-vector bundles on the category of manifolds.
They apply the Grothendieck plus-construction to sheafify these higher presheaves and obtain a 2-categorical sheaf which assigns to each manifold a 2-category of 2-vector bundles, modelled on Kapranov-Voevodsky 2-vector spaces (as defined in~\cite{KV:2-Cats_and_tetrahedra_equations}).
A similar 2-category was considered in~\cite{BDR:2VBuns_and_ell_coho}.

In~\cite{KLW:2VBuns} the authors defined a presheaf of 2-categories on the site of manifolds with surjective submersions as coverings which assigns to each manifold a 2-category%
\footnote{We use the terms 2-category and bicategory interchangeably, both referring to the non-strict version.}
built from bimodule bundles over bundles of super-algebras. 
The authors then sheafify this presheaf of 2-categories and propose the resulting sheaf of 2-categories as a model for 2-vector bundles on the site of manifolds and surjective submersions.

Extending this, one can consider presheaves of $(\infty,n)$-categories on a small category $\scC$ which assign to each object higher Morita categories of $\mathbb{E}_n$-bimodules; for $\scC = \Cart$ or $\scC = \Mfd$, these could, for instance, be $A$-$B$-bimodules for the $\mathbb{E}_\infty$-algebras $A = \CN$ and $B$ the algebra of smooth functions on the underlying manifold (or derived versions thereof).
Sheafifying these higher presheaves on $\Cart$ with respect to differentiably good open coverings would give viable candidates for generalisations of higher vector bundles.

In each of these cases, one is presented with a sheaf of $(\infty,n)$-categories in the sense of Section~\ref{sec:(oo,n)-sheaves}.
Then, Corollary~\ref{st:QAd for H_(oo,n)^loc} applies, and we obtain that each such sheaf of $(\infty,n)$-categories of higher vector bundles on $(\Cart, \tau_{dgop})$ extends to a sheaf of $(\infty,n)$-categories on \smash{$(\wCart, \wtau)$}.
The resulting extended higher sheaves are good candidates for a description of categorified vector bundles on generalised smooth spaces and $\wtau_{dgop}$-coverings, such as manifolds and surjective submersions, or diffeological spaces and subductions.

\section{The image of right Kan extension on sheaves of spaces}
\label{sec:char of sheaves on wC}

We now classify those presheaves of $\infty$-groupoids on $\wC$ which are extensions of sheaves of $\infty$-groupoids on $(\scC,\tau)$.
We achieve this by passing to $\infty$-categorical language using quasi-categories (the model for $(\infty,1)$-categories used in~\cite{Lurie:HTT, Cisinski:Higher_Cats_and_Ho_Alg}), rather than working with model-categorical presentations, and prove an enhancement of Theorem~\ref{st:QAd for H_(oo,0)^loc} to an equivalence of quasi-categories of homotopy-coherent sheaves of $\infty$-groupoids.

\subsection{On a quasi-categorical site}

Let $\sfC$ be a quasi-category, and let $(\sfC, \tau)$ be a small quasi-categorical Grothendieck site (as defined in~\cite[Def.~6.2.2.1]{Lurie:HTT}, where they are called `$\infty$-sites') with underlying quasi-category $\sfC$.
We let $\wsfC = \ul{\sSet_\rmL}{}(\sfC^\opp, \sfS_\rmS)$ be the quasi-category of presheaves of small spaces on $\sfC$.
Let $\sfPSh(\sfC) = \ul{\sSet_\XL}(\sfC^\opp, \sfS_\rmL)$ denote the quasi-category of presheaves of large spaces on $\sfC$, and let $\sfSh(\sfC,\tau) \subset \sfPSh(\sfC)$ denote the quasi-category of sheaves on $\sfC$ with respect to $\tau$.
This induces the standard reflective localisation
\begin{equation}
\begin{tikzcd}
	\sfPSh(\sfC) \ar[r, shift left=0.15cm, "\perp"' yshift=0.04cm, "\sfL"] & \sfSh(\sfC,\tau)\,, \ar[l, hookrightarrow, shift left=0.15cm]
\end{tikzcd}
\end{equation}
where $\rmL$ is the sheafification functor.

\begin{remark}
In this section all quasi-categories denoted $\sfPSh$ and $\sfSh$ are categories of functors taking values in $\sfS_\rmL$, i.e.~the quasi-category of large spaces, unless otherwise indicated.
\qen
\end{remark}

In this set-up, a morphism $f \colon Y \to X$ in $\wsfC$ is a \emph{$\tau$-local epimorphism} whenever its sheafification $\sfL f \colon \sfL Y \to \sfL X$ is an effective epimorphism in $\sfSh(\sfC,\tau)$~\cite[Lemma~6.2.2.8]{Lurie:HTT}.
Since effective epimorphisms are stable under pullback and $\sfL$ is left exact, $\tau$-local epimorphisms induce a coverage---and hence a Grothendieck topology---on the homotopy category $\rmh \wsfC$.
Since a Grothendieck topology on $\rmh \wsfC$ is the same as a Grothendieck topology on $\wsfC$~\cite[Rmk.~6.2.2.3]{Lurie:HTT}, we obtain a new quasi-categorical Grothendieck site $(\wsfC, \wtau)$, whose topology is generated by the $\tau$-local epimorphisms in $\wsfC$.
The corresponding quasi-category $\sfSh(\wsfC,\wtau)$ of sheaves on $\wsfC$ is the localisation of $\sfPSh(\wsfC)$ at the morphisms
\begin{equation}
	|\cC \wY_f| \longrightarrow \wY_X\,,
\end{equation}
where $f \colon Y \to X$ is any $\tau$-local epimorphism, $\wY$ is the Yoneda embedding of $\wsfC$, and $\cC \wY_f$ is the \v{C}ech nerve of $\wY_f \colon \wY_Y \to \wY_X$.
We thus obtain another reflective localisation
\begin{equation}
\begin{tikzcd}
	\sfPSh(\wsfC) \ar[r, shift left=0.15cm, "\perp"' yshift=0.04cm, "\widehat{\sfL}"] & \sfSh(\wsfC,\wtau) \ar[l, hookrightarrow, shift left=0.15cm]
\end{tikzcd}
\end{equation}
whose left adjoint $\widehat{\sfL}$ is left exact.
Note that since the Yoneda embedding preserves limits and $\wsfC$ is complete, there is a canonical equivalence $\wY_{\cC f} \simeq \cC \wY_f$.
This directly leads to:

\begin{lemma}
\label{st:char of wtau-oo-sheaves}
A presheaf $\CG \in \sfPSh(\wsfC)$ is $\wtau$-local, i.e.~it lies in the full sub-quasi-category $\sfSh(\wsfC,\wtau) \subset \sfPSh(\wsfC)$, if and only if \smash{$\CG(X) \longrightarrow \underset{\bbDelta}{\lim}\ \CG(\cC f)$} is an equivalence, for every $\tau$-local epimorphism $f \colon Y \to X$.
\end{lemma}

Let $\sfPSh_!(\wsfC) = \ul{\Hom}_!(\wsfC^\opp, \sfS_\rmL)$ denote the full sub-quasi-category of $\sfPSh(\wsfC)$ on those functors which preserve small limits (equivalently the small-\emph{colimit}-preserving functors $\wsfC \to \sfS_\rmL^\opp$).
We also define a quasi-category $\sfSh_!(\wsfC,\wtau)$ as the pullback of quasi-categories
\begin{equation}
\begin{tikzcd}
	\sfSh_!(\wsfC,\wtau) \ar[r] \ar[d] & \sfSh(\wsfC,\wtau) \ar[d]
	\\
	\sfPSh_!(\wsfC) \ar[r] & \sfPSh(\wsfC)
\end{tikzcd}
\end{equation}
That is, $\sfSh_!(\wsfC,\wtau) \subset \sfPSh(\wsfC)$ is the full sub-quasi-category on those $\wtau$-sheaves whose underlying functor $\wsfC^\opp \to \sfS_\rmL$ turns small colimits in $\wsfC$ into limits in $\sfS_\rmL$.
The functor $\scY \colon \sfC \to \wsfC$ induces a functor $\sfC^\opp \to \wsfC^\opp$, and by a slight abuse of notation we again denote this functor by $\scY$.
As before, we leave the inclusion $\sfS_\rmS \subset \sfS_\rmL$ of the quasi-category of small spaces into that of large spaces implicit.
Observe that the left and right Kan extensions $\scY_!, \scY_*$ yield functors $\sfPSh(\wsfC) \to \sfPSh(\wsfC)$.

\begin{lemma}
\label{st:mps out of h_! and h_*}
Let $\sfC$ be a small quasi-category.
\begin{myenumerate}
\item For any $A \in \wsfC \subset \sfPSh(\sfC)$ and $\CG \in \sfPSh(\wsfC)$ there is a natural equivalence
\begin{equation}
	\sfPSh(\wsfC)(\scY_*A, \CG) \simeq \CG(A)\,.
\end{equation}
Equivalently, the restriction $\scY_{*|\wsfC}$ agrees with $\wY$ as functors $\wsfC \to \sfPSh(\wsfC)$.

\item For any $A \in \wsfC \subset \sfPSh(\sfC)$ and $\CG \in \sfPSh_!(\wsfC)$ there is a natural equivalence
\begin{equation}
	\sfPSh(\wsfC)(\scY_!A, \CG) \simeq \CG(A)\,.
\end{equation}

\end{myenumerate}
\end{lemma}

\begin{proof}
Using the equivalence $X \simeq \colim_{c \in \sfC_{/X}} \scY_c$, for $X \in \wsfC$, we obtain
\begin{align}
	\wY_A(X) \simeq \underset{c \in (\sfC_{/X})^\opp}{\lim}\ \wsfC(\scY_c, A)
	\simeq (\scY_*A)(X)\,,
\end{align}
proving the first equivalence.
To see the second equivalence, we use that
\begin{align}
	\sfPSh(\wsfC)(\scY_!A, \CG)
	&\simeq \sfPSh(\sfC)(A, \scY^*\CG)
	\simeq \underset{c \in \sfC_{/A}}{\lim}\ \sfPSh(\sfC)(\scY_c, \scY^*\CG)
	\simeq \underset{c \in \sfC_{/A}}{\lim}\ \CG(\scY_c)
	\\
	&\simeq \CG \big( \underset{c \in \sfC_{/A}}{\colim}\ \scY_c \big)
	\simeq \CG(A)\,.
\end{align}
In the second-to-last step we have used that $\CG$ turns colimits in $\wsfC$ into limits in $\sfS_\rmL$.
\end{proof}

\begin{lemma}
\label{st:colim and h}
Let $D \colon \scJ \to \wsfC$ be any small diagram.
There is a canonical equivalence
\begin{equation}
	\scY^*(\colim \wY_D) \simeq \scY^*(\wY_{(\colim D)})\,.
\end{equation}
\end{lemma}

\begin{proof}
This follows by a repeated application of the Yoneda Lemma and the fact that colimits in presheaf categories are computed objectwise.
In particular, we have
\begin{align}
	\scY^*(\wY_{(\colim D)})(c)
	&\simeq (\wY_{(\colim D)})(\scY_c)
	\simeq \wsfC(\scY_c, \colim D)
	\simeq \colim (Dc)\,,
	\qquad \text{and}
	\\
	\scY^*(\colim \wY_D)(c)
	&\simeq (\colim \wY_D)(\scY_c)
	\simeq \colim \wsfC(\scY_c,D)
	\simeq \colim (Dc)\,,
\end{align}
which yields the desired equivalence.
\end{proof}

We then have the following quasi-categorical enhancement of Theorem~\ref{st:QAd for H_(oo,0)^loc}.
This fully unveils the importance of the topology induced on \smash{$\wsfC$} by the $\tau$-local epimorphisms.

\begin{theorem}
\label{st:image of h_*}
Let $\sfC$ be a small quasi-category, and let $\scY_* \colon \sfPSh(\sfC) \to \sfPSh(\wsfC)$ denote the right Kan extension along $\scY \colon \sfC^\opp \to \wsfC^\opp$.
\begin{myenumerate}
\item The adjunction $\scY^* \dashv \scY_*$ induces an equivalence
\begin{equation}
\begin{tikzcd}
	\sfPSh_!(\wsfC) \ar[r, shift left=0.1cm, "\scY^*"] & \sfPSh(\sfC)\,. \ar[l, shift left=0.1cm, "\scY_*"]
\end{tikzcd}
\end{equation}

\item More generally, let $\tau$ be a Grothendieck topology on $\sfC$.
The adjunction $\scY^* \dashv \scY_*$ restricts to an equivalence
\begin{equation}
\begin{tikzcd}
	\sfSh_!(\wsfC,\wtau) \ar[r, shift left=0.1cm, "\scY^*"] & \sfSh(\sfC,\tau)\,. \ar[l, shift left=0.1cm, "\scY_*"]
\end{tikzcd}
\end{equation}
\end{myenumerate}
\end{theorem}

\begin{proof}
Claim~(1) is an explicit incidence of a well-known theorem:
there is a canonical equivalence $\ul{\Hom}(\sfC, \sfD) \simeq \ul{\Hom}_!(\wsfC,\sfD)$ between the quasi-categories of functors from $\sfC$ to any $\rmS$-cocomplete quasi-category $\sfD$ and (small-)colimit-preserving functors $\wsfC \to \sfD$.
For $\sfD = \sfS_\rmL^\opp$, this equivalence is established by $\scY_*$.
In particular, $\scY_*$ is fully faithful even as a functor $\sfPSh(\sfC) \to \sfPSh(\wsfC)$.
The adjunction $\scY^* : \sfPSh(\wsfC) \rightleftarrows \sfPSh(\sfC) : \scY_*$ restricts to an adjunction $\scY^* : \sfPSh_!(\wsfC) \rightleftarrows \sfPSh(\sfC) : \scY_*$, whose right adjoint $\scY_*$ is an equivalence.
Thus, the left adjoint $\scY^*$ is an equivalence as well~\cite[Prop.~6.1.6]{Cisinski:Higher_Cats_and_Ho_Alg}.

We now prove~(2):
first, we show that $\scY_*$ maps $\sfSh(\sfC,\tau)$ to $\sfSh(\wsfC,\wtau)$, i.e.~that it maps $\tau$-sheaves on $\sfC$ to $\wtau$-sheaves on $\wsfC$.
Let $F \in \sfSh(\sfC,\tau)$, and let $f \colon Y \to X$ be a $\tau$-local epimorphism in $\wsfC$.
We need to prove that the morphism
\begin{equation}
	\scY_* F(X) \simeq \sfPSh(\wsfC) (\wY_X, \scY_*F) \longrightarrow \sfPSh(\wsfC) \big( |\cC \wY_f|, \scY_*F \big)
	\simeq \sfPSh(\wsfC) \big( |\wY_{\cC f}|, \scY_*F \big)
\end{equation}
is an equivalence.
By Lemma~\ref{st:mps out of h_! and h_*}(1), Lemma~\ref{st:colim and h}, and because $\scY_*$ is fully faithful, we have equivalences
\begin{align}
	\sfPSh(\wsfC) \big( |\wY_{\cC f}|, \scY_*F \big)
	&\simeq \sfPSh(\sfC) \big( \scY^*|\wY_{\cC f}|, F \big)
	\simeq \sfPSh(\sfC) \big( \scY^*\wY_{|\cC f|}, F \big)
	\\
	&\simeq \sfPSh(\sfC) \big( \scY^* \scY_* |\cC f|, F \big)
	\simeq \sfPSh(\wsfC) \big( \scY_*|\cC f|, \scY_*F \big)
	\\
	&\simeq \sfPSh(\sfC) \big( |\cC f|, F \big)\,.
\end{align}
Using that $F$ is a $\tau$-sheaf and that $\sfL$ is a left exact left adjoint, we then compute
\begin{align}
	\sfPSh(\sfC) \big( |\cC f|, F \big)
	&\simeq \sfPSh(\sfC) \big( \sfL |\cC f|, F \big)
	\simeq \sfPSh(\sfC) \big( \sfL X, F \big)
	\simeq \sfPSh(\sfC) (X,F)
	\simeq (\scY_*F)(X)\,.
\end{align}
Thus, $\scY_*F$ is a $\wtau$-sheaf whenever $F$ is a $\tau$-sheaf.
In particular, $\scY_*$ restricts to a fully faithful functor
\begin{equation}
	\scY_* \colon \sfSh(\sfC,\tau) \longrightarrow \sfSh_!(\wsfC,\wtau)\,.
\end{equation}

Finally, we claim that the functor $\scY^* \colon \sfPSh_!(\wsfC) \to \sfPSh(\sfC)$ restricts to a functor $\scY^* \colon \sfSh_!(\wsfC,\wtau) \to \sfSh(\sfC,\tau)$; by part (1), this restriction is automatically fully faithful.
To that end, let $\CG \in \sfSh_!(\wsfC, \wtau)$.
We have to show that $\scY^*\CG \in \sfPSh(\sfC)$ is $\tau$-local, i.e.~that the presheaf $\scY^*\CG$ is a $\tau$-sheaf.
That is the case precisely if
\begin{equation}
	\scY^*\CG(c) \simeq \CG(\scY_c)
	\simeq \sfPSh(\sfC) (\scY_c, \scY^*\CG)
	\longrightarrow \sfPSh(\sfC) \big( |\cC g|, \scY^*\CG)
\end{equation}
is an equivalence for every $\tau$-covering $g \colon Y \to \scY_c$.
(Note that if $g$ is a $\tau$-covering, then $\sfL (|\cC g| \to \scY_c)$ is an equivalence, i.e.~$g$ is a $\tau$-local epimorphism.)
On the right-hand side, we have canonical equivalences
\begin{align}
	\sfPSh(\sfC) \big( |\cC g|, \scY^*\CG)
	\simeq \sfPSh(\wsfC) \big( \scY_!|\cC g|, \CG)
	\simeq \CG \big( |\cC g| \big)
	\simeq \underset{\bbDelta}{\lim}\ \CG(\cC g)
	\simeq \CG(\scY_c)\,.
\end{align}
The first equivalence is simply the adjunction $\scY_! \dashv \scY^*$, the second equivalence is Lemma~\ref{st:mps out of h_! and h_*}(2), which applies since $\CG$ turns colimits into limits.
This property also provides the third equivalence.
The final equivalence arises from Lemma~\ref{st:char of wtau-oo-sheaves}, as $\CG$ is a $\wtau$-sheaf.

We have thus shown that $\scY^*$ and $\scY_*$ restrict to functors between $\sfSh(\sfC,\tau)$ and $\sfSh_!(\wsfC,\wtau)$.
Since $\sfSh_!(\wsfC,\wtau) \subset \sfPSh_!(\wsfC)$ and $\sfSh(\sfC,\tau) \subset \sfPSh(\sfC)$ are full subcategories, the natural equivalences $\scY^* \scY_* \to 1$ and $1 \to \scY_* \scY^*$ from part~(1) establish that these restrictions of $\scY^*$ and $\scY_*$ furnish an equivalence as claimed.
\end{proof}

\begin{corollary}
If $\tau_0$ is the trivial topology on $\sfC$ (whose coverings are generated solely by the identity morphisms), then $\sfSh_!(\wsfC,\wtau_0) \simeq \sfSh(\sfC, \tau_0) \simeq \sfPSh(\sfC) \simeq \sfPSh_!(\wsfC)$.
\end{corollary}

\begin{proof}
This follows from Theorem~\ref{st:image of h_*}(1) since pullbacks of equivalences in a quasi-category are again equivalences.
\end{proof}

\begin{remark}
In particular, if $\CG \in \sfPSh(\wsfC)$ is colimit-preserving, then it satisfies $\wtau_0$-descent.
One can also see this directly:
the $\wtau_0$-local epimorphisms are exactly the effective epimorphisms in $\wsfC$, i.e.~those morphisms $p \colon Y \to X$ whose induced morphism $|\cC p| \to X$ is an equivalence in $\wsfC$.
Thus, for $\CG \in \sfPSh_!(\sfC)$, we have that $\CG(X) \to \CG(|\cC p|) \simeq \lim_\bbDelta \CG(\cC p)$ is an equivalence, and hence that $\CG$ satisfies $\wtau_0$-descent.
\qen
\end{remark}

We summarise various characterisations of the sheaves in the image of $\scY_*$:
Let \smash{$\ul{\Hom}_{! W_\tau}(\wsfC^\opp, \sfS_\rmL) \subset \sfPSh_!(\wsfC)$} denote the full sub-quasi-category on those limit-preserving functors \smash{$\wsfC^\opp \to \sfS_\rmL$} (equivalently \emph{colimit}-preserving functors \smash{$\wsfC \to \sfS_\rmL^\opp$}) which send all morphisms in $W_\tau$ to equivalences.

\begin{theorem}
There are equivalences of quasi-categories
\begin{equation}
\begin{tikzcd}
	\sfSh_!(\wsfC,\wtau) \ar[r, shift left=0.075cm, "\scY^*"]
	& \sfSh(\sfC,\tau) \ar[l, shift left=0.075cm, "\scY_*"] \ar[r, "\wY"]
	& \sfPSh_! \big( \sfSh_\rmS(\sfC,\tau) \big) \ar[r, "\sfL^*"]
	& \ul{\Hom}_{! W_\tau}(\wsfC^\opp, \sfS_\rmL)\,.
\end{tikzcd}
\end{equation}
\end{theorem}

\begin{proof}
The first equivalence is Theorem~\ref{st:image of h_*}.
The second equivalence follows from~\cite[Cor.~4.4.5.16]{Lurie:HTT} and~\cite[Cor.~5.1.6.11]{Lurie:HTT} (note that $\wY \colon \sfSh(\sfC, \tau) \to \sfPSh_!(\sfSh(\sfC,\tau))$, with the codomain restricted to colimit-preserving presheaves, is colimit-preserving)%
\footnote{We thank the anonymous referee for pointing out this argument, which streamlined the proof in an earlier version.}.
The third is an application of~\cite[Thm.~7.7.9]{Cisinski:Higher_Cats_and_Ho_Alg} (with the class of cofibrations given by all morphisms in $\wsfC$ and the weak equivalences given by $W_\tau$).
\end{proof}

\subsection{On a 1-site}

For many geometric applications it is more convenient to work not with $\infty$-presheaves on a quasi-category, but only with $\Set_\rmL$-valued presheaves on an ordinary category.
For instance, many geometric problems are set on manifolds or diffeological spaces (such as in Sections~\ref{sec:Dfg VBuns} and~\ref{sec:Descent and coherence}), and it is often more efficient (though less general) to treat these as particular presheaves of sets.
Thus, in this section we prove a version of Theorem~\ref{st:char of wtau-oo-sheaves} where $\scC$ is an ordinary (1-)category and we replace $\wsfC$ by the ordinary category $\wC = \Fun(\scC^\opp, \Set_\rmS)$ of $\Set_\rmS$-valued presheaves on $\scC$.
The quasi-category of $\sfS_\rmS$-valued presheaves on $\scC$ is $\widehat{\sfN\scC}_\infty = \ul{\sSet_\rmL}(\sfN\scC, \sfS_\rmS)$.

We let $\wC = \Fun(\scC^\opp, \Set_\rmS)$ denote the category of ordinary, small presheaves on $\scC$.
Let $\sfN\scC$ denote the quasi-category given by the nerve of $\scC$.
The Yoneda embedding of $\sfN\scC$ factorises as
\begin{equation}
\label{eq:Y = i Y^0 for Cats}
\begin{tikzcd}
	\scY \colon \sfN\scC \ar[r, "\sfN\scY^0"] & \sfN\wC \ar[r, hookrightarrow, "\iota"] & \widehat{\sfN\scC}_\infty\,,
\end{tikzcd}
\end{equation}
where $\scY^0$ is the 1-categorical Yoneda embedding of $\scC$ and \smash{$\widehat{\sfN\scC}_\infty = \ul{\sSet_\rmL}(\sfN\scC, \sfS_\rmS)$} is the quasi-category of small presheaves of spaces on $\sfN\scC$.
Note that $\wC \hookrightarrow \widehat{\sfN\scC}_\infty$ is a reflective localisation with localisation functor $\pi_0$.
As before, we write $\sfPSh(\sfN\scC) = \ul{\sSet_\XL}(\scC^\opp, \sfS_\rmL)$.

\begin{lemma}
\label{st:CarStef Lemma}
Let $F \in \sfPSh(\sfN\scC)$.
There is a natural equivalence $\scY^0_* F \simeq \iota^* \scY_* F$ of objects in $\sfPSh(\sfN\wC)$.
\end{lemma}

\begin{proof}
This is a direct application of~\cite[Lemma~3.13]{CS:Univ_Prop_Der_Mfds} (noting that $\sfN\wC$ is large and the target $\sfS_\rmL$ is $\rmL$-cocomplete).
\end{proof}

Applying the functor $\iota_*$, we obtain a natural equivalence $\iota_* \scY^0_* F \simeq \iota_* \iota^* \scY_* F$.

\begin{lemma}
\label{st:i_*i^* on P_!(C)}
For $\CG \in \sfPSh_!(\widehat{\sfN\scC}_\infty)$, there is an equivalence $\iota_* \iota^*\CG \simeq \CG$.
\end{lemma}

\begin{proof}
By Theorem~\ref{st:image of h_*}, there exists $F \in \sfPSh(\sfN\scC)$ and an equivalence $\scY_*F \simeq \CG$.
We have equivalences
\begin{align}
	\iota_* \iota^* \CG
	&\simeq \iota_* \iota^* \scY_* F
	\simeq \iota_* \scY^0_* F
	\simeq (\iota \circ \scY^0)_* F
	\simeq \scY_* F
	\simeq \CG\,.
\end{align}
The second equivalence is Lemma~\ref{st:CarStef Lemma} and the second-to-last equivalence is the factorisation~\eqref{eq:Y = i Y^0 for Cats}.
\end{proof}

In particular, we have that $\iota_* \iota^* \scY_*F \simeq \scY_*F$, for each $F \in \Fun(\sfN\scC^\opp, \sfS_\rmL)$, by Theorem~\ref{st:image of h_*}.
Combining this with Lemma~\ref{st:CarStef Lemma}, we thus obtain that the canonical morphism
\begin{equation}
	\iota_* \scY^0_*F \to \scY_*F
\end{equation}
is an equivalence, for each $F \in \sfPSh(\sfN \scC)$.

\begin{theorem}
\label{st:image of h^0_*}
Let $\sfX$ denote the full sub-quasi-category of $\sfPSh(\sfN\wC)$ on those presheaves $F$ where $\iota_* F \in \sfPSh_!(\widehat{\sfN\scC}_\infty)$.
The functor $\scY^0_* \colon \sfPSh(\sfN\scC) \to \sfX$ is an equivalence with inverse $(\scY^0)^*$.
\end{theorem}

\begin{proof}
Combining the observation preceding Theorem~\ref{st:image of h^0_*} with Theorem~\ref{st:image of h_*}(1) shows that $\scY^0_*$ takes values in $\sfX$.
Since $\scY^0$ is fully faithful, so is $\scY^0_*$; it remains to check that $\scY^0_*$ is essentially surjective.
To that end, let $G \in \sfX$.
By Theorem~\ref{st:image of h_*}(1) and the definition of $\sfX$ there exists an $F \in \sfPSh(\sfN\scC)$ and an equivalence $\iota_*G = \scY_*F$.
Lemma~\ref{st:CarStef Lemma} implies that
\begin{equation}
	G \simeq \iota^* \iota_* G
	\simeq \iota^* \scY_* F
	\simeq \scY^0_* F\,.
\end{equation}
The first equivalence is (the inverse of) the counit of the adjunction $\iota^* \dashv \iota_*$; this is an equivalence since $\iota$---and therefore $\iota_*$---is fully faithful.
\end{proof}

\begin{corollary}
\label{st:sheaves on 1-site}
The functor $\iota_* \colon \sfX \to \sfPSh_!(\widehat{\sfN\scC}_\infty)$ is an equivalence with inverse $\iota^*$.
That is, presheaves in $\sfPSh_!(\widehat{\sfN\scC}_\infty)$ are fully determined by their restriction to $\sfN\wC$.
\end{corollary}

This merely reflects the fact that $\widehat{\sfN\scC}_\infty$ is generated under small colimits by the image of the inclusion $\sfN\wC \hookrightarrow \widehat{\sfN\scC}_\infty$.

\begin{theorem}
\label{st:image of Y^0_*}
Let $\sfX_\tau \subset \sfPSh(\sfN\wC)$ be the full sub-quasi-category on those presheaves $X$ where $\iota_*X \in \sfSh_!(\widehat{\sfN\scC}_\infty,\wtau)$.
There is a canonical equivalence $\sfSh(\scC,\tau) \simeq \sfX_\tau$.
\end{theorem}

\begin{proof}
Consider the diagram
\begin{equation}
\begin{tikzcd}
	\sfX_\tau \ar[r] \ar[d] & \sfX \ar[r] \ar[d, "\iota_*"] & \sfPSh(\sfN\wC) \ar[d, "\iota_*"]
	\\
	\sfSh_!(\widehat{\sfN\scC}_\infty,\wtau) \ar[r] & \sfPSh_!(\widehat{\sfN\scC}_\infty) \ar[r] & \sfPSh(\widehat{\sfN\scC}_\infty)
\end{tikzcd}
\end{equation}
By definition of $\sfX$ and $\sfX_\tau$ and the pasting law for pullbacks, both squares in this diagram are pullback squares.
We also consider the diagram
\begin{equation}
\begin{tikzcd}
	\sfSh(\scC,\tau) \ar[r] \ar[d] & \sfPSh(\scC) \ar[r] \ar[d, "\scY_*"] & \sfPSh(\sfN\wC) \ar[d, "\iota_*"]
	\\
	\sfSh_!(\widehat{\sfN\scC}_\infty,\wtau) \ar[r] & \sfPSh_!(\widehat{\sfN\scC}_\infty) \ar[r] & \sfPSh(\widehat{\sfN\scC}_\infty)
\end{tikzcd}
\end{equation}
It suffices to show that the outer square in this diagram is a pullback.
Theorem~\ref{st:image of h^0_*} states that the right-hand square in this diagram is a pullback square of quasi-categories.
Additionally, Theorem~\ref{st:image of h_*} implies that the left and centre vertical morphisms are equivalences of quasi-categories, so that the left-hand square is a pullback as well.
The claim now follows by the pasting law for pullbacks.
\end{proof}

This provides a very general perspective on the statement that open coverings and surjective submersions define the same quasi-categories of sheaves on the category $\Mfd$ of manifolds:

\begin{corollary}
Consider the case $\scC = \Mfd$.
Let $\tau_{dgop}$ be the topology on $\Mfd$ induced by differentiably good open coverings, $\tau_{op}$ the topology induced by open coverings and $\tau_{ssub, \sqcup}$ the topology induced by surjective submersions of manifolds and open coverings by \emph{disjoint} open subsets (i.e.~by connected components).
The canonical functors
\begin{equation}
\begin{tikzcd}
	\sfSh (\sfN \Mfd, \tau_{ssub, \sqcup}) \ar[r, hookrightarrow]
	& \sfSh(\sfN \Mfd, \tau_{op}) \ar[r, hookrightarrow]
	& \sfSh(\sfN \Mfd, \tau_{dgop})
\end{tikzcd}
\end{equation}
are equivalences.
\end{corollary}

\begin{proof}
It is clear that second inclusion exists.
For the first inclusion, we need to check that, for each $F \in \sfSh(\sfN \Mfd, \tau_{ssub, \sqcup})$, we also have that $F \in \sfSh(\sfN\Mfd, \tau_{op})$.
Note that for $F \colon \sfN\Mfd^\opp \to \sfS_\rmL$ to satisfy descent with respect to coverings by disjoint open subsets means precisely that $F$ sends coproducts in $\Mfd$ to products in $\sfS_\rmL$.
Let $\CU = \{U_i \hookrightarrow M\}_{i \in I}$ be an open covering of a manifold $M$.
Consider the surjective submersion
\begin{equation}
	\cC_\Mfd \CU_0 = \coprod_{i \in I} U_i \longrightarrow M\,.
\end{equation}
Its \v{C}ech nerve, taken in $\Mfd$, reads as
\begin{equation}
	\cC_\Mfd \CU_k = \coprod_{i_0, \ldots, i_k \in I} U_{i_0 \ldots i_k}\,.
\end{equation}
Note that the canonical map $\scY^0_M \sqcup \scY^0_N \longrightarrow \scY^0_{M \sqcup N}$ is not an isomorphism in \smash{$\widehat{\Mfd}$}, since manifolds are allowed to be non-connected.
However, for $F \colon \Mfd^\opp \to \sfS_\rmL$ sending coproducts in $\Mfd$ to products, the canonical morphism
\begin{equation}
	\sfPSh(\Mfd) \Big( \coprod_{i_0, \ldots, i_k \in I} \scY_{U_{i_0 \ldots i_k}}, F \Big)
	\longrightarrow \sfPSh(\Mfd)(\scY_{\cC_\Mfd \CU_k}, F)
\end{equation}
is an equivalence, for each $k \in \NN_0$.
The left-hand side are the terms in the descent condition for $F$ with respect to the open covering $\CU$, whereas the right-hand side are the terms in the descent condition for $F$ with respect to the surjective submersion $\coprod_{i \in I} U_i \to M$.
Since we assumed $F$ to satisfy descent with respect to surjective submersions, the latter descent condition is satisfied.
By the above observation, it thus follows that $F \in \sfSh(\sfN\Mfd, \tau_{op})$.

We are left to check that the inclusions in the claim are in fact essentially surjective.
For the first inclusion, we claim that $\wtau_{ssub, \sqcup} \subset \wtau_{op}$.
Indeed, let $\pi \colon Y \to X$ be a $\tau_{ssub, \sqcup}$-local epimorphism in \smash{$\widehat{\Mfd}$}.
Therefore, given any morphism $f \colon \scY^0_M \to X$ from a representable presheaf, there exists either an open cover by disjoint subsets of $M$ and lifts of $f$ to $Y$ over the patches, or a surjective submersion $p \colon N \to M$ together with a commutative square
\begin{equation}
\begin{tikzcd}
	\scY^0_N \ar[r, "\hat{f}"] \ar[d, "p"']
	& Y \ar[d, "\pi"]
	\\
	\scY^0_M \ar[r, "f"']
	& X
\end{tikzcd}
\end{equation}
in \smash{$\widehat{\Mfd}$}.
In the first case, there is nothing to prove, since we have already lifted $f$ through an open covering of $M$.
In the second case, let $\CU = \{U_i \to M\}_{i \in I}$ be an open covering of $M$ such that $p$ admits a section $s_i \colon U_i \to N$ over each patch $U_i$.
This gives, for each $i \in I$, a commutative diagram
\begin{equation}
\begin{tikzcd}
	\scY^0_{U_i} \ar[r, "\scY^0_{s_i}"] \ar[d]
	& \scY^0_N \ar[r, "\hat{f}"] \ar[d, "p"']
	& Y \ar[d, "\pi"]
	\\
	\scY^0_M \ar[r, equal]
	& \scY^0_M \ar[r, "f"']
	& X
\end{tikzcd}
\end{equation}
Thus, we have that $\pi \in \wtau_{op}$.
An analogous argument shows that $\wtau_{op} \subset \wtau_{dgop}$.
We thus have canonical inclusions
\begin{equation}
\begin{tikzcd}
	\sfSh_!(\widehat{\sfN\Mfd}_\infty, \wtau_{dgop}) \ar[r, hookrightarrow]
	& \sfSh_!(\widehat{\sfN\Mfd}_\infty, \wtau_{op}) \ar[r, hookrightarrow]
	& \sfSh_!(\widehat{\sfN\Mfd}_\infty, \wtau_{ssub, \sqcup})\,.
\end{tikzcd}
\end{equation}
Recalling the definition of $\sfX_\tau \subset \sfPSh(\sfN\scC)$ as the preimage under $\iota_*$ of $\sfSh_!(\widehat{\sfN\scC}_\infty, \wtau) \subset \sfPSh(\widehat{\sfN\scC}_\infty)$, we obtain inclusions
\begin{equation}
\begin{tikzcd}
	\sfX_{\tau_{dgop}} \ar[r, hookrightarrow]
	& \sfX_{\tau_{op}} \ar[r, hookrightarrow]
	& \sfX_{\tau_{ssub, \sqcup}}\,.
\end{tikzcd}
\end{equation}
By Theorem~\ref{st:image of Y^0_*} this is equivalently a chain of inclusions
\begin{equation}
\begin{tikzcd}
	\sfSh(\sfN\Mfd, \tau_{dgop}) \ar[r, hookrightarrow]
	& \sfSh(\sfN\Mfd, \tau_{op}) \ar[r, hookrightarrow]
	& \sfSh(\sfN\Mfd, \tau_{ssub, \sqcup})\,.
\end{tikzcd}
\end{equation}
This completes the proof.
\end{proof}

\begin{appendix}

\section{Proof of Theorem~\ref{st:VBun_Cat = VBun_Dfg}}
\label{app:Proof of VBun_Dfg Thm}

We will make use of the following description of colimits in $\Dfg$:

\begin{proposition}
\label{st:colims in Dfg via ULSets}
Let $(\scC,\tau)$ be a \emph{closed} site (cf.~Definition~\ref{def:closed site}), and let $D \colon \scJ \to \Dfg(\scC,\tau)$ be a small diagram.
The  colimit of $D$ can be described as follows:
its underlying set reads as
\begin{equation}
	\Ev_* \big( \colim^{\Dfg(\scC,\tau)}_\scJ D \big) = \colim^{\Set_\rmS}_\scJ (\Ev_* \circ D)\,.
\end{equation}
A plot of \smash{$\colim^{\Dfg(\scC,\tau)}_\scJ D$} is a map of sets $\varphi \colon \scY_c(*) \to \colim^{\Set_\rmS}_\scJ (\Ev_* \circ D)$ such that there exists a covering $\{f_i \colon c_i \to c\}_{i \in I}$ in $(\scC, \tau)$, a map $i \mapsto j_i$ from $I$ to the objects of $\scJ$, and a family of maps $\{\varphi_i \colon \scY_{c_i}(*) \to D(j_i)(*)\}_{i \in I}$ such that the diagram
\begin{equation}
\label{eq:colim^Dfg diagram}
\begin{tikzcd}[column sep=1.25cm, row sep=1cm]
	\scY_{c_i}(*) \ar[r, "\varphi_i"] \ar[d, "{\scY_{f_i}(*)}"'] & D(j_i)(*) \ar[d]
	\\
	\scY_c(*) \ar[r] & \colim^{\Set_\rmS}_\scJ (\Ev_* \circ D)
\end{tikzcd}
\end{equation}
in $\Set_\rmS$ commutes for every $i \in I$, and such that $\varphi_i$ is a plot of $D(j_i)$ for every $i \in I$.
\end{proposition}

\begin{proof}
Let $Z(*) \coloneqq \colim^{\Set_\rmS}_\scJ (\Ev_* \circ D)$.
For $c \in \scC$, let $Z(c)$ denote the set of maps $\varphi \colon \scY_c(*) \to Z(*)$ with the above lifting property.
We first show that $c \mapsto Z(c)$ defines a presheaf on $\scC$.
Thus, we consider a morphism $f \in \scC(c',c)$; it induces a map $f_{|*} \colon \scY_{c'}(*) \to \scY_c(*)$.
The morphism $Z(f) \colon Z(c) \to Z(c')$ is given by $\varphi \mapsto \varphi \circ f_{|*}$.
Hence, we need to show that for any $\varphi \in Z(c)$ and $f \in \scC(c',c)$, the composition $\varphi \circ f_{|*}$ again has the lifting property.
This follows readily from the factorisation property of covering families; see Definition~\ref{def:coverage and site}(1).

The presheaf $Z$ is concrete since we have constructed the value $Z(c)$ as a subset of $\Set_\rmS(\scY_c(*),Z(*))$ and since constant maps to $Z(*)$ trivially have the local lifting property.

Next, we show that $Z$ is a sheaf.
To that end, suppose that $\{c_i \to c\}_{i \in I}$ is a covering family for $c$ and that we are given morphisms $\{\varphi_i \colon \scY_{c_i} \to Z\}_{i \in I}$ such that the diagram
\begin{equation}
\begin{tikzcd}
	\scY_{c_i} \underset{\scY_c}{\times} \scY_{c_j} \ar[d] \ar[r] & \scY_{c_i} \ar[d, "\varphi_i"]
	\\
	\scY_{c_j} \ar[r, "\varphi_j"'] & Z
\end{tikzcd}
\end{equation}
commutes for every $i,j \in I$.
Since evaluation at any object of $\scC$ is a limit-preserving functor $\wC \to \Set_\rmS$, these data induce a family of maps $\{\scY_{c_i}(*) \to \scY_c(*)\}_{i \in I}$ and maps $\{\varphi_{i|*} \colon \scY_{c_i}(*) \to Z(*)\}_{i \in I}$ such that the diagram
\begin{equation}
\begin{tikzcd}
	\scY_{c_i}(*) \underset{\scY_c(*)}{\times} \scY_{c_j}(*) \ar[d] \ar[r] & \scY_{c_i}(*) \ar[d, "\varphi_{i|*}"]
	\\
	\scY_{c_j}(*) \ar[r, "\varphi_{j|*}"'] & Z(*)
\end{tikzcd}
\end{equation}
in $\Set_\rmS$ commutes for every $i,j \in I$.
Since $(\scC,\tau)$ is a concrete site, the family $\{\scY_{c_i}(*) \to \scY_c(*)\}_{i \in I}$ is jointly surjective, so that these data determine a unique map $\varphi \colon \scY_c(*) \to Z(*)$ such that all diagrams
\begin{equation}
\begin{tikzcd}
	\scY_{c_i}(*) \ar[r] \ar[dr, "\varphi_{i|*}"'] & \scY_c(*) \ar[d, "\varphi"]
	\\
	& Z(*)
\end{tikzcd}
\end{equation}
commute.
We claim that $\varphi \in Z(c)$.
First, for each $i \in I$, let $\{c_{i,k} \to c_i\}_{k \in K_i}$ be a covering family for $c_i$ such that there exist lifts $\varphi_{i,k} \colon \scY_{c_i} \to D(j_{i,k})$ of the morphisms $\varphi_i \colon \scY_{c_i} \to Z$ as in~\eqref{eq:colim^Dfg diagram}.
By the assumption that $(\scC,\tau)$ is closed we can choose a covering family $\{c_l \to c\}_{l \in L}$ of $c$ such that each morphism $c_l \to c$ factors through some morphism $c_{i,k} \to c_i \to c$.
The compositions $c_l \to c_{i,k} \to D(j_{i,k})$ then provide the desired lifts of $\varphi$.

Our next step is to make $Z$ into the vertex of a cocone under the diagram $D$.
Consider the map $\iota_j \colon D(j)(*) \to Z(*)$ induced by the definition of $Z(*)$ as the colimit in $\Set_\rmS$ of the diagram $\Ev_* \circ D$.
This map induces a morphism in $\Dfg(\scC,\tau)$:
explicitly, given any plot $\scY_c(*) \to D(j)(*)$, the composition $\scY_c(*) \to Z(*)$ trivially admits a lifting along a covering family of $c$ to maps to $D(j)(*)$.
Thus, composition by $\iota_j$ sends plots of $D(j)$ to plots of $Z$.

Finally, we need to show that $Z$ is a colimit of the diagram $D \colon \scJ \to \Dfg(\scC,\tau)$.
To see this, let $\{\psi_j \colon D(j) \to A\}_{j \in \scJ}$ be a cocone under $D$ in the category $\Dfg(\scC,\tau)$.
Evaluating at $* \in \scC$, we obtain a cocone $\psi_{j|*} \colon D(j)(*) \to A(*)$ under the diagram $\Ev_* \circ D$ in $\Set_\rmS$.
By construction, the set $Z(*)$ presents a colimit of that diagram; hence these data induce a unique map of sets $\psi \colon Z(*) \to A(*)$.
Recall that for any pair of objects $X,Y \in \Dfg(\scC,\tau)$ the map $\Ev_* \colon \Dfg(\scC,\tau)(X,Y) \to \Set_\rmS(X(*),Y(*))$, $\phi \mapsto \phi_{|*}$ is injective.
It follows that if the map $\psi$ gives rise to a morphism of $(\scC,\tau)$-spaces, then that morphism is the unique morphism in $\Dfg(\scC,\tau)$ inducing a morphism of cocones under $D$.

Therefore, we are left to show that $\psi \colon Z(*) \to A(*)$ gives rise to a morphism in $\Dfg(\scC,\tau)$.
That is equivalent to showing that composition by $\psi$ sends plots of $Z$ to plots of $A$.
Thus, let $\varphi \colon \scY_c \to Z$ be an arbitrary morphism.
As before, by definition of $Z$, we find a covering family $\{f_i \colon c_i \to c\}_{i \in I}$ and lifts $\varphi_i \colon \scY_{c_i} \to D(j_i)$ of $\varphi$ along the morphisms $f_i$ as in~\eqref{eq:colim^Dfg diagram}.
We claim that $\{\psi_{j_i} \circ \varphi_i\}_{i \in I}$ is a compatible family of morphisms $\scY_{c_i} \to A$ in $\wC$.
For $i,k \in I$, consider the diagram
\begin{equation}
\label{eq:colim^Dfg construction}
\begin{tikzcd}[column sep=1.5cm, row sep=1cm]
	& \scY_{c_i} \ar[d, "f_i"] \ar[r, "\varphi_i"] & D(j_i) \ar[dr, bend left=20, "\psi_{j_i}"] \ar[d, "\iota_{j_i}"']
	\\
	\scY_{c_i} \underset{\scY_c}{\times} \scY_{c_k} \ar[ur] \ar[dr] & \scY_c \ar[r, "\varphi"] & Z \ar[r, dashed, "\psi"] & A
	\\
	& \scY_{c_k} \ar[u, "f_k"'] \ar[r, "\varphi_k"'] & D(j_k) \ar[ur, bend left=-20, "\psi_{j_k}"'] \ar[u, "\iota_{j_k}"]
\end{tikzcd}
\end{equation}
in $\wC$.
The left-hand triangle commutes by definition of the pullback.
The two central squares commute by definition of $\varphi_i$.
We do not yet know whether $\psi$ is actually a morphism in $\wC$.
However, evaluating the whole diagram at $* \in \scC$, we obtain a diagram in $\Set_\rmS$ in which the two right-hand triangles also commute, since $\psi$ is a morphism of cocones under the diagram $\Ev_* \circ D$ in $\Set_\rmS$.
Thus, we infer that the outer hexagon in~\eqref{eq:colim^Dfg construction} commutes as maps on the underlying sets.
Recalling that the map which sends morphisms in $\Dfg(\scC,\tau)$ to maps of underlying sets is injective, it thus follows that the outer hexagon is commutative already as a diagram in $\Dfg(\scC,\tau)$ (i.e.~before evaluating at the terminal object).

Consequently, $\{\psi_{j_i} \circ \varphi_i\}_{i \in I}$ is indeed a compatible family of morphisms $\scY_{c_i} \to A$.
Thus, as $A$ is a sheaf, it defines a unique morphism $\varrho \colon \scY_c \to A$ in $\Dfg(\scC,\tau)$.
It follows from the construction of $\varrho$ that $\Ev_* \varrho = \varrho_{|*} = \psi \circ \varphi_{|*}$; hence, composition with $\psi$ sends plots of $Z$ to plots of $A$, so that $\psi$ gives rise to a morphism in $\Dfg(\scC,\tau)$.
\end{proof}

We start by describing the functor $\CA$ on objects:
recall the computation of the values of $\iota^*\VBun_{str}$ in Lemma~\ref{st:VBun_str as hoRan}.
Consider the category $\Cart_{/X}$, whose objects are plots $\varphi \in X(c)$ for any $c \in \Cart$ and whose morphisms $(\varphi \in X(c')) \to (\varphi' \in X(c))$ are smooth maps $f \colon c \to c'$ such that $f^*\varphi' = \varphi$.
Given an object $(n,h) \in \iota^* \VBun_{str}(X)$ we define a functor $D_{(n,h)} \colon \Cart_{/X} \to \Dfg$, which acts as
\begin{align}
	(\varphi \colon c \to X) &\longmapsto c \times \CN^n\,,
	\\
	(f \colon \varphi \to \varphi') &\longmapsto \big( (f,h_f) \colon c \times \CN^n \longrightarrow c' \times \CN^n \big)\,.
\end{align}
We then set
\begin{equation}
	E \coloneqq \underset{\Cart_{/X}}{\colim}^\Dfg\, D_{(n,h)}
\end{equation}
and denote the canonical morphism $D_{(n,h)}(\varphi) \to E$ by $\iota_\varphi$.
The object $X \in \Dfg$ is a cocone under $D_{(n,h)}$, which is established by the morphism of diagrams
\begin{equation}
	\widehat{\pi} \colon D_{(n,h)} \to X\,,
	\qquad
	\widehat{\pi}_{|\varphi} \colon D_{(n,h)}(\varphi) = c \times \CN^{n} \overset{\pr}{\longrightarrow} c \overset{\varphi}{\longrightarrow} X\,.
\end{equation}
We let $\pi \colon E \to X$ denote the unique morphism induced on the colimit.

\begin{definition}
\label{def:CA on objects}
In the above situation, we define the diffeological space over $X$,
\begin{equation}
	\CA(n,h) \coloneqq (E,\pi)\,.
\end{equation}
\end{definition}

Lemma~\ref{st:colims in Dfg via ULSets} implies that at the level of the underlying sets we have
\begin{align}
\label{eq:colim D_E ULSet}
	\Ev_* \Big( \underset{\Cart_{/X}}{\colim}^\Dfg\, D_{(n,h)} \Big)
	\cong \underset{\Cart_{/X}}{\colim}^\Dfg\, (\Ev_* \circ D_{(n,h)})
	&= \Big( \coprod_{\varphi \in X(c)} c \times \CN^{n} \Big) \Big/ {\sim}\,,
\end{align}
where $\sim$ is the equivalence relation generated by setting $(\varphi,y,v) \sim (\varphi',y',v')$ if there exists a morphism $f \colon \varphi \to \varphi'$ in $\Cart_{/X}$ such that $h_f (\varphi,y,v) = (\varphi',y',v')$.
Here, we use the convention $(\varphi, y, v) \in X(c) \times c \times \CN^{n}$.
The morphism $\pi$ sends an equivalence class $[\varphi,y,v]$ to $\varphi(y)$.
Since the morphisms $h$ act via linear isomorphisms of vector spaces, the fibre of $\pi$ carries a canonical $\CN$-vector space structure.

\begin{proposition}
\label{st:CA produces Dfg VBuns}
For each object $(n,h) \in \iota^* \VBun_{str}(X)$, the colimit $\CA(n,h) = (E,\pi)$ is a diffeological vector bundle on $X$,
\end{proposition}

\begin{proof}
To see this, let $\psi \in X(d)$ be a plot of $X$ over $d \in \Cart$.
Consider the morphism of diffeological spaces
\begin{equation}
\label{eq:triv of A(n,h) over c}
\begin{tikzcd}
	\Phi_\psi \colon d \times \CN^n \ar[r, "1"] & D_{(n,h)}(\psi) \ar[r, "\pr_d \times \iota_\psi"] & d \times_X E\,.
\end{tikzcd}
\end{equation}
Note that $\Phi_\psi$ is linear on the fibres.
We need to show that it is an isomorphism.
For $x \in X(*)$ and $c \in \Cart$, let $\sfc_x \colon \pt \to X(*)$ denote the constant plot of $X$ with value $x$.
Every element $(y, [\varphi, z,v]) \in (d \times_X E)(*)$ has a unique representative of the form
\begin{equation}
	\big( y, [\varphi, z,v] \big) = \big( y, [\psi, y, w] \big)\,.
\end{equation}
This is established by the pairs $(z,h_z)$ and $(y,h_y)$ (from the data of $(n,h)$) associated to the morphism $z \colon \sfc_{\psi(y)} \to \varphi$ and $y \colon \sfc_{\psi(y)} \to \psi$ in $\Cart_{/X}$, i.e.~to the commutative diagram
\begin{equation}
\begin{tikzcd}[column sep=1.25cm, row sep=0.5cm]
	& d \ar[dr, "\psi"] &
	\\
	\pt \ar[ur, hookrightarrow, "y"] \ar[dr, hookrightarrow, "z"'] \ar[rr, "\sfc_{\psi(y)} = \sfc_{\varphi(z)}" description] & & X
	\\
	& c \ar[ur, "\varphi"']
\end{tikzcd}
\end{equation}
in $\Dfg$.
The representative is well-defined by the composition properties of the pairs $(f,h)$ in the data of $E$.
We now define a map of sets
\begin{equation}
	\Psi_\psi \colon (d \times_X E)(*) \longrightarrow d \times \CN^n\,,
	\qquad
	\big( y, [\psi, y, w] \big) \longmapsto (y,w)\,.
\end{equation}
We readily see that the map $\Psi_\psi$ thus defined is an inverse for $\Phi_\psi$ as maps of sets.
It remains to prove that $\Psi_\psi$ is a morphism of diffeological spaces; that is, we need to show that for any plot $\varrho \colon c \to d \times_X E$, the composition $\Psi_\psi \circ \varrho \colon c \to c \times \CN^n$ is a plot.
A plot $\varrho \colon c \to d \times_X E$ is equivalently a pair of a smooth map $\varrho_d \colon c \to d$ and a plot $\varrho_E \colon c \to E$ such that $\psi \circ \varrho_d = \pi \circ \varrho_E$.
By Proposition~\ref{st:colims in Dfg via ULSets} there exists a covering $\{f_i \colon c_i \hookrightarrow c\}_{i \in I}$ of $c$ together with morphisms $\{\varrho_i \colon c_i \to D_{(n,h)}(\psi_i)\}_{i \in I}$ in $\Dfg$ such that $\varrho_E \circ f_i = \iota_{\psi_i} \circ \varrho_i$ for all $i \in I$.
Note that $\psi_i \colon e_i \to X$ are plots of $X$.
Using that $D_{(n,h)}(\psi_i) = e_i \times \CN^n$, we can further decompose $\varrho_i$ into pairs of smooth maps $\varrho_{i,e_i} \colon c_i \to e_i$ and $\varrho_{i, \CN^n} \colon c_i \to \CN^n$.
By construction, we thus obtain a commutative diagram
\begin{equation}
\begin{tikzcd}[column sep=1.5cm, row sep=1cm]
	c_i \ar[r, "\varrho_i"] \ar[d, "f_i"'] & d \times_X D_{(n,h)}(\psi_i) \ar[d, "1 \times \iota_{\psi_i}"]
	\\
	c \ar[r, "\varrho"'] & d \times_X E
\end{tikzcd}
\end{equation}
in $\Dfg$ for every $i \in I$.
Hence,
\begin{equation}
\begin{tikzcd}[column sep=1.25cm, row sep=0.6cm]
	& d \ar[dr, "\psi"] &
	\\
	c_i \ar[ur, "\varrho_d \circ f_i"] \ar[dr, "\varrho_{i,e_i}"'] \ar[r, "f_i" description] & c \ar[u, "\varrho_d"'] & X
	\\
	& e_i \ar[ur, "\psi_i"'] &
\end{tikzcd}
\end{equation}
commutes as well, for each $i \in I$.
Using the notation $h_f$ for morphisms as above, we compute the action of the composition $\Psi_\psi \circ \varrho \circ f_i = \Psi_\psi \circ (1 \times \iota_{\psi_i}) \circ \varrho_i \colon c_i \to d \times \CN^n$ on $y_i \in c_i$ as
\begin{align}
	y_i \longmapsto &\Psi_\psi \big( \varrho_d \circ f_i(y_i), [\psi_i,\, \varrho_{i,e_i}(y_i),\, \varrho_{i, \CN^n}(y_i)] \big)
	\\
	&= \Psi_\psi \big( \varrho_i (y_i), [\psi_i \circ \varrho_{i,e_i},\, y_i,\, h_{\varrho_{i,e_i}}^{-1} \circ \varrho_{i, \CN^n}(y_i)] \big)
	\\
	&= \Psi_\psi \big( \varrho_i (y_i), [\psi \circ \varrho_d \circ f_i,\, y_i,\, h_{\varrho_{i,e_i}}^{-1} \circ \varrho_{i, \CN^n}(y_i)] \big)
	\\
	&= \Psi_\psi \big( \varrho_i (y_i), [\psi,\, \varrho_d \circ f_i(y_i),\, h_{\varrho_d \circ f_i} \circ h_{\varrho_{i,e_i}}^{-1} \circ \varrho_{i, \CN^n}(y_i)] \big)
	\\
	&= \big( \varrho_i (y_i),\, h_{\varrho_d \circ f_i} \circ h_{\varrho_{i,e_i}}^{-1} \circ \varrho_{i, \CN^n}(y_i) \big)\,.
\end{align}
Both components of this map are smooth.
Thus, each composition $(\Psi_\psi \circ \varrho) \circ f_i$ is smooth.
Since, by construction, these maps agree on intersections $c_{ij}$, it follows that $\Psi_\psi \circ \varrho \colon c \to d \times \CN^n$ is smooth.
Therefore, $\Psi_\psi$ is a morphism of diffeological spaces, and $(E,\pi)$ is a diffeological vector bundle on $X$.
\end{proof}

We now define $\CA$ on morphisms:
let $g \colon (n_0, h_0) \to (n_1,h_1)$ be a morphism in $\iota^* \VBun_{str}(X)$.
It gives rise to a morphism of diagrams $D_g \colon D_{(n_0,h_0)} \to D_{(n_1,h_1)}$, and hence induces a morphism
\begin{equation}
	D_g \colon \underset{\Cart_{/X}}{\colim}^\Dfg D_{(n_0, h_0)}
	\longrightarrow \underset{\Cart_{/X}}{\colim}^\Dfg D_{(n_1, h_1)}
\end{equation}
of diffeological spaces.
By construction, this map is even a map of diffeological spaces over $X$.
Since for any plot $\varphi \colon c \to X$ the map $g_\varphi$ induces a morphism of vector bundles $c \times \CN^{n_0} \to c \times \CN^{n_1}$, the map $D_g$ is linear on each fibre.
Therefore, $D_g$ is indeed a morphism of diffeological vector bundles on $X$.

\begin{definition}
\label{def:CA on morphisms}
In the above situation, we set
\begin{equation}
	\CA(g) \coloneqq D_g \colon \CA(n_0,h_0) \to \CA(n_1,h_1)\,.
\end{equation}
\end{definition}

This completes the construction of the functor $\CA$.

\begin{proposition}
\label{st:CA is fully faithful}
For each $X \in \wCart$, the functor $\CA_{|X} \colon \iota^* \VBun_{str}(X) \longrightarrow \VBun_\Dfg(X)$ is fully faithful.
\end{proposition}

\begin{proof}
For $(n_0,h_0),\, (n_1,h_1) \in \iota^* \VBun_{str}(X)$, we write $E_i \coloneqq \colim_{\Cart_{/X}}^\Dfg D_{(n_i,h_i)}$, where $i = 0,1$.
Recall from~\eqref{eq:triv of A(n,h) over c} that for each plot $\varphi \colon c \to X$ we have constructed \emph{canonical} trivialisations
\begin{equation}
	\Phi^{E_i}_\varphi \colon c \times \CN^{n_i(\varphi)} \arisom c \times_X E_i\,.
\end{equation}
Thus, given a morphism $\xi \colon E_0 \to E_1$ in $\VBun_\Dfg(X)$ and a plot $\varphi \colon c \to X$, we obtain a morphism
\begin{equation}
\begin{tikzcd}[column sep=1.5cm, row sep=1cm]
	c \times \CN^{n_0(\varphi)} \ar[r, "\Phi^{E_0}_\varphi"] \ar[d, dashed, "\widetilde{\xi}_\varphi"']
	& c \times_X E_0 \ar[d, "1 \times_X \xi"]
	\\
	c \times \CN^{n_1(\varphi)}
	& c \times_X E_1 \ar[l, "{(\Phi^{E_1}_\varphi)^{-1}}"]
\end{tikzcd}
\end{equation}
which is linear on fibres; hence, $\widetilde{\xi}_\varphi$ corresponds to a smooth map $\widetilde{\xi}_\varphi \colon c \to \Mat(n_1 {\times} n_0,\CN)$.
Observe that for any morphism $f \colon \varphi_0 \to \varphi_1$ in $\Cart_{/X}$ the diagram
\begin{equation}
\begin{tikzcd}
	c_0 \times \CN^{n_i(\varphi_0)} \ar[r, "\Phi^{E_i}_{\varphi_0}"] \ar[d, "f \times 1"']
	& c_0 \times_X E_i \ar[d, "f \times_X 1"]
	\\
	c_1 \times \CN^{n_i(\varphi_1)} \ar[r, "\Phi^{E_i}_{\varphi_1}"']
	& c_1 \times_X E_i
\end{tikzcd}
\end{equation}
commutes (in this case, $n_0(\varphi_1) = n_0(\varphi_0)$).
This implies that $\widetilde{\xi}$ is a morphism $\widetilde{\xi} \colon (n_0,h_0) \to (n_1,h_1)$ in $\iota^* \VBun_{str}(X)$.

The map
\begin{equation}
\widetilde{(-)} \colon \VBun_\Dfg(X)(E_0,E_0) \longrightarrow \iota^* \VBun_{str}(X)\big( (n_0, h_0), (n_1,h_1) \big)\,,
\qquad
\xi \longmapsto \widetilde{\xi}
\end{equation}
is injective:
consider all constant plots $x \colon \pt \to X$, for $x \in X(*)$.
The set $\{\widetilde{\xi}_x\}_{x \in X(*)}$ uniquely determines the value of $\xi$ at every point $x \in X(*)$; hence it fully determines the morphism $\xi$.

Furthermore, if $g \colon (n_0, h_0) \to (n_1,h_1)$ is a morphism in $\iota^* \VBun_{str}(X)$, then, again by construction of $\Phi_\varphi$ from the cocone data, we have that
\begin{equation}
	\big( \widetilde{(-)} \circ \CA(g) \big)_\varphi = g_\varphi
\end{equation}
for all plots $\varphi \colon c \to X$.
Hence, the map $g \mapsto \CA(g)$ is a right inverse for $\widetilde{(-)}$.
It follows that $\widetilde{(-)}$ is a bijection, and hence that $\CA$ is fully faithful.
\end{proof}

\begin{proposition}
\label{st:CA is essentially surjective}
For each $X \in \wCart$, the functor $\CA_{|X} \colon \colon \iota^* \VBun_{str}(X) \longrightarrow \VBun_\Dfg(X)$ is essentially surjective.
\end{proposition}

\begin{proof}
Let $\pi_F \colon F \to X$ be a diffeological vector bundle on $X$.
For each plot $\varphi \colon c \to X$ choose a trivialisation
\begin{equation}
\label{eq:triv choice for F}
	\Xi_\varphi \colon c \times \CN^{n} \arisom c \times_X F\,.
\end{equation}
Given a morphism $f \colon \varphi_0 \to \varphi_1$ in $\Cart_{/X}$, the universal property and the pasting law for pullbacks determine a canonical isomorphism $c_0 \times_X F \cong c_0 \times_{c_1} (c_1 \times_X F)$.
This allows us to define an isomorphism
\begin{equation}
\begin{tikzcd}[column sep=1.25cm, row sep=1cm]
	c_0 \times \CN^{n(\varphi_0)} \ar[r, "\Xi_{\varphi_0}"] \ar[d, dashed, "h_f"']
	& c_0 \times_X F \ar[r, "\cong"]
	& c_0 \times_{c_1} (c_1 \times_X F) \ar[d, "f^*\Xi_{\varphi_1}^{-1}"]
	\\
	c_0 \times \CN^{n(\varphi_1)}
	&
	& c_0 \times_{c_1} (c_1 \times \CN^{n(\varphi_1)}) \ar[ll, "\cong"]
\end{tikzcd}
\end{equation}
The map $h_f$ is linear on fibres and hence determines a unique smooth map $h_f \colon c_0 \to \GL(n(\varphi_0),\CN)$.
Since the morphisms labelled `$\cong$' in the above diagram are chosen as canonical isomorphisms between different representatives for the same limits, the collection of morphisms $h_f$ assembles into an object $(n,h) \in \iota^* \VBun_{str}(X)$.

We claim that there is an isomorphism $\CA(n,h) \arisom F$ in $\VBun_\Dfg(X)$.
By a convenient abuse of notation, we denote this isomorphism by $\Xi$.
We set
\begin{equation}
	\Xi [\varphi,y,v] \coloneqq \pr_F \circ \Xi_\varphi(y,v)\,,
\end{equation}
where $\Xi_\varphi$ was chosen in~\eqref{eq:triv choice for F}, and where $\pr_F \colon c \times_X F \to F$ is the projection to $F$.
This defines a map $\Xi \colon \CA(n,h) \to F$, which is linear on fibres.
We need to show that it is smooth.
Consider a plot $\varrho \colon c \to \colim^\Dfg D_{(n,h)} = \CA(n,h)$.
By Proposition~\ref{st:colims in Dfg via ULSets}, there exist a covering $\{f_i \colon c_i \hookrightarrow c\}_{i \in I}$ of $c$ and lifts $\varrho_i \colon c_i \to D_{(n,h)}(\psi_i)$ for some plots $\psi_i \colon d_i \to X$.
Consider the diagram
\begin{equation}
\begin{tikzcd}[column sep=1.25cm, row sep=1cm]
	c_i \ar[r, "\varrho_i"] \ar[d, "f_i"'] & D_{(n,h)}(\psi_i) \ar[d, "\iota_{\psi_i}"'] \ar[rd, bend left=20, "\pr_F \circ \Xi_{\psi_i}"] &
	\\
	c \ar[r, "\varrho"'] & \colim^\Dfg D_{(n,h)} \ar[r, "\Xi"'] & F
\end{tikzcd}
\end{equation}
The left-hand square commutes by definition of $\varrho_i$, and the right-hand triangle commutes by construction of $\Xi$ and of $(n,h)$.
Thus, the map $\Xi \circ \varrho$ is locally given by plots $\pr_F \circ \Xi_{\psi_i} \circ \varrho_i$.
By the sheaf property of diffeological spaces, $\Xi \circ \varrho$ is a plot of $F$ itself.

Finally, we need to prove that $\Xi$ is an isomorphism.
As a map, it has an inverse, which is given by
\begin{equation}
	\Xi' \colon F \longrightarrow \colim^\Dfg D_{(n,h)}\,,
	\qquad
	\zeta \longmapsto \big[ \psi, y, \Xi_\psi^{-1}(y,\zeta) \big]\,,
\end{equation}
where $\psi \colon d \to X$ is some plot and $y \in d$ is any point such that $\psi(y) = \pi_F(\zeta)$.
It remains to show that $\Xi'$ is a morphism of diffeological spaces.
To that end, let $\widehat{\zeta} \colon d \to F$ be a plot.
This induces a plot $\psi \coloneqq \pi_F \circ \widehat{\zeta} \colon d \to X$.
Now, given any $y \in d$, we can write
\begin{equation}
	\Xi' \circ \widehat{\zeta} (y) = \big[ \psi, y, \Xi_\psi^{-1} (y,\widehat{\zeta}(y)) \big]
	= \iota_\psi \circ \Xi_\psi^{-1} (y,\widehat{\zeta}(y))
	= \iota_\psi \circ \Xi_\psi^{-1} \circ (1_d \times \widehat{\zeta})(y)\,.
\end{equation}
Thus, the composition $\Xi' \circ \widehat{\zeta}$ factors through a plot of $D_{(n,h)}(\psi)$, so that $\Xi'$ is a plot by Proposition~\ref{st:colims in Dfg via ULSets}.
\end{proof}

\begin{lemma}
$\CA$ is compatible with pullbacks of vector bundles along morphisms $F \colon X \to Y$ in $\Dfg$.
That is, $\CA \colon \iota^* \VBun_{str} \longrightarrow \VBun_\Dfg$ is a morphism of presheaves of categories on $\Dfg$.
\end{lemma}

\begin{proof}
Let $(n,h) \in \iota^* \VBun_{str}(Y)$.
As a diffeological space, we have
\begin{equation}
	F^* \big( \CA(n,h) \big)
	= F^* \big( \underset{\Cart_{/Y}}{\colim}^\Dfg D_{(n,h)} \big)
	= X \times_Y \big( \underset{\Cart_{/Y}}{\colim}^\Dfg D_{(n,h)} \big)
	= X \times_Y \CA(n,h)\,.
\end{equation}
We have to compare this to
\begin{equation}
	\CA \big( F^*(n,h) \big)
	= \underset{\Cart_{/X}}{\colim}^\Dfg \big( D_{F^*(n,h)} \big)\,.
\end{equation}
To that end, we consider the map
\begin{equation}
	\Xi_F \colon \CA \big( F^*(n,h) \big) \longrightarrow F^* \big( \CA(n,h) \big)\,,
	\qquad
	[\varphi,z,v] \longmapsto \big( \varphi(z), [ F \circ \varphi, z,v] \big)\,,
\end{equation}
where $\varphi \colon c \to X$ is a plot, $z \in c$ is some point, and where $v \in \CN^{n}$ is a vector.
We also define a map
\begin{equation}
	\Xi'_F \colon F^* \big( \CA(n,h) \big) \longrightarrow \CA \big( F^*(n,h) \big)\,,
	\qquad
	\big( x, [F(x), \pt, w] \big) \longmapsto [x, \pt, w]\,,
\end{equation}
where $x \in X(*)$ is any point in the underlying set of $X$, $w \in \CN^{n(F(x))}$ is a vector, and where we denote a constant plot $\pt \to X$ by its value in $X(*)$.
We readily see that $\Xi_F$ and $\Xi'_F$ are mutually inverse maps, fibrewise linear, compatible with morphisms in $\iota^* \VBun_{str}$, and that the diagram
\begin{equation}
\begin{tikzcd}[column sep=1.5cm, row sep=1cm]
	G^*F^* \big( \CA(n,h) \big) \ar[r, "\cong"] \ar[d, "G^*\Xi_F"']
	& (FG)^* \big( \CA(n,h) \big) \ar[d, "\Xi_{FG}"]
	\\
	G^* \big( \CA_{F^*(n,h)} \big) \ar[r, "\Xi_G"']
	& \CA_{G^*F^*(n,h)} = \CA_{(FG)^*(n,h)}
\end{tikzcd}
\end{equation}
commutes for every morphism $G \in \Dfg(W,X)$.
It thus remains to show that both $\Xi_F$ and $\Xi'_F$ are morphisms of diffeological spaces.

We start with $\Xi_F$:
let $\varrho \colon c \to \CA(F^*(n,h))$ be a plot.
Let $\{f_i \colon c_i \hookrightarrow c\}_{i \in I}$, $\psi_i \colon d_i \to X$, and $\varrho_i \colon c_i \to D_{F^*(n,h)}(\psi_i)$ be lifting data for $\varrho$ as before.
Then, we have a commutative diagram
\begin{equation}
\begin{tikzcd}[column sep=2.5cm, row sep=1.25cm]
	c_i \ar[r, "\varrho_i"] \ar[d, "f_i"']
	& D_{(n,h)}(\psi_i) \ar[d, "\iota_{\psi_i}"'] \ar[r, "{(\psi_i \circ \pr_{d_i}) \times_Y (1_{d_i \times \CN^n})}"]
	& X \times_Y \big( D_{(n,h)}(F \circ \psi_i) \big) \ar[d, "{1_X \times_Y (\iota_{F \circ \psi_i})}"]
	\\
	c \ar[r, "\varrho"'] & \CA_{F^*(n,h)} \ar[r, "\Xi_F"'] & X \times_Y \CA(n,h)
\end{tikzcd}
\end{equation}
which shows that $\Xi_F \circ \varrho$ is smooth (the $\CA(n,h)$-valued component factors through $D_{(n,h)}$ locally).

For $\Xi'_F$, consider a plot $\varrho \colon c \to X \times_Y \CA(n,h)$.
It decomposes into a plot $\varrho_X \colon c \to X$ and a plot $\varrho_\CA \colon c \to \CA_{(n,h)}$ such that $\pi \circ \varrho_\CA = F \circ \varrho_X$, where $\pi \colon \CA_{(n,h)} \to Y$ is the vector bundle projection.
As before, for the plot $\varrho_\CA$ there exists a covering $\{f_i \colon c_i \hookrightarrow c\}_{i \in I}$, plots $\psi \colon d_i \to Y$, and lifts $\varrho_{\CA,i} \colon c_i \to D_{(n,h)}(\psi_i)$ such that $\iota_{\psi_i} \circ \varrho_{\CA,i} = \varrho_\CA \circ f_i$ for each $i \in I$.
Using the morphisms $h_{\varrho_{\CA,i}}$, it is in fact always possible to choose $d_i = c_i$ and $\varrho_{\CA,i} \colon c_i \to c_i \times \CN^{n(\psi_i)}$ to be a section, i.e.~to satisfy $\pr_{c_i} \circ \varrho_{i,\CA} = 1_{c_i}$.
Observe that then $\psi_i = F \circ \varrho_X \circ f_i$ factors through $F$.
We obtain a commutative diagram
\begin{equation}
\begin{tikzcd}[column sep=2.5cm, row sep=1.25cm]
	c_i \ar[r, "{(\varrho_X \circ f_i) \times \varrho_{\CA,i}}"] \ar[d, "f_i"']
	& X \times_Y \big( D_{(n,h)}(\psi_i) \big) \ar[d, "1 \times \iota_{\psi_i}"'] \ar[r, "\pr_{c_i \times \CN^n}"]
	& D_{(n,h)}(\varrho_X \circ f_i) \ar[d, "\iota_{(\varrho_X \circ f_i)}"]
	\\
	c \ar[r, "\varrho_X \times \varrho_\CA"']
	& X \times_Y \CA(n,h) \ar[r, "\Xi'_F"']
	& \CA \big( F^* (n,h) \big)
\end{tikzcd}
\end{equation}
This shows that $\Xi'_F \circ \varrho$ is a plot of $\CA(F^*(n,h))$ by Proposition~\ref{st:colims in Dfg via ULSets}, which completes the proof.
\end{proof}

This also completes the proof of Theorem~\ref{st:VBun_Cat = VBun_Dfg}.
\qed

\end{appendix}

\begin{small}
\bibliographystyle{alphaurl}
\addcontentsline{toc}{section}{References}
\bibliography{Higher_Sheaves_Bib}

\begin{thebibliography}{KLW21}

\bibitem[Ati88]{Atiyah:TQFTs}
M.~Atiyah.
\newblock Topological quantum field theories.
\newblock {\em Inst. Hautes \'{E}tudes Sci. Publ. Math.}, (68):175--186, 1988.

\bibitem[Bar05]{Barwick:infty-n-Cat_as_closed_MoCat}
C.~Barwick.
\newblock {\em {$(\infty,n)$-Cat as a Closed Model Category}}.
\newblock PhD thesis, University of Pennsilvania, 2005.
\newblock URL: \url{https://repository.upenn.edu/dissertations/AAI3165639/}.

\bibitem[Bar10]{Barwick:Enriched_B-Loc}
C.~Barwick.
\newblock On left and right model categories and left and right {B}ousfield
  localizations.
\newblock {\em Homology Homotopy Appl.}, 12(2):245--320, 2010.
\newblock \href {https://arxiv.org/abs/0708.2067} {\path{arXiv:0708.2067}}.

\bibitem[BDR04]{BDR:2VBuns_and_ell_coho}
N.{\,}A. Baas, B.{\,}I. Dundas, and J.~Rognes.
\newblock Two-vector bundles and forms of elliptic cohomology.
\newblock In {\em Topology, geometry and quantum field theory}, volume 308 of
  {\em London Math. Soc. Lecture Note Ser.}, pages 18--45. Cambridge Univ.
  Press, Cambridge, 2004.
\newblock \href {https://arxiv.org/abs/math/0306027v1}
  {\path{arXiv:math/0306027v1}}.

\bibitem[BH11]{BH:Convenient_Cats_of_Smooth_Spaces}
J.{\,}C. Baez and A.{\,}E. Hoffnung.
\newblock Convenient categories of smooth spaces.
\newblock {\em Trans. Amer. Math. Soc.}, 363(11):5789--5825, 2011.
\newblock \href {https://arxiv.org/abs/0807.1704} {\path{arXiv:0807.1704}}.

\bibitem[BR13a]{BR:Comparison_for_oonCats_I}
J.{\,}E. Bergner and C.~Rezk.
\newblock Comparison of models for {$(\infty,n)$}-categories, {I}.
\newblock {\em Geom. Topol.}, 17(4):2163--2202, 2013.
\newblock \href {https://arxiv.org/abs/1204.2013v2} {\path{arXiv:1204.2013v2}}.

\bibitem[BR13b]{BR:Reedy_Cats_and_Theta_construction}
J.{\,}E. Bergner and C.~Rezk.
\newblock Reedy categories and the {$\varTheta$}-construction.
\newblock {\em Math. Z.}, 274(1-2):499--514, 2013.
\newblock \href {https://arxiv.org/abs/1110.1066v2} {\path{arXiv:1110.1066v2}}.

\bibitem[BR20]{BR:Comparison_for_oonCats_II}
J.{\,}E. Bergner and C.~Rezk.
\newblock Comparison of models for {$(\infty, n)$}-categories, {II}.
\newblock {\em J. Topol.}, 13(4):1554--1581, 2020.
\newblock \href {https://arxiv.org/abs/1406.4182v3} {\path{arXiv:1406.4182v3}}.

\bibitem[BSP11]{BSP:Unicity}
C.~Barwick and C.{\,}J. Schommer-Pries.
\newblock On the unicity of the homotopy theory of higher categories.
\newblock 2011.
\newblock \href {https://arxiv.org/abs/1112.0040v6} {\path{arXiv:1112.0040v6}}.

\bibitem[Bun17]{Bunk:Thesis}
S.~Bunk.
\newblock {\em Categorical structures on bundle gerbes and higher geometric
  prequantisation}.
\newblock PhD thesis, Heriot-Watt University, 2017.
\newblock \href {https://arxiv.org/abs/1709.06174v1}
  {\path{arXiv:1709.06174v1}}.

\bibitem[Bun22]{Bunk:R-loc_HoThy}
S.~Bunk.
\newblock The $\mathbb{R}$-local homotopy theory of smooth spaces.
\newblock {\em J. Homotopy Relat. Struct. (to appear)}, 2022.
\newblock \href {https://arxiv.org/abs/2207.14608v1}
  {\path{arXiv:2207.14608v1}}.

\bibitem[BV73]{BV:Ho-invar_alg_strs}
J.{\,}M. Boardman and R.{\,}M. Vogt.
\newblock {\em Homotopy invariant algebraic structures on topological spaces}.
\newblock Lecture Notes in Mathematics, Vol. 347. Springer-Verlag, Berlin-New
  York, 1973.

\bibitem[BW21]{BW:Smooth_OCFFTs}
S.~Bunk and K.~Waldorf.
\newblock Smooth functorial field theories from {B}-fields and {D}-branes.
\newblock {\em J. Homotopy Relat. Struct.}, 16(1):75--153, 2021.
\newblock \href {https://arxiv.org/abs/2207.14608v3}
  {\path{arXiv:2207.14608v3}}.

\bibitem[Cis19]{Cisinski:Higher_Cats_and_Ho_Alg}
D.-C. Cisinski.
\newblock {\em Higher categories and homotopical algebra}, volume 180 of {\em
  Cambridge Studies in Advanced Mathematics}.
\newblock Cambridge University Press, Cambridge, 2019.

\bibitem[CS19a]{CS:Bord_n}
D.~Calaque and C.~Scheimbauer.
\newblock {A Note on the $(\infty,n)$-Category of Bordisms}.
\newblock {\em Algebr. Geom. Topol.}, 19(2):533--655, 2019.
\newblock \href {https://arxiv.org/abs/1509.08906} {\path{arXiv:1509.08906}}.

\bibitem[CS19b]{CS:Univ_Prop_Der_Mfds}
D.~Carchedi and P.~Steffens.
\newblock On the universal property of derived manifolds.
\newblock 2019.
\newblock \href {https://arxiv.org/abs/1905.06195} {\path{arXiv:1905.06195}}.

\bibitem[DHI04]{DHI:Hypercovers_and_sPShs}
D.~Dugger, S.~Hollander, and D.{\,}C. Isaksen.
\newblock Hypercovers and simplicial presheaves.
\newblock {\em Math. Proc. Cambridge Philos. Soc.}, 136(1):9--51, 2004.
\newblock \href {https://arxiv.org/abs/math/0205027}
  {\path{arXiv:math/0205027}}.

\bibitem[Dug01]{Dugger:Universal_HoThys}
D.~Dugger.
\newblock Universal homotopy theories.
\newblock {\em Adv. Math.}, 164(1):144--176, 2001.
\newblock \href {https://arxiv.org/abs/math/0007070}
  {\path{arXiv:math/0007070}}.

\bibitem[FSS12]{FSS:Cech_diff_char_classes_via_L_infty}
D.~Fiorenza, U.~Schreiber, and J.~Stasheff.
\newblock \v{C}ech cocycles for differential characteristic classes: an
  {$\infty$}-{L}ie theoretic construction.
\newblock {\em Adv. Theor. Math. Phys.}, 16(1):149--250, 2012.
\newblock \href {https://arxiv.org/abs/1011.4735} {\path{arXiv:1011.4735}}.

\bibitem[Gam10]{Gambino:Weighted_limits_in_spl_HoThy}
N.~Gambino.
\newblock Weighted limits in simplicial homotopy theory.
\newblock {\em J. Pure Appl. Algebra}, 214(7):1193--1199, 2010.

\bibitem[GP20]{GP:Smooth_Geo_Bord}
D.~Grady and D.~Pavlov.
\newblock {Extended field theories are local and have classifying spaces}.
\newblock 2020.
\newblock \href {https://arxiv.org/abs/2011.01208} {\path{arXiv:2011.01208}}.

\bibitem[Hir03]{Hirschhorn:MoCats}
P.{\,}S. Hirschhorn.
\newblock {\em Model categories and their localizations}, volume~99 of {\em
  Mathematical Surveys and Monographs}.
\newblock American Mathematical Society, Providence, RI, 2003.

\bibitem[Hov99]{Hovey:MoCats}
M.~Hovey.
\newblock {\em Model categories}, volume~63 of {\em Mathematical Surveys and
  Monographs}.
\newblock American Mathematical Society, Providence, RI, 1999.

\bibitem[IZ13]{IZ:Diffeology}
P.~Iglesias-Zemmour.
\newblock {\em Diffeology}, volume 185 of {\em Mathematical Surveys and
  Monographs}.
\newblock American Mathematical Society, Providence, RI, 2013.

\bibitem[JFS17]{JFS:Oplax_twQFTs_and_Morita}
T.~Johnson-Freyd and C.~Scheimbauer.
\newblock ({O}p)lax natural transformations, twisted quantum field theories,
  and ``even higher'' {M}orita categories.
\newblock {\em Adv. Math.}, 307:147--223, 2017.
\newblock \href {https://arxiv.org/abs/1502.06526v3}
  {\path{arXiv:1502.06526v3}}.

\bibitem[Joh02]{Johnstone:Sketches_of_an_elephant_I}
P.{\,}T. Johnstone.
\newblock {\em Sketches of an elephant: a topos theory compendium. {V}ol. 1},
  volume~43 of {\em Oxford Logic Guides}.
\newblock The Clarendon Press, Oxford University Press, New York, 2002.

\bibitem[Joy02]{Joyal:QCats_and_Kan_complexes}
A.~Joyal.
\newblock Quasi-categories and {K}an complexes.
\newblock volume 175, pages 207--222. 2002.
\newblock Special volume celebrating the 70th birthday of Professor Max Kelly.

\bibitem[Kih20]{Kihara:Smooth_ho_of_oo-dim_Mfds}
H.~Kihara.
\newblock Smooth homotopy of infinite-dimensional {$C^\infty$}-manifolds.
\newblock {\em Mem. Amer. Math. Soc.}, to appear, 2020.
\newblock \href {https://arxiv.org/abs/2002.03618v1}
  {\path{arXiv:2002.03618v1}}.

\bibitem[KLW21]{KLW:2VBuns}
P.~Kristel, M.~Ludewig, and K.~Waldorf.
\newblock 2-vector bundles.
\newblock 2021.
\newblock \href {https://arxiv.org/abs/2106.12198v2}
  {\path{arXiv:2106.12198v2}}.

\bibitem[Koc04]{Kock:2D-TFTs}
J.~Kock.
\newblock {\em Frobenius algebras and 2{D} topological quantum field theories},
  volume~59 of {\em London Mathematical Society Student Texts}.
\newblock Cambridge University Press, Cambridge, 2004.

\bibitem[KS06]{KS:Cats_and_Sheaves}
M.~Kashiwara and P.~Schapira.
\newblock {\em Categories and sheaves}, volume 332 of {\em Grundlehren der
  Mathematischen Wissenschaften}.
\newblock Springer-Verlag, Berlin, 2006.

\bibitem[KV94]{KV:2-Cats_and_tetrahedra_equations}
M.{\,}M. Kapranov and V.{\,}A. Voevodsky.
\newblock {$2$}-categories and {Z}amolodchikov tetrahedra equations.
\newblock In {\em Algebraic groups and their generalizations: quantum and
  infinite-dimensional methods ({U}niversity {P}ark, {PA}, 1991)}, volume~56 of
  {\em Proc. Sympos. Pure Math.}, pages 177--259. Amer. Math. Soc., Providence,
  RI, 1994.

\bibitem[LS21]{LS:1d_smooth_FFTs}
M.~Ludewig and A.~Stoffel.
\newblock A framework for geometric field theories and their classification in
  dimension one.
\newblock {\em SIGMA Symmetry Integrability Geom. Methods Appl.}, 17:Paper No.
  072, 58, 2021.
\newblock \href {https://arxiv.org/abs/2001.05721} {\path{arXiv:2001.05721}}.

\bibitem[Lur09a]{Lurie:HTT}
J.~Lurie.
\newblock {\em Higher topos theory}, volume 170 of {\em Annals of Mathematics
  Studies}.
\newblock Princeton University Press, Princeton, NJ, 2009.

\bibitem[Lur09b]{Lurie:Classification_of_TQFTs}
J.~Lurie.
\newblock On the classification of topological field theories.
\newblock In {\em Current developments in mathematics, 2008}, pages 129--280.
  Int. Press, Somerville, MA, 2009.
\newblock \href {https://arxiv.org/abs/0905.0465} {\path{arXiv:0905.0465}}.

\bibitem[Min22]{Minichiello:Dfg_prBuns_and_Pr-oo-Buns}
E.~Minichiello.
\newblock Diffeological principal bundles and principal infinity bundles.
\newblock 2022.
\newblock \href {https://arxiv.org/abs/2202.11023 v3} {\path{arXiv:2202.11023
  v3}}.

\bibitem[Moe02]{Moerdijk:Intro_to_Stacks}
I.~Moerdijk.
\newblock {An introduction to stacks and gerbes}.
\newblock 2002.
\newblock \href {https://arxiv.org/abs/math/0212266}
  {\path{arXiv:math/0212266}}.

\bibitem[Mos20]{Moser:Double-oo-Cat_nerve}
L.~Moser.
\newblock A double $(\infty,1)$-categorical nerve for double categories.
\newblock 2020.
\newblock \href {https://arxiv.org/abs/math/0212266}
  {\path{arXiv:math/0212266}}.

\bibitem[NS11]{NS--Equivariance_in_higher_geometry}
T.~Nikolaus and C.~Schweigert.
\newblock Equivariance in {H}igher {G}eometry.
\newblock {\em Adv. Math.}, 226(4):3367--3408, 2011.
\newblock \href {https://arxiv.org/abs/1004.4558} {\path{arXiv:1004.4558}}.

\bibitem[Rez96]{Rezk:Canonical_MoStr_for_Cats}
C.~Rezk.
\newblock A model category for categories.
\newblock 1996.
\newblock Accessed Oct.~22.
\newblock URL:
  \url{https://ncatlab.org/nlab/files/Rezk_ModelCategoryForCategories.pdf}.

\bibitem[Rez01]{Rezk:HoTheory_of_HoTheories}
C.~Rezk.
\newblock A model for the homotopy theory of homotopy theory.
\newblock {\em Trans. Amer. Math. Soc.}, 353(3):973--1007, 2001.
\newblock \href {https://arxiv.org/abs/math/9811037}
  {\path{arXiv:math/9811037}}.

\bibitem[Rez10a]{Rezk:Cartesian_pres_of_oon-Cats}
C.~Rezk.
\newblock A {C}artesian presentation of weak {$n$}-categories.
\newblock {\em Geom. Topol.}, 14(1):521--571, 2010.
\newblock \href {https://arxiv.org/abs/0901.3602v3} {\path{arXiv:0901.3602v3}}.

\bibitem[Rez10b]{Rezk:Corrigendum_to_Cart_pres}
C.~Rezk.
\newblock Correction to ``{A} {C}artesian presentation of weak
  {$n$}-categories''.
\newblock {\em Geom. Topol.}, 14(4):2301--2304, 2010.

\bibitem[Rie14]{Riehl:Cat_HoThy}
E.~Riehl.
\newblock {\em Categorical homotopy theory}, volume~24 of {\em New Mathematical
  Monographs}.
\newblock Cambridge University Press, Cambridge, 2014.

\bibitem[Saf17]{Safronov:TQFT_Notes}
P.~Safronov.
\newblock {Topological Quantum Field Theories}.
\newblock 2017.
\newblock URL:
  \url{https://drive.google.com/file/d/0B3Hq3GkR_m3iT0pQcVpOZDFnQms/view}.

\bibitem[Sch13]{Schreiber:DCCT}
U.~Schreiber.
\newblock {Differential cohomology in a cohesive $\infty$-topos}.
\newblock 2013.
\newblock \href {https://arxiv.org/abs/1310.7930} {\path{arXiv:1310.7930}}.

\bibitem[Seg87]{Segal1987}
G.~Segal.
\newblock The definition of conformal field theory.
\newblock In {\em Schloss Ringberg, March 1987: Links between Geometry and
  Mathematical Physics}, volume~58 of {\em
  \href{https://www.mpim-bonn.mpg.de/preblob/4381}{The MPIM preprint series}},
  pages 13--17, 1987.

\bibitem[SP17]{SP:Invertible_TFTs}
C.{\,}J. Schommer-Pries.
\newblock {Invertible Topological Field Theories}.
\newblock 2017.
\newblock \href {https://arxiv.org/abs/1712.08029} {\path{arXiv:1712.08029}}.

\bibitem[ST11]{ST:SuSy_FTs_and_generalised_coho}
S.~Stolz and P.~Teichner.
\newblock Supersymmetric field theories and generalized cohomology.
\newblock volume~83 of {\em Proc. Sympos. Pure Math.}, pages 279--340. 2011.
\newblock \href {https://arxiv.org/abs/1108.0189} {\path{arXiv:1108.0189}}.

\bibitem[Wal07]{Waldorf--More_morphisms}
K.~Waldorf.
\newblock More morphisms between bundle gerbes.
\newblock {\em Theor. Appl. Cat.}, 18(9):240--273, 2007.
\newblock \href {https://arxiv.org/abs/math/0702652}
  {\path{arXiv:math/0702652}}.

\end{thebibliography}

\vspace{0.2cm}

\paragraph*{Availability of Data and Materials}

Data sharing not applicable to this article as no datasets were generated or analyzed during the current study.


\end{small}

\end{document}